\documentclass[13pt,reqno]{amsart}
\usepackage{amsmath,amssymb,amsthm}
\usepackage{bm}
\usepackage{graphicx}
\usepackage{color}
\usepackage{float}
\usepackage{url}
\usepackage{mathrsfs}
\usepackage{caption}
\usepackage{url}
\newcommand{\bel}[1]{\begin{equation*}\label{#1}}
	
	\newcommand{\be}{\begin{equation}}

		\newcommand{\ba}{\begin{eqnarray}}
			\newcommand{\ea}{\end{eqnarray}}

		\newcommand{\qe}{\end{equation}}
	\newcommand{\R}{{\mathbb R}}
	\newcommand{\N}{{\mathbb N}}
	\newcommand{\Z}{{\mathbb Z}}
	\newcommand{\C}{{\mathbb C}}
	\newcommand{\T}{{\mathbb T}}
	\newcommand{\supp}{{\mathrm{supp}}}

	\newcommand{\eg}{\begin{example}}
		\newcommand{\egd}{\end{example}}
	\newcommand{\tm}{\begin{thm}}
		\newcommand{\tmd}{\end{thm}}
	\newcommand{\co}{\begin{coro}}
		\newcommand{\cod}{\end{coro}}
	\newcommand{\enu}{\begin{enumerate}}
		\newcommand{\enud}{\end{enumerate}}
	\newcommand{\rmk}{\begin{rem}}
		\newcommand{\rmkd}{\end{rem}}
	
	\theoremstyle{plain}
	\newtheorem{thm}{Theorem}[section]
	\newtheorem{prop}[thm]{Proposition}
	\theoremstyle{plain}
	\newtheorem{example}[thm]{Example}
	\newtheorem{coro}[thm]{Corollary}
	\theoremstyle{plain}
	\newtheorem{lemma}[thm]{Lemma}
	\theoremstyle{definition}
	\newtheorem{defi}[thm]{Definition}
	\theoremstyle{plain}
	
	\theoremstyle{remark}
	\newtheorem{rem}[thm]{Remark}
	\theoremstyle{remark}

	\usepackage[T1]{fontenc}
	\usepackage{mathtools,booktabs,longtable}
	\usepackage[hidelinks]{hyperref}

	\allowdisplaybreaks
\begin{document}
		\pagestyle{plain}
		
		\title[Polynomial growth of Sobolev norms of solutions of the fractional
		NLS equation on $\T^{d}$]{Polynomial growth of Sobolev norms of solutions of the fractional NLS equation on $\T^{d}$}
		\author{
			\textbf{Author:} Jiajun Wang$^\ast$\\
			\\
			\textbf{Affiliation:} Courant Institute of Mathematical Sciences\\
			New York University, New York, NY 10012, USA\\
			\\
			\textbf{Email:} \lowercase{jw9409@nyu.edu}\\
			$^\ast$ \lowercase{Corresponding author}
		}
		\begin{abstract}
			We study polynomial growth of Sobolev norms for the defocusing fractional
			nonlinear Schr\"odinger equation on the standard torus $\T^d$, $d\ge2$.
			Uniform stationary phase and a frequency-dependent subdivision of time
			give Strichartz estimates which yield local well-posedness and, for the
			cubic equation, unconditional polynomial growth when $\alpha>d^2/(d+2)$.
		    The local theory and growth estimates extend to every integer-power nonlinearity
			$|u|^{2\sigma}u$, with explicit thresholds and growth exponents.
			
			In the final section, we identify a further approach to improving the polynomial growth bounds. By combining our arguments with the circle-method, we obtain stronger Strichartz estimates, which yield local well-posedness at lower regularity and improved polynomial growth bounds. 
		\end{abstract}
		\maketitle
		\numberwithin{equation}{section}
		
		\section{Introduction}
		This paper studies the growth of Sobolev norms for the defocusing equation
		\begin{equation}\label{FNLS}
			i\partial_tu=|D|^\alpha u+|u|^2u,\qquad u(0)=u_0,
			\qquad (t,x)\in\R\times\T^d,
		\end{equation}
		where $\T^d=(\R/2\pi\Z)^d$, $d\ge2$, and $|D|^\alpha=(-\Delta)^{\alpha/2}$.
		Our linear and local results concern $\alpha\in(0,1)\cup(1,\infty)$.
		The polynomial-growth results concern the ranges $\alpha>1$ specified below.
		Sections 2--5 treat only the cubic equation. Both local well-posedness and
		polynomial growth extend to the algebraic nonlinearities $|u|^{2\sigma}u$,
		$\sigma\in\N$, as stated and proved in Section~\ref{general nonlinearity growth}.
		
		Fractional Schr\"odinger equations arise in fractional quantum mechanics
		\cite{34}, in continuum limits of discrete systems with long-range
		interactions \cite{KLS,38}, and in dispersive models related to water waves
		\cite{36}. Local and global theories in Euclidean and one-dimensional
		periodic settings have been studied in \cite{26,27,28,29,30,31,32,37}.
		For dispersive and Strichartz estimates associated with 
		fractional Schr\"odinger flows, see Cho, Ozawa and Xia \cite{COX11}.
		
		At $\alpha=2$, growth of higher Sobolev norms has a substantial
		history; see \cite{18,16,19,17,20,PTV}.
		Catoire and Wang \cite{CW10} established Strichartz estimates
		on general flat tori and obtained polynomial Sobolev bounds
		for the defocusing cubic NLS in dimensions two and three.
		Sohinger proved polynomial growth bounds for higher-power
		NLS \cite{Soh11}, and for Hartree equations on
		$\T^2$ and $\R^2$ \cite{Soh12}.
		On three-dimensional irrational tori,
		Deng and Germain \cite{DG19} used long-time Strichartz estimates
		to improve polynomial growth bounds for nonlinearities
		$|u|^{p-1}u$ with $3<p<5$.
		Deng \cite{Den19} established polynomial bounds for higher
		Sobolev norms of solutions to the energy-critical quintic NLS
		on such tori under a small-energy assumption.
		The modified-energy approach in this paper is related to
		\cite{1,20,22,PTV}.
		
		For $\alpha=1$, the radial Hessian eigenvalue of $|\xi|^\alpha$
		vanishes. The one-dimensional half-wave flow is non-dispersive,
		whereas in dimensions $d\ge2$ it exhibits wave-type dispersion.
		The nonlinear half-wave equation has nevertheless been studied
		extensively. On the $\T$, G\'erard and Grellier \cite{GG12}
		related the dynamics of the cubic half-wave equation, for small
		initial data supported on nonnegative Fourier modes, to the cubic
		Szeg\H{o} equation, while Thirouin \cite{1} established upper bounds
		for the growth of higher Sobolev norms. On the real line,
		Ozawa and Visciglia \cite{20} proved global well-posedness for the
		defocusing quartic half-wave equation with initial data in $H^1$.
		We exclude $\alpha=1$ because the nondegenerate stationary-phase
		argument used here does not apply.
		
		For $\alpha>1$ and $1\le p<\infty$, define
		\begin{equation}\label{new parameters}
			\gamma_{\alpha,p}:=
			\max\left\{d-\frac\alpha p,\frac{d^2p}{dp+2}\right\},
			\qquad \gamma_{\alpha,\infty}=d.
		\end{equation}
		For $0<\alpha<1$, define $\ell_{\alpha,p}:=d-\alpha/p$, $\ell_{\alpha,\infty}:=d$.
		
		\begin{thm}\label{Strichartz main}
			Let $d\ge2$, $1\le p\le\infty$, and $|I|\le1$.
			If $\alpha>1$ and $\gamma>\gamma_{\alpha,p}/2$, then
			\begin{equation}\label{main linear estimate}
				\|e^{-it|D|^\alpha}f\|_{L^{2p}(I;L^\infty(\T^d))}
				\lesssim\|f\|_{H^\gamma(\T^d)}.
			\end{equation}
			For $0<\alpha<1$ the same assertion holds if $\gamma>\ell_{\alpha,p}/2$.
			The constants are independent of the position of $I$.
		\end{thm}
		
		\begin{rem}\label{linear comparison}
			The compact-manifold estimates of Burq--G\'erard--Tzvetkov \cite{BGT04}
			and Dinh \cite{D17}, followed by spatial Sobolev embedding, give the strict
			threshold $\gamma>d/2-1/(2p)$ for $\alpha>1$.  A direct unit-interval application of our kernel bound would instead give
			\begin{equation*}
				\gamma>\max\left\{\frac{d\alpha}{4},\frac{d}{2}-\frac{\alpha}{2p}\right\}.
			\end{equation*}
			This improves the compact-manifold threshold only for
			$1<\alpha<2-2/(dp)$; for $\alpha\ge2$ it does not improve the Sobolev
			algebra threshold. The time subdivision used below replaces the first
			branch by $d^2p/[2(dp+2)]$. Its difference from the compact-manifold
			threshold is
			\[
			\frac d2-\frac1{2p}-\frac{d^2p}{2(dp+2)}
			=\frac{dp-2}{2p(dp+2)}.
			\]
			Thus the improvement is strict for finite $p$, $dp>2$, and $\alpha>1$.
			At $d=2,p=1$ the thresholds agree. 
		\end{rem}
		
		As an application, we obtain the following local well-posedness result at regularities $\gamma$ below the Sobolev algebra threshold $d/2$.
		\begin{thm}\label{local well-posedness}
			Let $\alpha>1$, $1<p<\infty$, $\gamma>\gamma_{\alpha,p}/2$, and
			$1/2<b<1-1/(2p)$. For $u_0\in H^\gamma$, there is
			$T_0=T_0(\|u_0\|_{H^\gamma})>0$ and a unique solution $u$ of equation (\ref{FNLS}) such that
			\[
			u\in X_\alpha^{\gamma,b}([-T_0,T_0]\times\T^d)
			\subseteq C([-T_0,T_0];H^\gamma).
			\]
			Moreover, we have the following bound
			\begin{equation*}
				\|u\|_{X_\alpha^{\gamma,b}([-T_0,T_0])}
				\lesssim\|u_0\|_{H^\gamma}.
			\end{equation*}
		If $u_0\in H^s$ with $s\ge\gamma$, then the solution remains in $H^s$ throughout the local existence interval, whose length depends only on $\|u_0\|_{H^\gamma}$. For $0<\alpha<1$, the same conclusions hold with $\gamma_{\alpha,p}$ replaced by $\ell_{\alpha,p}$.
		\end{thm}
	    \begin{rem}
	    		Letting $p\downarrow1$ gives the sufficient condition
	    	\begin{equation*}
	    		\gamma>\max\left\{\frac{d-\alpha}{2},\frac{d^2}{2(d+2)}\right\}
	    		\quad(\alpha>1).
	    	\end{equation*}
	    	For $0<\alpha<1$ it gives $\gamma>(d-\alpha)/2$. 
	    \end{rem}
	    	\begin{rem}
	    		We compare Theorem~\ref{local well-posedness} with earlier
	    		local well-posedness results for the cubic equation.
	    		For $\alpha>1$ and $d\ge2$, Dinh
	    		\cite[Theorem~1.6]{D17} establishes local well-posedness
	    		on compact manifolds without boundary under the condition
	    		\begin{equation*}
	    			\gamma>\frac{d-1}{2},
	    		\end{equation*}
	    		extending the compact-manifold Strichartz theory of
	    		Burq--G\'erard--Tzvetkov \cite{BGT04} to fractional dispersion.
	    		On $\T^d$, our estimates give the sufficient condition
	    		\begin{equation*}
	    			\gamma>
	    			\max\left\{
	    			\frac{d-\alpha}{2},
	    			\frac{d^2}{2(d+2)}
	    			\right\}.
	    		\end{equation*}
	    		For $d\ge3$ and $\alpha>1$, this strictly lowers the
	    		preceding threshold, since
	    		\begin{equation*}
	    			\frac{d-1}{2}
	    			-
	    			\max\left\{
	    			\frac{d-\alpha}{2},
	    			\frac{d^2}{2(d+2)}
	    			\right\}
	    			=
	    			\min\left\{
	    			\frac{\alpha-1}{2},
	    			\frac{d-2}{2(d+2)}
	    			\right\}>0.
	    		\end{equation*}
	    		For $d=2$ and $\alpha>1$, both results require
	    		$\gamma>\frac{1}{2}$. For $0<\alpha<1$, our condition
	    		$\gamma>\frac{d-\alpha}{2}$ also agrees with
	    		\cite[Theorem~1.6]{D17}. 
	    		
	    		Our threshold is not optimal for every dispersion exponent.
	    		In particular, when $\alpha=2$, Bourgain \cite{13}
	    		established local well-posedness for the cubic NLS on
	    		the square torus $\T^2$ in $H^\gamma$ for every $\gamma>0$,
	    		whereas Theorem~\ref{local well-posedness} requires
	    		$\gamma>\frac{1}{2}$.
	    		The improvements obtained by using the circle method are stated separately
	    		in Section~\ref{improved section}.
	    \end{rem}

		Next, we also obtain some global well-posedness results and corresponding  polynomial growth.
		 Throughout, $a+$ denotes
		$a+\varepsilon$ for an arbitrary $\varepsilon>0$; constants may depend on
		$\varepsilon$. The following is proved in Section 5.
		
	\begin{thm}\label{d}
		Let $d\ge2$ and $u_0\in C^\infty(\T^d)$. If
		\begin{equation}\label{unconditional range}
			\alpha>\frac{d^2}{d+2},
		\end{equation}
		then \eqref{FNLS} has a unique global smooth solution.
		
		For $\alpha>d$ and every integer $n\ge0$,
		\begin{equation}\label{high main rate}
			\|u(t)\|_{H^{\alpha+n}}
			\lesssim_{u_0}
			\langle t\rangle^{
				\frac{\alpha+2n}{\min\{2,\alpha-d\}}+},
			\qquad t\in\R.
		\end{equation}
		
		For $d^2/(d+2)<\alpha\le d$, set
		\begin{equation*}
			\kappa_\alpha
			:=\min\left\{2,\alpha-\frac{d^2}{d+2}\right\}.
		\end{equation*}
		Then, for every integer $n\ge0$,
		\begin{equation}\label{low main rate}
			\|u(t)\|_{H^{\alpha+n}}
			\lesssim_{u_0}
			\langle t\rangle^{\frac{\alpha+2n}{\kappa_\alpha}+},
			\qquad t\in\R.
		\end{equation}
	\end{thm}
	
	\begin{rem}
		The argument yielding \eqref{low main rate} (after simple modification) applies throughout
		\eqref{unconditional range}, including $\alpha>d$.
		In the latter range,
		\begin{equation*}
			\kappa_\alpha
			\ge\min\{2,\alpha-d\},
		\end{equation*}
		so \eqref{low main rate} is always at least as strong as
		\eqref{high main rate}. Its growth exponent is strictly smaller
		when $d<\alpha<d+2$, while the two exponents agree when
		$\alpha\ge d+2$.
		
		We retain the two arguments to distinguish their inputs.
		For $\alpha>d$, energy conservation and Sobolev embedding
		suffice to prove global existence and \eqref{high main rate}
		without Strichartz estimates.
		The Strichartz estimates provide the spacetime control used
		to obtain \eqref{low main rate}, extending the range to
		\eqref{unconditional range} and improving the growth exponent
		also for $d<\alpha<d+2$.
	\end{rem}

		\begin{rem}\label{growth comparison}
			Planchon--Tzvetkov--Visciglia \cite{PTV} already prove polynomial Sobolev
			bounds for the standard NLS on two-dimensional compact manifolds and obtain
			related results in dimension three. We do not claim to improve upon their polynomial growth exponents when $\alpha=2$. The purpose or novelty here is an explicit fractional and higher-dimensional argument.
		\end{rem}
		
		Section~\ref{general nonlinearity growth} proves similar results for general nonlinearity $|u|^{2\sigma}u$, with unconditional threshold
		$\alpha>d^2\sigma/(d\sigma+2)$. Section~\ref{improved section} proves a fractional circle-method estimate on the
		torus $\T^d$ for $d\ge5$ (from representing fractional kernel by standard quadratic exponential sum), and formally states the resulting improved
		local well-posedness result and polynomial growth bounds. For example, its unconditional threshold (for cubic nonlinearity $\sigma=1$)
		in dimension five is $7/2$, compared with $25/7$ above. 
		
		\vspace{10pt}
		\noindent
		\textbf{Disclosure of AI assistance.}
		The author used ChatGPT-6 Astra during the modification of this
		paper to assist with developing partial arguments (especially in Section \ref{improved section} and Appendix C), checking calculations,
		and revising the exposition and \LaTeX{} code.
		The author takes full responsibility for the mathematical claims,
		proofs, and references presented in the paper.

		\vspace{10pt}
		\noindent
		\textbf{Organization.}
		
		Section~2 proves the Strichartz estimates.
		Section~3 recalls classical estimates in Bourgain spaces and establishes  
		multilinear estimates used in Section~4 to prove local well-posedness
		and spacetime bounds.
		Section~5 proves polynomial growth of Sobolev norms for the cubic
		equation, treating $\alpha>d$ and $d^2/(d+2)<\alpha\le d$ separately.
		Section~6 extends the local well-posedness and polynomial growth
		results to nonlinearities of the form $|u|^{2\sigma}u$.
		Section~7 derives stronger Strichartz estimates using the circle method, and establishes their consequences for
		local well-posedness and polynomial growth.
		Appendix~A collects uniform estimates for oscillatory integrals.
		Appendix~B presents fractional Leibniz rules and related product
		and commutator estimates. Appendix C provides detailed proofs of some product estimates appearing in Section \ref{general nonlinearity growth}.
		
		\vspace{10pt}
		\noindent
	\textbf{Notation.}
	\begin{itemize}
		\item By $u\in C([0,T];B)$ (or $L^p([0,T];B)$) for a
		Banach space $B$, we mean that $u$ is a continuous
		(or Bochner $L^p$) map from $[0,T]$ to $B$;
		see \cite[page 301]{evans2022partial}.
		
		\item By $A\lesssim B$ (resp. $A\sim B$), we mean that
		there exists a positive constant $C$ such that $A\le CB$
		(resp. $C^{-1}B\le A\le CB$).
		If the constant $C$ depends on $p$, then we write
		$A\lesssim_p B$ (resp. $A\sim_p B$).
		
		\item By $A\ll B$, we mean that $A/B$ is sufficiently small.
		
		\item We write $\langle z\rangle=(1+|z|^2)^{1/2}$
		for a real number or a Euclidean vector $z$.
		
		\item By $a+$, we mean $a+\varepsilon$ for an arbitrarily
		small $\varepsilon>0$. An estimate with exponent $a+$
		holds for every $\varepsilon>0$, with an implicit constant
		that may depend on $\varepsilon$.
		
		\item We work on $\mathbb T^d=(\mathbb R/2\pi\mathbb Z)^d$
		with normalized Haar measure. For $s>0$, we write
		$|D|^s=(-\Delta)^{s/2}$, so that
		\[
		\widehat{|D|^s f}(k)=|k|^s\widehat f(k),
		\qquad k\in\mathbb Z^d.
		\]
		We also write $\langle D\rangle^s=(1-\Delta)^{s/2}$
		and use the norm
		$\|f\|_{H^s}=\|\langle D\rangle^s f\|_{L^2}$.
		
	\end{itemize}
		
		\vspace{10pt}
		\section{Strichartz estimates for the fractional Schr\"odinger equation}
		 Recall the Littlewood-Paley decomposition: let $\psi\in C_{c}^{\infty}(\R^d)$ satisfy $0\le \psi\le1$, $\operatorname{supp} \psi\subseteq\lbrace \frac{1}{2}<|x|<2\rbrace$ with
		$\sum_{j=1}^{\infty}\psi(2^{-j}x)\equiv1$ for all $|x|\ge2$. The Littlewood-Paley operator, for $N=2^{j}$ and $u: \T^d\to \C$, is given by 
		\begin{equation*}
			\Delta_{N}u:=\psi(D/N)u=\sum_{k\in\Z^d}e^{ikx}\psi\left(\frac{k}{N}\right)\widehat{u}(k),
		\end{equation*}
		where $\widehat{u}$ denotes the Fourier transform on $\T^d$. Conventionally, $\Delta_{1}u:=u-\sum_{j\ge1}\Delta_{2^{j}}u$.
		
		We first prove the following dispersive estimate for every $u\in L^{1}(\T^{d})$, $N=2^{j}$, $j\ge1$:
		\begin{equation}\label{dispersive}
			\|e^{-it|D|^{\alpha}}\Delta_Nu\|_{L^{\infty}(\T^d)}\lesssim \omega_N(t)\|u\|_{L^{1}(\T^d)}, \quad \forall t\in(-1,1),
		\end{equation} 
		where the function $\omega_N(t)$ is defined in Lemma \ref{zcv} below.
		
		Then we can apply the $TT^{\ast}$-argument and Young's inequality for convolutions to establish Strichartz estimates.
		
		Next, we write down the convolution kernel of the operator $e^{-it|D|^{\alpha}}\Delta_N$ as follows:
		\begin{equation*}
			e^{-it|D|^{\alpha}}\Delta_Nu(x)=\sum_{k\in\Z^{d}}e^{i(kx-|k|^{\alpha}t)}\widehat{\Delta_Nu}(k)
		\end{equation*}
		\vspace{-3pt}
		\begin{equation*}
			=\sum_{k\in\Z^{d}}\int_{\T^{d}}e^{i(k(x-y)-|k|^{\alpha}t)}u(y)\psi\left(\frac{k}{N}\right)dy:=(u\ast_x\kappa_N)(x,t),
		\end{equation*}
		where the convolution kernel $\kappa_N$ is defined as 
		\begin{equation*}
			\kappa_N(x,t)=\sum_{k\in\Z^{d}}e^{i(kx-|k|^{\alpha}t)}\psi\left(\frac{k}{N}\right).
		\end{equation*}
		
		Then the dispersive estimate (\ref{dispersive}) can be totally reduced to the following crucial lemma.
		\begin{lemma}\label{zcv}
			We define $\omega_N(t)$ as follows,
			\begin{itemize}
				\item $\alpha>1$:
				\begin{equation*}
					\omega_N(t) =
					\begin{cases}
						N^d, & |t| \lesssim N^{-\alpha}, \\[8pt]
						N^{d-\frac{d\alpha}{2}} |t|^{-\frac{d}{2}}, & N^{-\alpha} \lesssim |t| \lesssim N^{1-\alpha}, \\[8pt]
						N^{\frac{d\alpha}{2}} |t|^{\frac{d}{2}}, & N^{1-\alpha} \lesssim |t| \leq 1.
					\end{cases}
				\end{equation*}
				
				\item $0<\alpha<1$: 
				\begin{equation*}
					\omega_N(t) =
					\begin{cases}
						N^d, & |t| \lesssim N^{-\alpha}, \\[8pt]
						N^{d-\frac{d\alpha}{2}} |t|^{-\frac{d}{2}}, & N^{-\alpha} \lesssim |t|\le 1.
					\end{cases}
				\end{equation*}
			\end{itemize}
			Then we have the following estimate
			\begin{equation}\label{omega}
				\|\kappa_N(\cdot,t)\|_{L^{\infty}(\T^{d})}\lesssim \omega_N(t), \quad \forall t\in (-1,1).
			\end{equation}
		\end{lemma}
			\begin{proof}
			
			To apply the oscillatory integral theory, we fix $t\in (-1,1)\setminus\{0\}$, $x\in (-\pi,\pi]^{d}$ and apply the Poisson summation formula to the following function
			\begin{equation*}
				F_{x,t}(y):=e^{i(yx-|y|^{\alpha}t)}\psi\left(\frac{y}{N}\right).
			\end{equation*}
			Then we can rewrite the convolution kernel $\kappa_N$ as 
			\begin{equation*}
				\kappa_N(x,t)=\sum_{k\in\Z^{d}}F_{x,t}(k)=\sum_{n\in\Z^{d}}\widehat{F_{x,t}}(2\pi n)=\sum_{n\in\Z^{d}}N^d\int_{\R^{d}}e^{i\phi_{n,N}(\xi)}\psi(\xi)\,d\xi:=\sum_{n\in\Z^d}I_{n,N},
			\end{equation*}
			where the phase function $\phi_{n,N}$ is given by 
			\begin{equation*}
				\phi_{n,N}(\xi):=N(x-2\pi n)\xi-N^{\alpha}t|\xi|^{\alpha}.
			\end{equation*}
			
			To prove \eqref{omega}, we first deal with the more subtle case $\alpha>1$, which is not covered in \cite{1}, and consider the following two cases.

			\medskip
			\noindent
			\textbf{Case 1:} $\bm{|t| \lesssim N^{1-\alpha}}.$ 
			
			In this case, we claim that 
			\begin{equation*}
				|\kappa_N(x,t)|\lesssim \min\left\lbrace  N^{d-\frac{d\alpha}{2}}\dfrac{1}{|t|^{d/2}},N^d \right\rbrace=:f(t), \quad \forall x\in(-\pi,\pi]^{d}, t\in (-1,1)\setminus \{0\}.
			\end{equation*}
			Note that $\operatorname{supp} \psi\subseteq \lbrace \frac{1}{2}<|\xi|<2\rbrace$, so $\xi$ is strictly away from zero. 
			
			We treat separately the cases $n=0$ and $n\neq 0$. When $n=0$, we split the argument into the following subcases:
			\begin{itemize}
				\item $|x|\sim N^{\alpha-1}|t|$:
				
				We rewrite the phase function $\phi_{0,N}(\xi)$ as follows:
				\begin{equation*}
					\phi_{0,N}(\xi)= N^{\alpha}t \left(N^{1-\alpha}\cdot\frac{x}{t}\xi-|\xi|^{\alpha}\right):=\lambda\left(v\xi-|\xi|^{\alpha}\right),
				\end{equation*}
				where we denote $\lambda:=N^{\alpha}t$, $v:=N^{1-\alpha}\cdot\frac{x}{t}$. Then, in this case, we have $|v|\sim 1$. Now, for any such $v$, the critical point $\xi_{v}$ satisfies $v=\alpha |\xi_{v}|^{\alpha-2}\xi_{v}$ and is unique. Invoking the uniform decay estimate in Corollary \ref{Q}, there exist $\delta_{v}, \varepsilon_{v}>0$ such that we have 
				\begin{equation*}
					\bigg|\int_{B(\xi_{v}, \varepsilon_{v})}e^{i\lambda (u\xi-|\xi|^{\alpha})}\psi(\xi)\,d\xi\bigg|\le C_{v}\langle\lambda\rangle^{-d/2}\le C_{v}N^{-\frac{d\alpha}{2}}\frac{1}{|t|^{d/2}}, 
				\end{equation*}
				for any $u\in B(v,\delta_{v})$, with some positive constant $C_v$ depending only on $v$. 
				
				By the fundamental theorem of calculus, we have the following observation, which also explains why our argument does not extend to the half-wave case $\alpha = 1$.
				
				\medskip
				
				\noindent\textbf{Observation.}
				If $\frac{1}{2} < |\xi_1| \le |\xi_2| < 2$ and $|\xi_1 - \xi_2| \ge c > 0$, then
				\begin{equation*}
					\Big| |\xi_1|^{\alpha-2}\xi_1 - |\xi_2|^{\alpha-2}\xi_2 \Big|
					\ge \min\lbrace1,(\alpha - 1)\rbrace\, 2^{-|\alpha - 2|} \, c, \quad \forall \alpha\ne1.
				\end{equation*}
				
				As a consequence, we may take $\delta_v \ll \varepsilon_v$ so as to ensure that there are no critical points outside $B(\xi_v, \varepsilon_v)$.
				
				Integrating by parts a sufficient number of times, for example $[d/2]+1$ times, we can also obtain
				\begin{equation*}
					\bigg|\int_{B(\xi_{v}, \varepsilon_{v})^{c}}e^{i\lambda (u\xi-|\xi|^{\alpha})}\psi(\xi)\,d\xi\bigg|\le C_{v}\langle\lambda\rangle^{-d/2}\le C_{v}N^{-\frac{d\alpha}{2}}\frac{1}{|t|^{d/2}},
				\end{equation*}
				uniformly in $u\in B(v,\delta_{v})$.  Combining the two estimates, we can derive the uniform bound:
				\begin{equation*}
					\bigg|\int_{\R^d}e^{i\lambda (u\xi-|\xi|^{\alpha})}\psi(\xi)\,d\xi\bigg|\le C_{v}N^{-\frac{d\alpha}{2}}\frac{1}{|t|^{d/2}}, \quad \forall u\in B(v,\delta_{v}).
				\end{equation*} 
				Now $\bigcup_{v}B(v,\delta_{v})$ forms an open covering of $\lbrace|v|\sim 1\rbrace$, so we can take a finite subcovering, which directly yields the desired bound $|I_{0,N}|\lesssim N^{d-\frac{d\alpha}{2}}\frac{1}{|t|^{d/2}}$.
				\vspace{7pt}
				\item $|x|\lesssim N^{\alpha-1}|t|$ or $|x|\gtrsim N^{\alpha-1}|t|$:
				
				A simple calculation shows that 
				\begin{equation*}
					|\nabla_{\xi}\phi_{0,N}(\xi)|=|Nx-\alpha N^{\alpha}t|\xi|^{\alpha-2}\xi|\gtrsim\max\lbrace N|x|, N^{\alpha}|t|\rbrace\ge N^{\alpha}|t|.
				\end{equation*}
				Then, applying integration by parts a sufficient number of times, we can also derive the bound $|I_{0,N}|\lesssim N^{d-\frac{d\alpha}{2}}\frac{1}{|t|^{d/2}}$.
			\end{itemize}
			Since the bound $|I_{0,N}|\lesssim N^d$ is always trivial, we can conclude that $|I_{0,N}|\lesssim f(t)$. 
			
			Now we turn to the case $n\neq 0$, which is much easier to deal with. Observe that 
			\begin{equation*}
				|\nabla_{\xi}\phi_{n,N}(\xi)|=|N(x-2\pi n)-\alpha N^{\alpha}t|\xi|^{\alpha-2}\xi|\gtrsim N|x-2\pi n|\gtrsim N|n|.
			\end{equation*}
			Then, using integration by parts a sufficient number of times, the following holds:
			\begin{equation*}
				\sum_{n\neq 0}|I_{n,N}|\lesssim \sum_{n\neq 0}N^d(N|n|)^{-d\alpha}\lesssim N^{d-d\alpha}\le f(t).
			\end{equation*}
			In conclusion, when $|t|\lesssim N^{1-\alpha}$, we have established the desired control:
			\begin{equation*}
				|\kappa_N(x,t)|\le |I_{0,N}|+\sum_{n\neq 0}|I_{n,N}|\lesssim f(t), \quad \forall x\in(-\pi,\pi]^{d}, t\in (-1,1)\setminus\{0\}.
			\end{equation*}
			
			\medskip
			\noindent
			\textbf{Case 2:} $\bm{N^{1-\alpha} \lesssim |t|<1}.$
			
			In this case, we claim that 
			\begin{equation*}
				|\kappa_N(x,t)|\lesssim \max\lbrace f(t), N^{d(\alpha-1)}|t|^{d}f(t) \rbrace = N^{d(\alpha-1)}|t|^{d}f(t)=: g(t).
			\end{equation*}
			
			When $n=0$, we still claim that $|I_{0,N}|\lesssim f(t)$. The estimate $|I_{0,N}|\lesssim N^d$ is still trivial, and we observe that
			\begin{equation*}
				|\nabla_{\xi}\phi_{0,N}(\xi)|=|Nx-\alpha N^{\alpha}t|\xi|^{\alpha-2}\xi|\gtrsim N^{\alpha}|t|,
			\end{equation*}
			which implies $|I_{0,N}|\lesssim N^{d-\frac{d\alpha}{2}}\frac{1}{|t|^{d/2}}$ by applying integration by parts a sufficient number of times.
			
			When $n\neq 0$, we consider the following subcases:
			\begin{itemize}
				\item $|n|\sim |x-2\pi n|\lesssim N^{\alpha-1}|t|$:
				
				We can rewrite the phase function $\phi_{n,N}(\xi)$ as before,
				\begin{equation*}
					\phi_{n,N}(\xi)=N^{\alpha}t\left(N^{1-\alpha}\cdot\frac{x-2\pi n}{t}\xi-|\xi|^{\alpha}\right):=\lambda (v\xi-|\xi|^{\alpha}),
				\end{equation*}
				where we denote $\lambda:=N^{\alpha}t$, $v:=N^{1-\alpha}\cdot\frac{x-2\pi n}{t}$; then $|v|\lesssim 1$. As in \textbf{Case 1}, we can use the uniform decay estimate and obtain the familiar bound $|I_{n,N}|\lesssim f(t)$, which leads to 
				\begin{equation*}
					\sum_{|n|\lesssim N^{\alpha-1}|t|} |I_{n,N}|\lesssim N^{d(\alpha-1)}|t|^{d} f(t).
				\end{equation*}
				\item $|n|\sim |x-2\pi n|\gtrsim N^{\alpha-1}|t|$:
				
				Note that, in this case, the following lower bound still holds:
				\begin{equation*}
					|\nabla_{\xi}\phi_{n,N}(\xi)|\gtrsim N|n|,
				\end{equation*}
				ensuring that 
				\begin{equation*}
					\sum_{|n|\gtrsim N^{\alpha-1}|t|}|I_{n,N}|\lesssim \sum_{n\neq 0} N^d(N|n|)^{-d\alpha}\lesssim f(t).
				\end{equation*}
			\end{itemize}
			In conclusion, we have shown, for all $x\in(-\pi,\pi]^{d}$ and $t\in (-1,1)\setminus\{0\}$,
			\begin{equation*}
				|\kappa_N(x,t)|\le |I_{0,N}|+\sum_{n\neq 0}|I_{n,N}|\lesssim \max \lbrace f(t), N^{d(\alpha-1)}|t|^{d}f(t)\rbrace = g(t).
			\end{equation*}
			Then combining the result in \textbf{Case 1}, we can take 
			\begin{equation*}
				\omega_{N}(t)=f(t)\mathbf{1}_{|t|\lesssim N^{1-\alpha}}+g(t)\mathbf{1}_{N^{1-\alpha}\lesssim|t|\le 1},
			\end{equation*}
			which yields the desired bound.
			
			For the simpler case $0<\alpha<1$, the argument in \textbf{Case 1} still holds, yielding $|\kappa_{N}(x,t)|\lesssim f(t)$. Then we can determine $\omega_{N}(t)$ as $f(t)$.
			
		\end{proof}

		We first record the direct unit-interval consequence of the preceding kernel estimate. Define the unit-interval exponent
		\begin{equation*}
			\frac{\Gamma_{\alpha,p}}2:=
			\max\left\{\frac{d\alpha}{4},\frac d2-\frac{\alpha}{2p}\right\},\qquad \alpha>1,
		\end{equation*}
		and recall $\ell_{\alpha,p}=d-\alpha/p$ for $0<\alpha<1$. For finite $p$, direct integration of $\omega_N$ on a unit interval gives the following bounds:
		\begin{equation*}
			\|\omega_N\|_{L^p([0,1])}\lesssim_\varepsilon N^{\Gamma_{\alpha,p}+\varepsilon},\qquad \alpha>1,
		\end{equation*}
		and
		\begin{equation*}
			\|\omega_N\|_{L^p([0,1])}\lesssim_\varepsilon N^{\ell_{\alpha,p}+\varepsilon},\qquad 0<\alpha<1.
		\end{equation*}
		Let $T_Nf=e^{-it|D|^\alpha}\Delta_Nf$. Then
		\begin{equation*}
			T_NT_N^*g(t,x)=\int_0^1\Delta_Ne^{-i(t-s)|D|^\alpha}\Delta_Ng(s,x)\,ds.
		\end{equation*}
		The dispersive estimate and Young's inequality, with $q_1=2p$, $q_2=p$ and $q_3=(2p)'$, imply
		\begin{equation*}
			\|T_NT_N^*g\|_{L^{2p}([0,1];L^\infty)}\lesssim N^{\Gamma_{\alpha,p}+\varepsilon}\|g\|_{L^{(2p)'}([0,1];L^1)},\qquad \alpha>1,
		\end{equation*}
		and hence the unit-interval estimate
		\begin{equation}\label{unit interval Stri1}
			\|e^{-it|D|^\alpha}\Delta_Nf\|_{L^{2p}([0,1];L^\infty)}\lesssim N^{\Gamma_{\alpha,p}/2+\varepsilon}\|f\|_2.
		\end{equation}
		For $0<\alpha<1$ the same argument gives
		\begin{equation}\label{Stri2}
			\|e^{-it|D|^\alpha}\Delta_Nf\|_{L^{2p}([0,1];L^\infty)}\lesssim N^{\ell_{\alpha,p}/2+\varepsilon}\|f\|_2.
		\end{equation}
		Thus the direct unit-interval argument produces the branch $d\alpha/4$ in (\ref{unit interval Stri1}). To exploit the time decay more efficiently, we carry out the same $TT^*$ argument on a frequency-dependent interval. 
		Let $\alpha>1$, and let $I$ be an interval of length $T\le1$.
		For $1\le p<\infty$ and $dp>2$, the kernel bound proved above gives
		\begin{equation*}
			\|\omega_N\|_{L^p([-T,T])}
			\lesssim
			N^{d-\frac{\alpha}{p}}
			+N^{\frac{d\alpha}{2}}T^{\frac d2+\frac1p}.
		\end{equation*}
		Applying the same $TT^*$ argument on $I$, we obtain
		\begin{equation}\label{local short estimate}
			\|e^{-it|D|^\alpha}\Delta_Nf\|_{L^{2p}(I;L^\infty)}
			\lesssim
			\left(
			N^{\frac d2-\frac{\alpha}{2p}}
			+N^{\frac{d\alpha}{4}}T^{\frac d4+\frac{1}{2p}}
			\right)\|f\|_2.
		\end{equation}
		Partition $[0,1]$ into at most $\frac2T$ intervals of length
		at most $T$. Summing the $2p$-th powers of the local estimates,
		and using unitarity and time translation, yields
		\begin{equation}\label{partition estimate}
			\|e^{-it|D|^\alpha}\Delta_Nf\|_{L^{2p}([0,1];L^\infty)}
			\lesssim
			\left(
			N^{\frac d2-\frac{\alpha}{2p}}T^{-\frac{1}{2p}}
			+N^{\frac{d\alpha}{4}}T^{\frac d4}
			\right)\|f\|_2.
		\end{equation}
		Choosing
		\begin{equation*}
			T=\min\left\{1,N^{\frac{2dp}{dp+2}-\alpha}\right\},
		\end{equation*}
		 we obtain
		\begin{equation}\label{Stri1}
			\|e^{-it|D|^\alpha}\Delta_Nf\|_{L^{2p}([0,1];L^\infty)}
			\lesssim_\varepsilon
			N^{\frac{\gamma_{\alpha,p}}2+\varepsilon}\|f\|_2,
		\end{equation}
		where
		\begin{equation*}
			\frac{\gamma_{\alpha,p}}2
			=
			\max\left\{
			\frac d2-\frac{\alpha}{2p},
			\frac{d^2p}{2(dp+2)}
			\right\}
		\end{equation*}
		as in \eqref{new parameters}. 
		
		The only finite endpoint not covered directly is $d=2$, $p=1$.
		For each $\varepsilon>0$, choose $p'>1$ sufficiently close to $1$
		and apply the preceding estimate with $p'$ in place of $p$.
		H\"older's inequality in time and the continuity of
		$\gamma_{\alpha,p}$ in $p$ then give \eqref{Stri1} at $p=1$.
		The case $p=\infty$ follows from Bernstein's inequality
		and unitarity.
		
		For $0<\alpha<1$, the long-time growth term in $\omega_N$
		is absent. Taking $T=1$ therefore gives \eqref{Stri2},
		with an arbitrarily small $N^\varepsilon$ loss allowed
		at $d=2$, $p=1$.

		\begin{proof}[Proof of Theorem \ref{Strichartz main}]
			Choose $0<\varepsilon<\gamma-\gamma_{\alpha,p}/2$. By \eqref{Stri1},
			\begin{equation*}
				\|e^{-it|D|^\alpha}f\|_{L^{2p}L^\infty}
				\le\sum_N\|e^{-it|D|^\alpha}\Delta_Nf\|_{L^{2p}L^\infty}
				\lesssim\sum_NN^{\gamma_{\alpha,p}/2+\varepsilon}\|\Delta_Nf\|_2.
			\end{equation*}
			Cauchy--Schwarz in the dyadic sum yields
			\begin{equation*}
				\sum_NN^{\gamma_{\alpha,p}/2+\varepsilon}\|\Delta_Nf\|_2
				\lesssim
				\left(\sum_NN^{2\gamma}\|\Delta_Nf\|_2^2\right)^{1/2}
				\lesssim\|f\|_{H^\gamma}.
			\end{equation*}
			The low frequencies are harmless. For $0<\alpha<1$, use \eqref{Stri2} in the same way.
			
		\end{proof}
		
	\begin{rem}
		We present an example, pointed out by Prof.~Alex Cohen,
		which shows that the regularity threshold in our estimate
		is sharp up to the endpoint whenever
		$\gamma_{\alpha,p}$ or $\ell_{\alpha,p}$ equals
		$d-\frac{\alpha}{p}$.
		
		For $N\ge2$, let
		\begin{equation*}
			f_N:=
			\sum_{\substack{k\in\mathbb Z^d\\N\le |k|\le2N}}
			e^{ik\cdot x},\qquad e^{-it|D|^\alpha}f_N
			=
			\sum_{\substack{k\in\mathbb Z^d\\N\le|k|\le2N}}
			e^{i(k\cdot x-t|k|^\alpha)}.
		\end{equation*}

		The number of frequencies in this sum is comparable to $N^d$.
		Consequently, 
		\[
		\|f_N\|_{H^\gamma(\mathbb T^d)}
		\sim N^{\gamma+\frac d2}.
		\]
		Choose $c_\alpha>0$ sufficiently small. 
		If $0\le t\le c_\alpha N^{-\alpha}$, then
		\[
		0\le t|k|^\alpha\le\frac{\pi}{3}
		\qquad\text{for }N\le|k|\le2N.
		\]
		Evaluating the solution at $x=0$, we obtain
		\[
		|e^{-it|D|^\alpha}f_N(0)|
		\gtrsim N^d.
		\]
		Thus, for $1\le p<\infty$,
		\[
		\|e^{-it|D|^\alpha}f_N\|_
		{L^{2p}([0,1];L^\infty(\mathbb T^d))}
		\gtrsim
		N^d\bigl(N^{-\alpha}\bigr)^{\frac1{2p}}
		=
		N^{d-\frac{\alpha}{2p}}.
		\]
		The analogous lower bound for $p=\infty$ is $N^d$.
		Therefore, an estimate of the form
		\[
		\|e^{-it|D|^\alpha}f\|_
		{L^{2p}([0,1];L^\infty(\mathbb T^d))}
		\lesssim
		\|f\|_{H^\gamma(\mathbb T^d)}
		\]
		necessarily requires
		\[
		\gamma\ge\frac d2-\frac{\alpha}{2p}.
		\]

		Testing the same estimate on a single Fourier mode
		$f(x)=e^{ik\cdot x}$, with $|k|\to\infty$, also gives
		$\gamma\ge0$. Hence a necessary condition is
		\[
		\gamma\ge
		\max\left\{
		\frac d2-\frac{\alpha}{2p},0
		\right\}.
		\]
		For $0<\alpha<1$, the first entry is positive and equals
		$\frac{\ell_{\alpha,p}}2$.
		For $\alpha>1$, this example matches the threshold
		$\frac{\gamma_{\alpha,p}}2$ whenever
		$\gamma_{\alpha,p}=d-\frac{\alpha}{p}$.
		It does not establish sharpness of the other branch
		in the definition of $\gamma_{\alpha,p}$.
		
		At $p=\infty$, the necessary condition is $\gamma\ge\frac d2$.
		The endpoint $\gamma=\frac d2$ fails for the full Sobolev
		estimate. Indeed, for
		\[
		g_M(x):=
		\sum_{\substack{k\in\mathbb Z^d\\1\le|k|\le M}}
		|k|^{-d}e^{ik\cdot x},
		\]
		we have
		\[
		g_M(0)\sim\log M,
		\qquad
		\|g_M\|_{H^{\frac d2}}\sim(\log M)^{\frac12}.
		\]
		Since the evolution of each trigonometric polynomial
		is continuous in time, the spacetime $L^\infty$ norm
		is at least its initial spatial $L^\infty$ norm.
		Thus $\gamma>\frac d2$ is required at $p=\infty$,
		consistently with Sobolev embedding.
	\end{rem}

		\vspace{10pt}
		\section{Bourgain spaces and multi-linear estimates}
		Unless an interval is indicated, the Bourgain norms in this section and Section 4 are taken on $\R\times\T^d$.
		
		To establish well-posedness of the fractional nonlinear Schr\"{o}dinger equation on $\T^d$, we recall the Bourgain space on the torus, which was introduced by Bourgain \cite{13}.
		\begin{defi}
			The space $X_{\alpha}^{s,b}$ is the closure of Schwartz functions $\mathcal{S}_{t,x}(\R\times \T^d)$ under the following norm:
			\begin{equation*}
				\|u\|_{X_{\alpha}^{s,b}}:=\left(\frac{1}{(2\pi)^{d}}\sum_{k\in\Z^d}\int_{\R}(1+|k|^{2})^{s}(1+|\tau+|k|^{\alpha}|^{2})^{b}|\mathcal{F}u(\tau,k)|^{2}\,d\tau\right)^{1/2},
			\end{equation*}
			where $\mathcal{F}$ denotes the Fourier transform in both time and space.
		\end{defi}
		\begin{rem}
			If we denote $H^{s,b}:=H^{b}(\R;H^{s}(\T^d))$, then we have another representation of the Bourgain space $X_{\alpha}^{s,b}$: 
			\begin{equation*}
				\|u\|_{X_{\alpha}^{s,b}}=\|e^{it|D|^{\alpha}}u\|_{H^{s,b}}.
			\end{equation*}
		\end{rem}
		The Bourgain space $X_{\alpha}^{s,b}$ defined above is global in time, which requires the following modification, especially when establishing local well-posedness theory.
		\begin{defi}
			For $T>0$, the local Bourgain space $X_{\alpha}^{s,b}([-T,T]\times\T^d)$ consists of all functions $u$ such that there exists another function $\widetilde{u}\in X_{\alpha}^{s,b}$ whose restriction to $[-T,T]$ is $u$, i.e., $\widetilde{u}|_{[-T,T]}\equiv u$. The norm is given by 
			\begin{equation*}
				\|u\|_{X_{\alpha}^{s,b}([-T,T]\times\T^d)}:=\inf\left\lbrace  \|\widetilde{u}\|_{X_{\alpha}^{s,b}}\; \Big| \; \widetilde{u}\in X_{\alpha}^{s,b},\; \widetilde{u}|_{[-T,T]}\equiv u\right\rbrace.
			\end{equation*}
		\end{defi}
		
		In order to obtain (global) Strichartz estimates from local ones, we need a localization lemma for the Triebel-Lizorkin space $F_{p,q}^{r}$; see \cite{14}.
		
		Before presenting this lemma, we first recall the following definition of Triebel-Lizorkin space $F_{p,q}^{r}$. 
		\begin{defi}
			The Triebel--Lizorkin space \(F^r_{p,q}(\mathbb{R}^n)\) consists of all tempered distributions \(f \in \mathcal{S}'(\mathbb{R}^n)\) for which
			\[
			\| f \|_{F^r_{p,q}(\R^n)} := 
			\Bigl\| \Bigl( \sum_{j=0}^\infty \bigl( 2^{jr} |\Delta_{2^j} f| \bigr)^q \Bigr)^{1/q} \Bigr\|_{L^p(\mathbb{R}^n)} < \infty.
			\]
			
			If \(q=\infty\), the definition reads
			\[
			\| f \|_{F^r_{p,\infty}(\R^n)} := 
			\Bigl\| \sup_{j \ge 0} \, 2^{jr} |\Delta_{2^j} f| \Bigr\|_{L^p(\mathbb{R}^n)}.
			\]
		\end{defi}
		
		\begin{lemma}\label{Localization}
			Let \( 0 < p < \infty \), \( 0 < q \leq \infty \), 
			and let \( r \in \mathbb{R} \).  
			Let \( \psi \) be a compactly supported \( C^{\infty} \) function on \( \mathbb{R}^n \) such that
			\[
			\sum_{k \in \mathbb{Z}^n} \psi(x - k/2) = 1, \qquad \forall x \in \mathbb{R}^n.
			\]
			Then we have the equivalence
			\[
			\Bigl( \sum_{k \in \mathbb{Z}^n} \|\psi(\cdot - k/2)f\|_{F_{p,q}^r(\R^n)}^p \Bigr)^{1/p}\sim \|f\|_{F_{p,q}^{r}(\R^n)}.
			\]
			In particular, since $F_{2,2}^{r}(\R^n)=H^{r}(\R^n)$, we have 
			\[
			\Bigl( \sum_{k \in \mathbb{Z}^n} \|\psi(\cdot - k/2)f\|_{H^r(\R^n)}^2 \Bigr)^{1/2}\sim \|f\|_{H^{r}(\R^n)}.
			\]
		\end{lemma}
		
		With this localization lemma, we can obtain the following estimates involving Bourgain spaces.
		\begin{lemma}\label{Stri3}
			For $b>\frac{1}{2}$, $\alpha>1$, and $\gamma>\frac{\gamma_{\alpha,p}}{2}$, we have 
			\begin{equation*}
				\|u\|_{L^{2p}(\R; L^{\infty}(\T^{d}))}\lesssim\|u\|_{X_{\alpha}^{\gamma,b}}, \quad 1\le p\le \infty.
			\end{equation*}
			Similarly, for $b>\frac{1}{2}$, $0<\alpha<1$, and $\gamma>\frac{\ell_{\alpha,p}}{2}$, the above estimate also holds.
		\end{lemma}
		\begin{proof}
			As usual, we only deal with the case $\alpha>1$. We first assume that $u$ is supported on some interval $I$ of length $1$. Writing $v:=e^{it|D|^{\alpha}}u$, we have 
			\begin{equation*}
				\begin{aligned}&\|u\|_{L^{2p}(I; L^{\infty}(\T^{d}))}=\|e^{-it|D|^{\alpha}}v\|_{L^{2p}(I; L^{\infty}(\T^{d}))}\\ &=\frac{1}{2\pi}\left\|e^{-it|D|^{\alpha}}\int_{\R}\widehat{v}(s)e^{its}\,ds\right\|_{L^{2p}(I; L^{\infty}(\T^{d}))}
				\end{aligned}
			\end{equation*}
			\begin{equation*}
				\lesssim\int_{\R}\|e^{-it|D|^{\alpha}}\widehat{v}(s)\|_{L^{2p}(I; L^{\infty}(\T^{d}))}\,ds,
			\end{equation*}
			where $\widehat{v}$ denotes the Fourier transform in time. Then, applying Theorem \ref{Strichartz main} and the Cauchy-Schwarz inequality, we obtain 
			\begin{equation*}
				\begin{aligned}&\|u\|_{L^{2p}(I; L^{\infty}(\T^{d}))}\lesssim \int_{\R}\|\widehat{v}(s)\|_{H^{\gamma}(\T^d)}\,ds\\ &\lesssim\left(\int_{\R}\|\widehat{v}(s)\|_{H^{\gamma}(\T^d)}^{2}(1+|s|^2)^{b}\,ds\right)^{\frac{1}{2}}\left(\int_{\R}\frac{1}{(1+|s|^{2})^{b}}\,ds\right)^{\frac{1}{2}}
				\end{aligned}
			\end{equation*}
			\begin{equation*}
				\lesssim \|v\|_{H^{\gamma,b}(\R\times \T^d)}=\|u\|_{X_{\alpha}^{\gamma,b}}.
			\end{equation*}
			For the general case, we can find $\psi\in C_{c}^{\infty}([0,1])$ such that
			\begin{equation*}
				\sum_{k\in\Z}\psi(t-k/2)\equiv 1, \quad 0\le \psi\le 1.
			\end{equation*}
			The case $p=\infty$ follows by taking the supremum of the local bounds. For finite $p$, since the supports of $\lbrace\psi(\cdot-k/2)\rbrace_{k\in \Z}$ are finitely overlapping, we derive that
			\begin{equation*}
				\|u\|_{L^{2p}(\R; L^{\infty}(\T^{d}))}^{2p}=\int_{\R}\left\|\sum_{k\in\Z}u(t,\cdot)\psi(t-k/2)\right\|_{L^{\infty}(\T^d)}^{2p}dt
			\end{equation*}
			\begin{equation*}
				\begin{aligned}&\le \int_{\R}\left(\sum_{k\in\Z}\psi(t-k/2)\|u(t,\cdot)\|_{L^{\infty}(\T^d)}\right)^{2p}\,dt\\ &\lesssim \sum_{k\in\Z}\|\psi(t-k/2)u(t,\cdot)\|_{L^{2p}(\R; L^{\infty}(\T^{d}))}^{2p}.
				\end{aligned}
			\end{equation*}
			Then, applying Lemma \ref{Localization} and the estimate for the special case, we conclude that
			\begin{equation*}
				\|u\|_{L^{2p}(\R; L^{\infty}(\T^{d}))}\lesssim \left(\sum_{k\in\Z}\|\psi(t-k/2)v(t,\cdot)\|_{H^{\gamma,b}}^{2p}\right)^{\frac{1}{2p}}
			\end{equation*}
			\begin{equation*}
				\le \left(\sum_{k\in\Z}\|\psi(t-k/2)v(t,\cdot)\|_{H^{\gamma,b}}^{2}\right)^{\frac{1}{2}}\sim \|v\|_{H^{\gamma,b}}=\|u\|_{X_{\alpha}^{\gamma,b}}.
			\end{equation*}
			
		\end{proof}
		
		\begin{rem}\label{local version}
			
			Note that for any time interval $J \subseteq \mathbb{R}$ and any extension $\widetilde{u} \in X_{\alpha}^{s,b}(\mathbb{R} \times \mathbb{T}^d)$ satisfying $\widetilde{u}|_{J} \equiv u$, we have
			\begin{equation*}
				\|u\|_{L^{2p}(J; L^{\infty}(\mathbb{T}^d))}
				\le \|\widetilde{u}\|_{L^{2p}(\mathbb{R}; L^{\infty}(\mathbb{T}^d))}
				\lesssim \|\widetilde{u}\|_{X_{\alpha}^{\gamma,b}(\mathbb{R} \times \mathbb{T}^d)}.
			\end{equation*}
			Taking the infimum over all such extensions $\widetilde{u}$, we obtain the following local version.
			
			\begin{itemize}
				\item For $b>\frac{1}{2}$, $\alpha>1$, and $\gamma>\frac{\gamma_{\alpha,p}}{2}$, we have 
				\begin{equation}\label{local strichartz}
					\|u\|_{L^{2p}(J; L^{\infty}(\T^{d}))}\lesssim\|u\|_{X_{\alpha}^{\gamma,b}(J\times\T^d)}, \quad 1\le p\le \infty.
				\end{equation}
				Similarly, for $b>\frac{1}{2}$, $0<\alpha<1$, and $\gamma>\frac{\ell_{\alpha,p}}{2}$, the above estimate also holds.
			\end{itemize}
		\end{rem}
		
		Besides Lemma \ref{Stri3}, we also need the following inequality obtained by interpolation.
		\begin{lemma}\label{Stri4}
			For $b>\frac{1}{2p}$, we have 
			\begin{equation*}
				\|u\|_{L^{\frac{2p}{p-1}}(\R;L^{2}(\T^d))}\lesssim\|u\|_{X_{\alpha}^{0,b}}.
			\end{equation*}
		\end{lemma}
		\begin{proof}
			This follows from the simple fact that $\|u\|_{L^{2}(\R;L^{2}(\T^d))}=\|u\|_{X_{\alpha}^{0,0}}$ and 
			\begin{equation*}
				\|u\|_{L^{\infty}(\R;L^{2}(\T^d))}\lesssim \|u\|_{X_{\alpha}^{0,b}},\quad  \forall b>\frac{1}{2},
			\end{equation*}
			which can be derived from the inverse Fourier transform and the Cauchy-Schwarz inequality.
			
		\end{proof}
		
		To establish the local well-posedness theory for the fractional nonlinear Schr\"odinger equation (\ref{FNLS}), we also need a multilinear estimate to control the nonlinear term, where Lemma \ref{Stri3} and Lemma \ref{Stri4} will be used. The proof below follows the ideas in \cite{15,1}.
		
		\begin{prop}\label{multi}
			For $\alpha>1$, $1<p<\infty$, $s\ge\gamma>\frac{\gamma_{\alpha,p}}{2}$, $b>\frac{1}{2}$, $b'>\frac{1}{2p}$, and $b>b'$, we have
			\begin{equation*}
				\|u_1\overline{u}_2u_3\|_{X_{\alpha}^{s,-b'}}\lesssim \sum_{j=1}^{3}\left(\|u_{j}\|_{X_{\alpha}^{s,b}}\prod_{k\ne j}\|u_{k}\|_{X_{\alpha}^{\gamma,b}}\right).
			\end{equation*}
			If we replace $\gamma_{\alpha,p}$ with $\ell_{\alpha,p}$, the above multilinear estimate still holds for $0<\alpha<1$.
		\end{prop}
		\begin{proof}
			As usual, we only consider the case $\alpha>1$, since the case $0<\alpha<1$ is similar. By duality, the multilinear estimate is reduced to 
			\begin{equation*}
				\left|\int_{\R\times\T^d}u_1\overline{u}_2u_3\overline{u}_0 \; dxdt\right|\lesssim\sum_{j=1}^{3}\|u_{0}\|_{X_{\alpha}^{-s,b'}} \left(\|u_{j}\|_{X_{\alpha}^{s,b}}\prod_{k\ne j}\|u_{k}\|_{X_{\alpha}^{\gamma,b}}\right).
			\end{equation*}
			Next, we define, for $j=1,2,3$,
			\begin{equation*}
				\mathcal{F}\omega_{j}^{(\gamma)}(\tau,k):=(1+|k|^{2})^{\frac{\gamma}{2}}(1+|\tau+|k|^{\alpha}|^{2})^{\frac{b}{2}}\mathcal{F}u_{j}(\tau,k),
			\end{equation*}
			\begin{equation*}
				\mathcal{F}\omega_{j}^{(s)}(\tau,k):=(1+|k|^{2})^{\frac{s}{2}}(1+|\tau+|k|^{\alpha}|^{2})^{\frac{b}{2}}\mathcal{F}u_{j}(\tau,k),
			\end{equation*}
			and 
			\begin{equation*}
				\mathcal{F}\omega_{0}(\tau,k):=(1+|k|^{2})^{-\frac{s}{2}}(1+|\tau+|k|^{\alpha}|^{2})^{\frac{b'}{2}}\mathcal{F}u_{0}(\tau,k).
			\end{equation*}
			
			From the definition of the Bourgain space, we can rewrite the right-hand side of the above estimate as 
			\begin{equation*}
				\sum_{j=1}^{3}\|\omega_{0}\|_{L^{2}}\left(\|\omega_{j}^{(s)}\|_{L^{2}}\prod_{k\ne j}\|\omega_{k}^{(\gamma)}\|_{L^{2}}\right),
			\end{equation*}
			
			Then we dyadically decompose $u_{j}$ and $u_{0}$ in both time and space. For dyadic integers $L_0, L_j, N_0, N_j$, we define, for $j=1,2,3$,
			\begin{equation*}
				u_{j}^{L_{j} N_{j}}(t,x):=\frac{1}{(2\pi)^{d}}\sum_{|k|\sim N_{j}}e^{ikx}\int_{|\tau+|k|^{\alpha}|\sim L_{j}}\mathcal{F}u_{j}(\tau,k)e^{it\tau}\,d\tau,
			\end{equation*}
			\begin{equation*}
				u_{0}^{L_{0} N_{0}}(t,x):=\frac{1}{(2\pi)^{d}}\sum_{|k|\sim N_{0}}e^{ikx}\int_{|\tau+|k|^{\alpha}|\sim L_{0}}\mathcal{F}u_{0}(\tau,k)e^{it\tau}\,d\tau.
			\end{equation*}
			Given $\beta, \sigma\in\R$, we can explicitly write down the $X_{\alpha}^{\sigma,\beta}$-norm for $j=1,2,3$:
			\begin{equation*}
				\left\|u_{j}^{L_{j} N_{j}}\right\|_{X_{\alpha}^{\sigma,\beta}}^{2} =\frac{1}{(2\pi)^{d}}\sum_{|k|\sim N_{j}}\int_{|\tau+|k|^{\alpha}|\sim L_{j}}(1+|k|^{2})^{\sigma-\gamma}(1+|\tau+|k|^{\alpha}|^{2})^{\beta-b}|\mathcal{F}\omega_{j}^{(\gamma)}(\tau,k)|^2\,d\tau
			\end{equation*}
			\begin{equation*}
				\lesssim N_{j}^{2(\sigma-\gamma)}L_{j}^{2(\beta-b)}\sum_{|k|\sim N_j}\int_{|\tau+|k|^{\alpha}|\sim L_{j}}|\mathcal{F}\omega_{j}^{(\gamma)}(\tau,k)|^2\,d\tau.
			\end{equation*}
			We further denote, for $j=1,2,3$,
			\begin{equation*}
				c_{j}^{\gamma}(L_j, N_j)^2:=\sum_{|k|\sim N_j}\int_{|\tau+|k|^{\alpha}|\sim L_{j}}|\mathcal{F}\omega_{j}^{(\gamma)}(\tau,k)|^2\,d\tau.
			\end{equation*}
			
			Thus we have the estimate 
			\begin{equation*}
				\left\|u_{j}^{L_j N_j}\right\|_{X_{\alpha}^{\sigma,\beta}}\lesssim N_{j}^{\sigma-\gamma}L_{j}^{\beta-b}c_{j}^{(\gamma)}(L_j,N_j),
			\end{equation*}
			with the relations
			\begin{equation*}
				\sum_{L_j,N_j}c_{j}^{(\gamma)}(L_j,N_j)^{2}\sim \|\omega_{j}^{(\gamma)}\|_{L^{2}}^2.
			\end{equation*}
			Similarly, we also denote, for $j=1,2,3$, 
			\begin{equation*}
				c_{j}^{(s)}(L_j, N_j)^2:=\sum_{|k|\sim N_j}\int_{|\tau+|k|^{\alpha}|\sim L_{j}}|\mathcal{F}\omega_{j}^{(s)}(\tau,k)|^2\,d\tau
			\end{equation*}
			and 
			\begin{equation*}
				c_{0}(L_0, N_0)^2:=\sum_{|k|\sim N_0}\int_{|\tau+|k|^{\alpha}|\sim L_{0}}|\mathcal{F}\omega_{0}(\tau,k)|^2\,d\tau.
			\end{equation*}
			Then we can obtain
			\begin{equation*}
				\left\|u_{j}^{L_j N_j}\right\|_{X_{\alpha}^{\sigma,\beta}}\lesssim N_{j}^{\sigma-s}L_{j}^{\beta-b}c_{j}^{(s)}(L_j,N_j),
			\end{equation*}
			\begin{equation*}
				\left\|u_{0}^{L_0 N_0}\right\|_{X_{\alpha}^{\sigma,\beta}}\lesssim N_{0}^{\sigma+s}L_{0}^{\beta-b'}c_{0}(L_0,N_0),
			\end{equation*}
			with the corresponding relations.
			
			We further denote $L=(L_1,L_2,L_3, L_0)$, $N=(N_1,N_2,N_3,N_0)$ and 
			\begin{equation*}
				I(L,N):=\left|\int_{\R\times\T^d}u_{1}^{L_1 N_1}\overline{u_{2}^{L_2 N_2}}u_{3}^{L_3 N_3}\overline{u_{0}^{L_0 N_0}}\,dxdt\right|.
			\end{equation*}
			Note that when $\max\lbrace N_1,N_2,N_3\rbrace\ll N_0$, $I(L,N)$ will vanish, so we can assume that $N_0\lesssim \max\lbrace N_1,N_2,N_3\rbrace$. By symmetry, we also suppose that $N_1=\max\lbrace N_1,N_2,N_3\rbrace$.
			
			Now we take $\gamma_{0}\in (\frac{\gamma_{\alpha,p}}{2}, \gamma)$, $b_{0}\in (\frac{1}{2}, b)$, and $b_{0}'\in (\frac{1}{2p}, b')$. Then, applying Lemma \ref{Stri3}, Lemma \ref{Stri4}, and H\"{o}lder's inequality, we obtain
			\begin{equation*}
				\begin{aligned}&I(L,N)\le \left\|u_{1}^{L_1 N_1}\right\|_{L^{\frac{2p}{p-1}}(\R; L^{2}(\T^d))}\left\|u_{2}^{L_2 N_2}\right\|_{L^{2p}(\R; L^{\infty}(\T^d))}\\ &\quad\times\left\|u_{3}^{L_3 N_3}\right\|_{L^{2p}(\R; L^{\infty}(\T^d))}\left\|u_{0}^{L_0 N_0}\right\|_{L^{\frac{2p}{p-1}}(\R; L^{2}(\T^d))}
				\end{aligned}
			\end{equation*}
			\begin{equation*}
				\lesssim \left\|u_{1}^{L_1 N_1}\right\|_{X_{\alpha}^{0,b_0'}}\left\|u_{2}^{L_2 N_2}\right\|_{X_{\alpha}^{\gamma_0,b_0}}\left\|u_{3}^{L_3 N_3}\right\|_{X_{\alpha}^{\gamma_0,b_0}}\left\|u_{0}^{L_0 N_0}\right\|_{X_{\alpha}^{0,b_0'}}
			\end{equation*}
			\begin{equation*}
				\begin{aligned}
					\lesssim\;&
					L_1^{b_0'-b} N_2^{\gamma_0-\gamma}
					L_2^{b_0-b} N_3^{\gamma_0-\gamma}
					L_3^{b_0-b}
					\left(\frac{N_0}{N_1}\right)^s
					L_0^{b_0'-b'}
					\\
					&\times
					c_1^{(s)}(L_1,N_1)\,
					c_2^{(\gamma)}(L_2,N_2)\,
					c_3^{(\gamma)}(L_3,N_3)\,
					c_0(L_0,N_0).
				\end{aligned}
			\end{equation*}
			Notice that the exponents on $L_1$, $N_2$, $L_2$, $N_3$, $L_3$, and $L_0$ are all negative. Summing over those variables yields
			\begin{equation*}
				\sum_{L_0,L_1} \sum_{N_2, L_2, N_3, L_3}I(L,N)\lesssim\left( \sum_{L_0,L_1}L_0^{b_0'-b'}L_1^{b_0'-b}c_1^{(s)}(L_1,N_1)c_0(L_0,N_0)\right)
			\end{equation*}
			\begin{equation*}
				\times \left(\frac{N_0}{N_1}\right)^s \|\omega_{2}^{(\gamma)}\|_{L^{2}}\|\omega_{3}^{(\gamma)}\|_{L^{2}},
			\end{equation*}
			where we used $c_2^{(\gamma)}(L_2,N_2)\lesssim \|\omega_{2}^{(\gamma)}\|_{L^{2}}$ and $c_3^{(\gamma)}(L_3,N_3)\lesssim \|\omega_{3}^{(\gamma)}\|_{L^{2}}$.
			
			Then, applying the Cauchy-Schwarz inequality, we obtain
			\begin{equation*}
				\lesssim \left(\frac{N_0}{N_1}\right)^s \left(\sum_{L_0}c_{0}(L_0,N_0)^{2}\right)^\frac{1}{2}\left(\sum_{L_1}c_{1}^{(s)}(L_1,N_1)^{2}\right)^\frac{1}{2}\|\omega_{2}^{(\gamma)}\|_{L^{2}}\|\omega_{3}^{(\gamma)}\|_{L^{2}}.
			\end{equation*}
			Note that $N_0\lesssim N_{1}$. Hence we can find an integer $M$ such that $N_1=2^{j}N_0$ for some $j\ge -M$. Finally, summing over $N_0$ and $N_1$, we obtain
			\begin{equation*}
				\begin{aligned}&\sum_{L,N}I(L,N)\lesssim \sum_{j=-M}^{\infty}\sum_{N_0}2^{-js}\left(\sum_{L_0}c_{0}(L_0,N_0)^{2}\right)^\frac{1}{2}\left(\sum_{L_1}c_{1}^{(s)}(L_1,2^{j}N_0)^{2}\right)^\frac{1}{2}\\ &\quad\times\|\omega_{2}^{(\gamma)}\|_{L^{2}}\|\omega_{3}^{(\gamma)}\|_{L^{2}}
				\end{aligned}
			\end{equation*}
			\begin{equation*}
				\lesssim \sum_{j=-M}^{\infty}2^{-js}\left(\sum_{L_0, N_0}c_{0}(L_0,N_0)^{2}\right)^\frac{1}{2}\left(\sum_{L_1, N_0}c_{1}^{(s)}(L_1,2^{j}N_0)^{2}\right)^\frac{1}{2}\|\omega_{2}^{(\gamma)}\|_{L^{2}}\|\omega_{3}^{(\gamma)}\|_{L^{2}}
			\end{equation*}
			\begin{equation*}
				\lesssim \|\omega_{0}\|_{L^{2}}\|\omega_{1}^{(s)}\|_{L^{2}}\|\omega_{2}^{(\gamma)}\|_{L^{2}}\|\omega_{3}^{(\gamma)}\|_{L^{2}}.
			\end{equation*}
			Thus we have proven Proposition \ref{multi}.
			
		\end{proof}

		\vspace{10pt}
		\section{Local well-posedness and spacetime bounds}
		
		To prove Theorem \ref{local well-posedness}, we consider the framework of contraction mapping with the following map $\Phi$:
		\begin{equation}\label{Phi}
			\Phi: u\mapsto \varphi(t)e^{-it|D|^{\alpha}}u_{0}-i\varphi\left(\frac{t}{T}\right)\int_{0}^{t}e^{-i(t-\tau)|D|^{\alpha}}|u(\tau)|^2u(\tau)\,d\tau,
		\end{equation} 
		where $\varphi\in C_{c}^{\infty}(\R)$ satisfies $\varphi(t)\equiv1$ for all $t\in [-1,1]$.
		
		Besides, we need a useful lemma from Ginibre \cite{21}, which connects the second term in (\ref{Phi}) with Proposition \ref{multi}.
		\begin{lemma}\label{Gini}
			Let $0\le b'<\frac{1}{2}$, $0<b<1-b'$, and $T \le 1$. Then we have 
			\begin{equation*}
				\left\|\varphi\left(\frac{t}{T}\right)\int_{0}^{t}e^{-i(t-\tau)|D|^{\alpha}}f(\tau)\,d\tau\right\|_{X_{\alpha}^{\gamma,b}} \le C\, T^{1-(b+b')} \|f\|_{X_{\alpha}^{\gamma,-b'}},
			\end{equation*}
			for every spatial regularity $\gamma\in\R$, with a constant independent of $0<T\le1$.
		\end{lemma}
		
		Now we are ready to establish the local well-posedness of the fractional nonlinear Schr\"{o}dinger equation (\ref{FNLS}).
		\begin{proof}[Proof of Theorem \ref{local well-posedness}]
			Consider the ball
			\[
			B_R:=\left\lbrace u\in X_{\alpha}^{\gamma,b}\;\Big|\; \|u\|_{X_{\alpha}^{\gamma,b}}\le R\|u_0\|_{H^{\gamma}(\T^d)}\right\rbrace
			\]
			where $R>0$ will be determined later.
			
			We need to show that the map $\Phi$ is a contraction on $B_R$. First, we verify that $\Phi$ maps $B_R$ into $B_R$. 
			
			From the assumption $\frac{1}{2}<b<1-\frac{1}{2p}$, we can find $b'$ such that all the conditions in Proposition \ref{multi} and Lemma \ref{Gini} are satisfied. Therefore, for $u\in B_R$, we deduce that
			\begin{equation*}
				\|\Phi(u)\|_{X_{\alpha}^{\gamma,b}}
				\le \|\varphi\|_{H^{b}(\R)}\|u_0\|_{H^{\gamma}(\T^d)}+C T^{1-(b+b')}\||u|^{2}u\|_{X_{\alpha}^{\gamma,-b'}}
			\end{equation*}
			\begin{equation*}
				\le \|\varphi\|_{H^{b}(\R)}\|u_0\|_{H^{\gamma}(\T^d)}+ \widetilde{C}T^{1-(b+b')}\|u\|_{X_{\alpha}^{\gamma,b}}^{3}.
			\end{equation*}
			Choosing $R>\|\varphi\|_{H^{b}(\R)}$ and $T=T(\|u_{0}\|_{H^{\gamma}(\T^d)})$ sufficiently small, we can ensure $\|\Phi(u)\|_{X_{\alpha}^{\gamma,b}}\le R\|u_0\|_{H^{\gamma}(\T^d)}$, i.e., $\Phi(u)\in B_R$.
			
			Next, we show that $\Phi$ is a contraction. Given $u, v\in B_R$, we apply Proposition \ref{multi} and Lemma \ref{Gini} again to obtain 
			\begin{equation*}
				\|\Phi(u)-\Phi(v)\|_{X_{\alpha}^{\gamma,b}}
			\end{equation*}
			\begin{equation*}
				=\left\|\varphi\left(\frac{t}{T}\right)\int_{0}^{t}e^{-i(t-\tau)|D|^{\alpha}}\big[|u(\tau)|^2u(\tau)-|v(\tau)|^2v(\tau)\big]\,d\tau\right\|_{X_{\alpha}^{\gamma,b}}
			\end{equation*}
			\begin{equation*}
				\le C T^{1-(b+b')}\left\||u(\tau)|^2u(\tau)-|v(\tau)|^2v(\tau)\right\|_{X_{\alpha}^{\gamma,-b'}}
			\end{equation*}
			\begin{equation}\label{102}
				\le\widetilde{C} T^{1-(b+b')}R^{2}\|u_{0}\|_{H^{\gamma}(\T^d)}^2\|u-v\|_{X_{\alpha}^{\gamma,b}}.
			\end{equation}
			Then, taking $T$ sufficiently small, we can guarantee that $\Phi$ is a contraction.
			
			Hence, there exists a unique fixed point $u\in B_R$, i.e.,
			\begin{equation*}
				u= \varphi(t)e^{-it|D|^{\alpha}}u_{0}-i\varphi\left(\frac{t}{T}\right)\int_{0}^{t}e^{-i(t-\tau)|D|^{\alpha}}|u(\tau)|^2u(\tau)\,d\tau,
			\end{equation*}
			which implies that $u$ is a solution to the nonlinear Schr\"odinger equation (\ref{FNLS}) on $[-T,T]$, with $\|u\|_{X_{\alpha}^{\gamma,b}([-T,T]\times\T^d)}\le R\|u_0\|_{H^{\gamma}(\T^d)}$. Uniqueness follows directly from the contraction and a continuity argument.
			
			If $u_0\in H^{s}(\T^d)$ with $s>\gamma$, then we define 
			\begin{equation*}
				E_R:=\left\lbrace u\;\big|\; \|u\|_{X_{\alpha}^{\gamma,b}}\le R\|u_0\|_{H^{\gamma}(\T^d)}\right\rbrace\cap \left\lbrace u\;\big|\; \|u\|_{X_{\alpha}^{s,b}}\le R\|u_0\|_{H^{s}(\T^d)}\right\rbrace
			\end{equation*}
			and apply the above procedure again. For $u\in E_R$, the preceding estimates imply, for
			$r\in\{\gamma,s\}$,
			\[
			\|\Phi(u)\|_{X_\alpha^{r,b}}
			\le
			\left(
			\|\varphi\|_{H^b(\mathbb R)}
			+\widetilde C R^3T^{1-(b+b')}
			\|u_0\|_{H^\gamma}^2
			\right)\|u_0\|_{H^r}.
			\]
			Choosing $R>\|\varphi\|_{H^b(\mathbb R)}$ and 
			$0<T\le1$ sufficiently small, one can ensure $\Phi(E_R)\subseteq E_R$, with $T$ depending only on $\|u_0\|_{H^\gamma}$.
			
			Equip $E_R$ with the metric
			\[
			d(u,v):=\|u-v\|_{X_\alpha^{\gamma,b}}.
			\]
			This metric space is complete. Indeed, a Cauchy sequence in $E_R$
			converges strongly in $X_\alpha^{\gamma,b}$ to some $u$.
			Its uniform $X_\alpha^{s,b}$ bound gives a weakly convergent
			subsequence in $X_\alpha^{s,b}$. By the continuous embedding
			$X_\alpha^{s,b}\hookrightarrow X_\alpha^{\gamma,b}$, the weak
			limit is also $u$. The bound in $X_{\alpha}^{\gamma,b}$ passes to the strong limit,
			while weak lower semicontinuity gives
			$\|u\|_{X_\alpha^{s,b}}\le R\|u_0\|_{H^s}$.
			Hence $u\in E_R$.
			
			The difference estimate (\ref{102}) gives, for $u,v\in E_R$,
			\[
			d(\Phi(u),\Phi(v))
			\le
			CR^2T^{1-(b+b')}\|u_0\|_{H^\gamma}^2d(u,v).
			\]
			Decreasing $T$ further makes the coefficient at most $\frac12$,
			without introducing any dependence on $\|u_0\|_{H^s}$.
			Banach's fixed-point theorem therefore yields a fixed point
			$u\in E_R$, and in particular
			\[
			\|u\|_{X_\alpha^{s,b}}\le R\|u_0\|_{H^s}.
			\]
			Its restriction to $[-T,T]$ solves the equation, and
			$b>\frac12$ implies $u\in C([-T,T];H^s)$.
			
			For any two local solutions with the same initial
			datum, the corresponding difference estimate on a sufficiently
			short interval has a coefficient smaller than one. They therefore agree on that interval, and
			iteration proves uniqueness, which completes the proof of Theorem~\ref{local well-posedness}.
			
		\end{proof}
		
		We next combine the local theory with an a priori polynomial
		bound on $\|u(t)\|_{H^\gamma}$ to obtain global existence and
		the spacetime estimate needed in the polynomial-growth argument.
		\begin{prop}\label{spacetime polynomial control}
			Let $u$ be a global smooth solution satisfying
			\begin{equation*}
				\|u(t)\|_{H^\gamma}\le C_0\langle t\rangle^A,
				\qquad A\ge0,
			\end{equation*}
			where $1<p<\infty$ and $\gamma>\gamma_{\alpha,p}/2$.
			Choose $\gamma_{\alpha,p}/2<\gamma_0<\gamma$, $b>1/2$,
			$b'>1/(2p)$ and $c=1-b-b'>0$. Then we have the following bound
			\begin{equation}\label{new year}
				\|u\|_{L^{2p}([0,T];W^{\gamma-\gamma_0,\infty})}
				\lesssim
				\langle T\rangle^{A+\frac{1}{2p}(1+\frac{2A}{c})}.
			\end{equation}
			In particular, when $\gamma=\frac{\alpha}{2}>\frac{\gamma_{\alpha,p}}{2}$, mass and energy conservation ensure 
			\begin{equation*}
				\|u\|_{L^{2p}([0,T];W^{\frac{\alpha}{2}-\gamma_0,\infty})}
				\lesssim \langle T\rangle^{\frac{1}{2p}}.
			\end{equation*}
		\end{prop}
	\begin{proof}
		Recall that the local theory provides a solution on an interval of length
		\[
		c_0\bigl(1+\|u(t_0)\|_{H^\gamma}\bigr)^{-\frac2c}
		\]
		starting at any time $t_0$, with the corresponding
		$X_\alpha^{\gamma,b}$ norm bounded by
		$C\|u(t_0)\|_{H^\gamma}$.
		
		Fix $T>0$. Since
		$\|u(t_0)\|_{H^\gamma}\le C_0\langle T\rangle^A$
		for $0\le t_0\le T$, we may choose a common local interval
		length
		\[
		\tau_T\sim \langle T\rangle^{-\frac{2A}{c}}.
		\]
		 Partition $[0,T]$ into consecutive intervals
		$I_j=[t_j,t_{j+1}]$ of length $\tau_T$. The number $J$ of intervals satisfies
		\[
		J\le1+\frac{T}{\tau_T}
		\lesssim\langle T\rangle^{1+\frac{2A}{c}}.
		\]
		On each $I_j$, applying the Strichartz
		estimate (\ref{local strichartz}) to $\langle D\rangle^{\gamma-\gamma_0}u$
		and using the local solution bound gives
		\[
		\|u\|_{L^{2p}(I_j;W^{\gamma-\gamma_0,\infty})}
		\lesssim
		\|u\|_{X_\alpha^{\gamma,b}(I_j)}
		\lesssim
		\|u(t_j)\|_{H^\gamma}
		\lesssim\langle T\rangle^A.
		\]
		Summing the $2p$-th powers over these intervals, we obtain
		\[
		\|u\|_{L^{2p}([0,T];W^{\gamma-\gamma_0,\infty})}^{2p}
		\lesssim
		J\langle T\rangle^{2pA}
		\lesssim
		\langle T\rangle^{2pA+1+\frac{2A}{c}},
		\]
		which proves \eqref{new year}.
		
	\end{proof}

		For later use, H\"older's inequality and \eqref{new year} give
		\begin{equation*}
			\int_0^T
			\|u(t)\|_{W^{\gamma-\gamma_0,\infty}}^2\,dt
			\le
			T^{1-\frac1p}
			\|u\|_{L^{2p}([0,T];
				W^{\gamma-\gamma_0,\infty})}^2
			\lesssim
			\langle T\rangle^{1+2A+\frac{2A}{pc}}.
		\end{equation*}
		Choosing $b=\frac12+\varepsilon_1$ and
		$b'=\frac1{2p}+\varepsilon_2$, with $\varepsilon_1,\varepsilon_2>0$
		sufficiently small, gives
		\begin{equation*}
			pc=p(1-b-b')
			=\frac{p-1}{2}-p(\varepsilon_1+\varepsilon_2).
		\end{equation*}
		Thus $pc$ can approach $\frac{p-1}{2}$ from below.
		Consequently, for every $\varepsilon>0$, a suitable choice
		of $b$ and $b'$ yields
		\begin{equation*}
			\int_0^T
			\|u(t)\|_{W^{\gamma-\gamma_0,\infty}}^2\,dt
			\lesssim_\varepsilon
			\langle T\rangle^{1+2A+\frac{4A}{p-1}+\varepsilon}.
		\end{equation*}

\vspace{10pt}
\section{Polynomial growth of Sobolev norms}

For smooth solutions $u$ of cubic equation (\ref{FNLS}), mass and energy are conserved:
\begin{equation*}
	M(u)=\|u\|_2^2,
	\qquad
	E(u)=\frac12\||D|^{\frac{\alpha}{2}}u\|_2^2
	+\frac14\|u\|_4^4.
\end{equation*}
 Consequently, $H^{\frac{\alpha}{2}}$-norm and $L^4$-norm are uniformly bounded.

First, we briefly recall some well-known results, which will be used frequently.
\begin{lemma}\label{Sobolev product estimate}
	For $s>0$ and smooth functions $f,g$ on $\T^d$,
	\begin{equation*}
		\|fg\|_{H^s}
		\lesssim
		\|f\|_\infty\|g\|_{H^s}
		+\|g\|_\infty\|f\|_{H^s}.
	\end{equation*}
	In particular,
	\begin{equation}\label{nf:eq:tame}
		\||u|^2u\|_{H^s}
		\lesssim
		\|u\|_\infty^2\|u\|_{H^s}.
	\end{equation}
\end{lemma}

Applying interpolation and Sobolev embedding (or the Gagliardo-Nirenberg inequality), we can obtain 
\begin{lemma}\label{Gag}
	Let $0<\theta<1$ and $a,s_0,s_1\ge0$. If
	\begin{equation*}
		a+\frac d2(1-\theta)
		<(1-\theta)s_0+\theta s_1,
	\end{equation*}
	then
	\begin{equation}\label{xvv}
		\|u\|_{W^{a,\frac2\theta}}
		\lesssim
		\|u\|_{H^{s_0}}^{1-\theta}
		\|u\|_{H^{s_1}}^\theta.
	\end{equation}
	Moreover, if $r,s\ge0$ and
	$a<(1-\theta)r+\theta s$, then
	\begin{equation}\label{xv}
		\|u\|_{W^{a,\frac2\theta}}
		\lesssim
		\|u\|_{W^{r,\infty}}^{1-\theta}
		\|u\|_{H^s}^{\theta}.
	\end{equation}
\end{lemma}

Next, we combine the fractional Leibniz rule in Appendix B to control the commutator.

\begin{lemma}\label{eee}
	Let $\alpha>\frac d2$, let $n\ge0$ be an integer, and define the commutator
	\[
	F_\alpha(u)
	:=
	u|D|^\alpha\bar u+\bar u|D|^\alpha u-|D|^\alpha(|u|^2).
	\]
	If $n\ge1$ and
	\[
	0<\eta<\min\left\{1,\alpha-\frac d2\right\},
	\]
	then
	\begin{equation*}
		\|F_\alpha(u)\|_{H^n}
		\lesssim
		\|u\|_{H^\alpha}\|u\|_{H^{\alpha+n-\eta}}
		\lesssim
		\|u\|_{H^\alpha}^{1+\frac{\eta}{n}}
		\|u\|_{H^{\alpha+n}}^{1-\frac{\eta}{n}}.
	\end{equation*}
	If $0<\delta<r<\min\{1,\frac{\alpha}{2}\}$, then
	\begin{equation*}
		\|F_\alpha(u)\|_{H^n}
		\lesssim
		\|u\|_{W^{r,\infty}}
		\|u\|_{H^{\alpha+n-\delta}}.
	\end{equation*}
\end{lemma}

\begin{proof}
	We use the periodic version of Theorem~\ref{fractional Leibniz 2}, Case~1(i), with
	$s=\alpha$, $s_1=s_2=\frac{\alpha}{2}$, and $p=2$.
	
	First suppose $n\ge1$ and put $S=\alpha+n-\eta$.
	For $|\kappa|\le n$, the ordinary Leibniz rule writes
	$\partial^\kappa F_\alpha(u)$ as a finite sum of
	\[
	f|D|^\alpha g+g|D|^\alpha f-|D|^\alpha(fg),
	\qquad
	f=\partial^\beta u,\quad
	g=\partial^{\kappa-\beta}\bar u,\quad
	\beta\le\kappa.
	\]
	Theorem~\ref{fractional Leibniz 2} bounds each summand by a remainder
	\[
	\bigl\||D|^{\frac{\alpha}{2}}f\bigr\|_{L^{p_1}}
	\bigl\||D|^{\frac{\alpha}{2}}g\bigr\|_{L^{p_2}},
	\qquad
	\frac1{p_1}+\frac1{p_2}=\frac12,
	\]
	together with the $L^2$ norms of
	\[
	\partial^\mu f\,|D|^{\alpha,\mu}g,
	\qquad
	\partial^\mu g\,|D|^{\alpha,\mu}f,
	\qquad
	1\le|\mu|\le\frac{\alpha}{2}.
	\]
	The latter sum is empty when $\alpha<2$.
	
	For the product
	$\bigl\||D|^{\frac{\alpha}{2}}f\bigr\|_{L^{p_1}}
	\bigl\||D|^{\frac{\alpha}{2}}g\bigr\|_{L^{p_2}}$,
	the derivative orders on $u$ and $\bar u$ are
	\[
	a=|\beta|+\frac{\alpha}{2},
	\qquad
	b=|\kappa-\beta|+\frac{\alpha}{2},
	\]
	so $a,b\le n+\frac{\alpha}{2}$.
	For $\partial^\mu f\,|D|^{\alpha,\mu}g$, Remark~\ref{cv}
	and H\"older's inequality give the derivative orders
	\[
	a=|\beta|+|\mu|\le n+\frac{\alpha}{2},
	\qquad
	b=|\kappa-\beta|+\alpha-|\mu|
	\le n+\alpha-1,
	\]
	since $1\le|\mu|\le\frac{\alpha}{2}$.
	The other product is treated by interchanging $f$ and $g$.
	
	Thus, in every case, the two derivative orders satisfy
	\[
	a,b<S=\alpha+n-\eta,
	\qquad
	a+b=\alpha+|\kappa|
	\le S+\eta<S+\alpha-\frac d2.
	\]
	Here $\eta<\frac{\alpha}{2}$ follows from
	$d\ge2$ and $\eta<\min\{1,\alpha-\frac d2\}$.
	The preceding inequalities allow a choice $0<\theta<1$ such that
	\[
	a<(1-\theta)\left(\alpha-\frac d2\right)+\theta S,
	\qquad
	b<\theta\left(\alpha-\frac d2\right)+(1-\theta)S.
	\]
	Apply (\ref{xvv}) in Lemma~\ref{Gag} with
	$p_1=\frac2\theta$ and $p_2=\frac2{1-\theta}$.
	Each product is bounded by
	\[
	\left(
	\|u\|_{H^\alpha}^{1-\theta}\|u\|_{H^S}^{\theta}
	\right)
	\left(
	\|u\|_{H^\alpha}^{\theta}\|u\|_{H^S}^{1-\theta}
	\right)
	=
	\|u\|_{H^\alpha}\|u\|_{H^S}.
	\]
	Summation proves the first inequality in the first estimate.
	The second follows from
	\[
	\|u\|_{H^{\alpha+n-\eta}}
	\le
	\|u\|_{H^\alpha}^{\frac{\eta}{n}}
	\|u\|_{H^{\alpha+n}}^{1-\frac{\eta}{n}}.
	\]
	
	For the second estimate, let $n\ge0$ and put
	$S=\alpha+n-\delta$.
	Use the same expansion of $\partial^\kappa F_\alpha(u)$.
	Since $\delta<\min\{1,\frac{\alpha}{2}\}$ and $\delta<r$,
	the derivative orders above satisfy
	\[
	a,b<S,
	\qquad
	a+b=\alpha+|\kappa|
	\le S+\delta<S+r.
	\]
	Consequently, we can choose $0<\theta<1$ such that
	\[
	a<(1-\theta)r+\theta S,
	\qquad
	b<\theta r+(1-\theta)S.
	\]
	Indeed, if $S>r$, these follow from
	$a,b<S$ and $a+b<S+r$; if $S\le r$, one may take
	$\theta=\frac12$.
	
	Applying (\ref{xv}) in Lemma~\ref{Gag} with
	$p_1=\frac2\theta$ and $p_2=\frac2{1-\theta}$ gives
	\[
	\|u\|_{W^{a,p_1}}
	\lesssim
	\|u\|_{W^{r,\infty}}^{1-\theta}
	\|u\|_{H^S}^{\theta},
	\qquad
	\|u\|_{W^{b,p_2}}
	\lesssim
	\|u\|_{W^{r,\infty}}^{\theta}
	\|u\|_{H^S}^{1-\theta}.
	\]
	Hence each remainder and correction term is bounded by
	\[
	\|u\|_{W^{r,\infty}}\|u\|_{H^S}.
	\]
	Summing over the terms and over $|\kappa|\le n$ proves
	the second estimate.
	
\end{proof}

Similarly, we also have the following estimate on commutator (see also (\ref{cbb})):
\begin{lemma}\label{eee1}
	Let $\alpha>d$, $n\ge0$ be an integer, and choose
	\begin{equation*}
		0<\eta<\min\left\{1,\frac{\alpha-d}{2}\right\}.
	\end{equation*}
	Then
	\begin{equation*}
		\|F_{\alpha}(u)\|_{H^n}
		\lesssim
		\|u\|_{H^{\frac{\alpha}{2}}}
		\|u\|_{H^{\alpha+n-\eta}}
		\lesssim
		\|u\|_{H^{\frac{\alpha}{2}}}^{1+\frac{2\eta}{2n+\alpha}}
		\|u\|_{H^{\alpha+n}}^{1-\frac{2\eta}{2n+\alpha}}.
	\end{equation*}
\end{lemma}

In order to derive polynomial growth, we need the following variants of Gronwall's inequality.
\begin{lemma}\label{Variant2}
	Let $0<\lambda_k\le 1$, $A_k\ge0$ and $p_k\ge 1$, $k=1,2,\cdots,m$. If $f\in C^{1}([0,+\infty))$, $g_{k}\in L^{p_k}_{loc}([0,+\infty))$ satisfy
	\[
	f'(t)\le \sum_{k=1}^{m} f(t)^{1-\lambda_k} g_k(t),\quad \|g_k\|_{L^{p_k}([0,t])}\lesssim\langle t\rangle^{A_k},
	\]
	with $f(t), g_k(t)\ge0$, then we have the following bound:
	\[
	f(t)\lesssim \langle t\rangle^{\Gamma}, \quad 
	\Gamma:=\max_{1\le k\le m}\left\{\frac{A_k+1/p_k'}{\lambda_k}\right\}.
	\]
\end{lemma}

\begin{proof}
	Integrating the differential inequality gives
	\[
	f(t)\le f(0)+\sum_{k=1}^{m}
	\int_0^t f(s)^{1-\lambda_k} g_k(s)\,ds.
	\]
	Fix $T>0$ and set
	\[
	M(T):=\sup_{0\le s\le T}\frac{f(s)}{\langle s\rangle^{\Gamma}}.
	\]
	Then for $s\le T$,
	\[
	f(s)\le M(T)\langle s\rangle^{\Gamma}.
	\]
	Substituting this into the integral inequality yields
	\[
	f(t)
	\le
	f(0)
	+
	\sum_{k=1}^{m}
	M(T)^{1-\lambda_k}
	\int_0^t
	\langle s\rangle^{\Gamma(1-\lambda_k)}
	g_k(s)\,ds.
	\]
	By H\"{o}lder's inequality,
	\[
	\int_0^t
	\langle s\rangle^{\Gamma(1-\lambda_k)}
	g_k(s)\,ds
	\le
	\|g_k\|_{L^{p_k}([0,t])}
	\,
	\left\|\langle s\rangle^{\Gamma(1-\lambda_k)}\right\|_{L^{p_k'}([0,t])}
	\lesssim
	\langle t\rangle^{A_k+\Gamma(1-\lambda_k)+1/p_k'}.
	\]
	Because of the choice $\Gamma=\max_{1\le k\le m}
	\left\{\frac{A_k+1/p_k'}{\lambda_k}\right\}$, we have 
	\[
	A_k+\Gamma(1-\lambda_k)+\frac1{p_k'}
	\le \Gamma, \quad \forall k=1,2,\cdots, m.
	\]
	Therefore
	\[
	f(t)
	\lesssim
	f(0)
	+
	\sum_{k=1}^{m}
	M(T)^{1-\lambda_k}
	\langle t\rangle^{\Gamma}.
	\]
	Dividing by $\langle t\rangle^{\Gamma}$ and taking the supremum over $0\le t\le T$ gives
	\[
	M(T)
	\le
	C_1+C_2\sum_{k=1}^{m} M(T)^{1-\lambda_k}.
	\]
	Since $1-\lambda_k<1$, the right-hand side is sublinear in $M(T)$; hence $M(T)$ remains bounded uniformly in $T$. The conclusion follows.
	
\end{proof}

	As a direct corollary, we have the following control:
\begin{lemma}\label{Variant}
	Let \(0<\lambda_k\le 1\) and \(\beta_k\ge0\), \(k=1,2,\dots,m\). If \(f\in C^{1}([0,\infty))\) satisfies
	\[
	f'(t)\le \sum_{k=1}^{m}f(t)^{1-\lambda_k}\langle t\rangle^{\beta_k},
	\quad 
	\forall t\ge0,
	\]
	with \(f(t)\ge 0\), then we have the following bound:
	\[
	f(t)\lesssim 
	\langle t\rangle^{\alpha_*},
	\quad
	\alpha_*=
	\max_{1\le k\le m}\Big\{
	\frac{\beta_k+1}{\lambda_k}
	\Big\}.
	\]
\end{lemma}

\vspace{7pt}
\subsection{The case $\alpha>d$}
For completeness, we first establish the existence of smooth solutions when the initial data $u_0$ is smooth. By the second part of Theorem~\ref{local well-posedness}, there exists $T_0>0$ such that on the interval $[-T_0,T_0]$ one has a smooth solution
$u\in C^{\infty}([-T_0,T_0];C^{\infty}(\T^d))$.
Applying the polynomial growth estimate for Sobolev norms proved below on the maximal interval of existence, we deduce that no Sobolev norm can blow up at a finite endpoint. Consequently, the solution extends globally.

Throughout this case, fix
\begin{equation*}
	0<\eta<\min\left\{1,\frac{\alpha-d}{2}\right\},
	\qquad
	\theta_{\alpha,n}:=\frac{2\eta}{2n+\alpha}.
\end{equation*}
Recall that from conservation and Sobolev embedding, we have
\begin{equation*}
	\|u(t)\|_{L^\infty(\T^d)}\lesssim \|u(t)\|_{H^{\frac{\alpha}{2}}(\T^d)}\lesssim_{u_0} 1
	\qquad \alpha>d.
\end{equation*}

Next, we introduce the following modified energy:
\begin{equation*}
	\mathcal{E}_{\alpha,n}(u)
	:=\|u\|_{L^{2}(\T^d)}^2
	+\||D|^{\alpha+n}u\|_{L^{2}(\T^d)}^2
	+J_1(u)+J_2(u),
	\qquad n\ge0,
\end{equation*}
where $n$ is an integer and
\begin{equation*}
	J_1(u)
	:=2\Re \bigl(|D|^{\alpha+n}u,|D|^n(|u|^2u)\bigr),
	\qquad
	J_2(u)
	:=-\frac{1}{2}
	\bigl\||D|^{\frac{\alpha}{2}+n}(|u|^2)\bigr\|_{L^2(\T^d)}^2.
\end{equation*}

The modified energy $\mathcal{E}_{\alpha,n}(u)$ serves as an alternative to the squared $H^{\alpha+n}$-norm of $u$, with a more favorable time derivative.

To show this, we only need to control $J_1(u)$ and $J_2(u)$, since the first two terms of $\mathcal{E}_{\alpha,n}(u)$ are equivalent to $\|u\|_{H^{\alpha+n}(\T^d)}^2$. Applying product estimate (\ref{nf:eq:tame}) and interpolation, one can see
\begin{equation*}
	|J_1(u)|
	\lesssim
	\|u\|_{H^{\alpha+n}(\T^d)}
	\|u\|_{H^n(\T^d)}
	\|u\|_{L^\infty(\T^d)}^2
	\lesssim_{u_0}
	\|u\|_{H^{\alpha+n}(\T^d)}^{2-\varepsilon_{\alpha,n}},
\end{equation*}
and
\begin{equation*}
	|J_2(u)|
	\lesssim
	\|u\|_{L^\infty(\T^d)}^2
	\|u\|_{H^{\frac{\alpha}{2}+n}(\T^d)}^2
	\lesssim_{u_0}
	\|u\|_{H^{\alpha+n}(\T^d)}^{2-\varepsilon_{\alpha,n}},
\end{equation*}
where
\begin{equation*}
	\varepsilon_{\alpha,n}
	:=\min\left\{\frac{2\alpha}{2n+\alpha},1\right\}.
\end{equation*}
Since $\varepsilon_{\alpha,n}>0$, there exists a constant $K_{\alpha,n}>0$, such that
\begin{equation}\label{high energy equivalence}
	K_{\alpha,n}+\mathcal{E}_{\alpha,n}(u)
	\sim_{u_0}
	1+\|u\|_{H^{\alpha+n}(\T^d)}^2,
	\qquad
	K_{\alpha,n}+\mathcal{E}_{\alpha,n}(u)\ge1.
\end{equation}

Next, we calculate the time derivative of $\mathcal{E}_{\alpha,n}(u)$, where $\dot u$ denotes $\partial_tu$.

Since the $L^2$-norm is conserved, we denote
$J_0(u):=\||D|^{\alpha+n}u\|_{L^2(\T^d)}^2$. Then
\begin{equation*}
	\frac{d}{dt}J_0(u)
	=2\Re\bigl(|D|^{\alpha+n}\dot u,|D|^{\alpha+n}u\bigr)
	=2\Im\bigl(|D|^{\alpha+n}\dot u,|D|^n\dot u\bigr)
	-2\Re\bigl(|D|^{\alpha+n}\dot u,|D|^n(|u|^2u)\bigr),
\end{equation*}
where we substituted the equation
$|D|^\alpha u=i\dot u-|u|^2u$.
The first term on the right-hand side vanishes. Combining this with the time derivative of $J_1(u)$, we obtain
\begin{equation*}
	\frac{d}{dt}[J_0+J_1](u)
	=2\Re\bigl(|D|^{\alpha+n}u,|D|^n(|u|^2u)^{\cdot}\bigr)
	=2\sum_{|\beta|=n}\frac{n!}{\beta!}
	\Re\bigl(\partial_x^\beta|D|^\alpha u,
	\partial_x^\beta(\dot u|u|^2+u(|u|^2)^{\cdot})\bigr).
\end{equation*}
For a fixed $\beta$, the Leibniz rule gives
\begin{equation}\label{qq}
	2\Re\bigl(\partial_x^\beta|D|^\alpha u,
	\partial_x^\beta(\dot u|u|^2+u(|u|^2)^{\cdot})\bigr)
	=I_\beta+II_\beta+III_\beta,
\end{equation}
where
\begin{equation*}
	I_\beta
	:=2\Re\bigl(\partial_x^\beta|D|^\alpha u,
	(\partial_x^\beta\dot u)|u|^2\bigr),
\end{equation*}
\begin{equation*}
	II_\beta
	:=2\Re\left(
	\partial_x^\beta|D|^\alpha u,
	\sum_{\substack{\rho\le\beta\\|\rho|<n}}
	\binom{\beta}{\rho}
	\left[
	(\partial_x^\rho\dot u)(\partial_x^{\beta-\rho}|u|^2)
	+(\partial_x^{\beta-\rho}u)
	(\partial_x^\rho(|u|^2)^{\cdot})
	\right]\right),
\end{equation*}
and
\begin{equation*}
	III_\beta
	:=2\Re\bigl(\partial_x^\beta|D|^\alpha u,
	u\partial_x^\beta(|u|^2)^{\cdot}\bigr).
\end{equation*}
Here $\rho\le\beta$ is understood componentwise.

For the first term, we substitute
$|D|^\alpha u=i\dot u-|u|^2u$ and obtain
\begin{equation*}
	|I_\beta|
	=\left|2\Re\bigl(
	\partial_x^\beta(|u|^2u),
	(\partial_x^\beta\dot u)|u|^2\bigr)\right|
	\lesssim
	\||u|^2u\|_{H^n(\T^d)}
	\|u\|_{L^\infty(\T^d)}^2
	\|\dot u\|_{H^n(\T^d)}.
\end{equation*}
The equation gives
\begin{equation*}
	\|\dot u\|_{H^n(\T^d)}
	\lesssim
	\|u\|_{H^{\alpha+n}(\T^d)}
	+\|u\|_{L^\infty(\T^d)}^2\|u\|_{H^n(\T^d)}.
\end{equation*}
Consequently, the same interpolation as above yields
\begin{equation*}
	|I_\beta|
	\lesssim_{u_0}
	\bigl(1+\|u\|_{H^{\alpha+n}(\T^d)}\bigr)^{2-\varepsilon_{\alpha,n}}.
\end{equation*}

For the second term in \eqref{qq}, fix $\rho\le\beta$ with
$|\rho|=k<n$. Choose
\begin{equation*}
	0<s_1<\min\left\{\frac d2,n-k-\eta\right\},
	\qquad
	s_2:=\frac d2-s_1.
\end{equation*}
This is possible because $n-k\ge1>\eta$.
H\"older's inequality and Sobolev embedding give
\begin{equation*}
	\left|
	2\Re\bigl(
	\partial_x^\beta|D|^\alpha u,
	(\partial_x^\rho\dot u)
	(\partial_x^{\beta-\rho}|u|^2)
	\bigr)
	\right|
	\lesssim
	\|u\|_{H^{\alpha+n}}
	\|\dot u\|_{W^{k,\frac{2d}{d-2s_1}}}
	\||u|^2\|_{W^{n-k,\frac{2d}{d-2s_2}}}.
\end{equation*}
Using the equation, product estimate (\ref{nf:eq:tame}) and Sobolev embedding, we obtain
\begin{equation*}
	\|\dot u\|_{W^{k,\frac{2d}{d-2s_1}}}
	\lesssim_{u_0}\|u\|_{H^{\alpha+k+s_1}},
	\qquad
	\||u|^2\|_{W^{n-k,\frac{2d}{d-2s_2}}}
	\lesssim_{u_0}\|u\|_{H^{n-k+s_2}}.
\end{equation*}
For the other summand in $II_\beta$, the equation also gives
\begin{equation*}
	\|(|u|^2)^{\cdot}\|_{W^{k,\frac{2d}{d-2s_1}}}
	\lesssim
	\|(|u|^2)^{\cdot}\|_{H^{k+s_1}}
	\lesssim_{u_0}\|u\|_{H^{\alpha+k+s_1}}.
\end{equation*}
Hence both summands are bounded by
\begin{equation*}
	\|u\|_{H^{\alpha+n}}
	\|u\|_{H^{\alpha+k+s_1}}
	\|u\|_{H^{n-k+s_2}}.
\end{equation*}

Note that the last two orders are strictly smaller than $\alpha+n$.
For orders at least $\frac{\alpha}{2}$, interpolate between
$H^{\frac{\alpha}{2}}$ and $H^{\alpha+n}$; for smaller orders,
use the conserved $H^{\frac{\alpha}{2}}$ bound directly.
Since $\alpha+k+s_1>\frac{\alpha}{2}$ and $s_1+s_2=\frac d2$,
this gives
\begin{equation*}
	\|u\|_{H^{\alpha+k+s_1}}
	\|u\|_{H^{n-k+s_2}}
	\lesssim_{u_0}
	\bigl(1+\|u\|_{H^{\alpha+n}}\bigr)^{
		\frac{
			\max\{\frac{\alpha}{2}+k+s_1,\,
			n+\frac d2\}
		}{n+\frac{\alpha}{2}}}.
\end{equation*}
Our choices imply
\begin{equation*}
	\frac{\alpha}{2}+k+s_1
	<n+\frac{\alpha}{2}-\eta,
	\qquad
	n+\frac d2
	<n+\frac{\alpha}{2}-\eta.
\end{equation*}
Summing over $\rho$ therefore yields
\begin{equation}\label{q}
	|II_\beta|
	\lesssim_{u_0}
	\bigl(1+\|u\|_{H^{\alpha+n}}\bigr)^{
		2-\theta_{\alpha,n}},
	\qquad
	\theta_{\alpha,n}=\frac{2\eta}{2n+\alpha}.
\end{equation}

For the third term, we have
\begin{equation*}
	III_\beta
	=2\Re\bigl(
	\bar u\partial_x^\beta|D|^\alpha u,
	\partial_x^\beta(|u|^2)^{\cdot}\bigr),
\end{equation*}
and hence
\begin{equation*}
	III_\beta
	=2\Re\bigl(
	\partial_x^\beta(\bar u|D|^\alpha u),
	\partial_x^\beta(|u|^2)^{\cdot}\bigr)
	-2\sum_{\substack{\rho\le\beta\\|\rho|<n}}
	\binom{\beta}{\rho}
	\Re\bigl(
	(\partial_x^{\beta-\rho}\bar u)
	(\partial_x^\rho|D|^\alpha u),
	\partial_x^\beta(|u|^2)^{\cdot}\bigr).
\end{equation*}
Note that the second term above can be bounded by the right-hand side of (\ref{q}) in a similar manner as for the term $II_\beta$.

Summing $III_\beta$ with weights $\frac{n!}{\beta!}$, the remaining term is
\begin{equation*}
	2\Re\bigl(
	|D|^n(\bar u|D|^\alpha u),
	|D|^n(|u|^2)^{\cdot}\bigr).
\end{equation*}
Combining this with
\begin{equation*}
	\frac{d}{dt}J_2(u)
	=-\bigl(
	|D|^{\frac{\alpha}{2}+n}(|u|^2),
	|D|^{\frac{\alpha}{2}+n}(|u|^2)^{\cdot}\bigr)
	=-\bigl(
	|D|^{\alpha+n}(|u|^2),
	|D|^n(|u|^2)^{\cdot}\bigr),
\end{equation*}
we are left with
\begin{equation}\label{Leibniz structure}
	\bigl(
	|D|^n(u|D|^\alpha\bar u+\bar u|D|^\alpha u
	-|D|^\alpha(u\bar u)),
	|D|^n(|u|^2)^{\cdot}\bigr).
\end{equation}

When $n=0$, Cauchy--Schwarz bounds \eqref{Leibniz structure} by
\begin{equation*}
	\|u|D|^\alpha\bar u+\bar u|D|^\alpha u
	-|D|^\alpha(u\bar u)\|_{L^2(\T^d)}
	\|(|u|^2)^{\cdot}\|_{L^2(\T^d)}
	=:I_1I_2.
\end{equation*}
Lemma~\ref{eee1} gives
\begin{equation*}
	I_1
	\lesssim
	\|u\|_{H^{\frac{\alpha}{2}}(\T^d)}^{1+\frac{2\eta}{\alpha}}
	\|u\|_{H^\alpha(\T^d)}^{1-\frac{2\eta}{\alpha}}.
\end{equation*}
For $I_2$, the equation also yields
\begin{equation*}
	I_2
	\lesssim
	\|u\|_{L^\infty(\T^d)}
	\|u\|_{H^\alpha(\T^d)}
	\lesssim_{u_0}\|u\|_{H^\alpha(\T^d)}.
\end{equation*}
Since $II_\beta$ and the sum in the expansion of $III_\beta$ vanish when $n=0$, it follows that
\begin{equation*}
	\left|\frac{d}{dt}\mathcal{E}_{\alpha,0}(u)\right|
	\lesssim_{u_0}
	\bigl(1+\|u\|_{H^\alpha(\T^d)}\bigr)^{2-\theta_{\alpha,0}}.
\end{equation*}
By \eqref{high energy equivalence} and Lemma~\ref{Variant}, applied with $\beta_1=0$ and
$\lambda_1=\frac{\theta_{\alpha,0}}{2}$, this implies
\begin{equation*}
	\|u(t)\|_{H^\alpha(\T^d)}
	\lesssim_{u_0,\eta}
	\langle t\rangle^{\frac{\alpha}{2\eta}}.
\end{equation*}

Now we are ready to bound the $H^{\alpha+n}$-norm.
By Cauchy--Schwarz and Lemma~\ref{eee1}, the absolute value of
\eqref{Leibniz structure} is bounded by
\begin{equation*}
	\|F_\alpha(u)\|_{H^n(\T^d)}
	\|(|u|^2)^{\cdot}\|_{H^n(\T^d)}
	\lesssim_{u_0}
	\|u\|_{H^{\alpha+n}(\T^d)}^{2-\theta_{\alpha,n}}.
\end{equation*}
Combining this with \eqref{qq}, \eqref{q}, and the estimate for $I_\beta$, we obtain
\begin{equation*}
	\left|\frac{d}{dt}\mathcal{E}_{\alpha,n}(u)\right|
	\lesssim_{u_0}
	\bigl(1+\|u\|_{H^{\alpha+n}(\T^d)}\bigr)^{2-\varepsilon_{\alpha,n}}
	+\bigl(1+\|u\|_{H^{\alpha+n}(\T^d)}\bigr)^{2-\theta_{\alpha,n}}
	\lesssim_{u_0}
	\bigl(1+\|u\|_{H^{\alpha+n}(\T^d)}\bigr)^{2-\theta_{\alpha,n}},
\end{equation*}
where $0<\theta_{\alpha,n}<\varepsilon_{\alpha,n}$.
By \eqref{high energy equivalence},
\begin{equation*}
	\left|\frac{d}{dt}
	\bigl(K_{\alpha,n}+\mathcal{E}_{\alpha,n}(u)\bigr)\right|
	\lesssim_{u_0}
	\bigl(K_{\alpha,n}+\mathcal{E}_{\alpha,n}(u)\bigr)^{
		1-\frac{\theta_{\alpha,n}}{2}}.
\end{equation*}
Applying Lemma~\ref{Variant} with $\beta_1=0$ and
$\lambda_1=\frac{\theta_{\alpha,n}}{2}$, in both time directions, gives
\begin{equation*}
	\|u(t)\|_{H^{\alpha+n}(\T^d)}
	\lesssim_{u_0,\eta}
	\langle t\rangle^{\frac{1}{\theta_{\alpha,n}}}
	=
	\langle t\rangle^{\frac{2n+\alpha}{2\eta}}.
\end{equation*}
These estimates hold uniformly up to any finite endpoint of the maximal interval of existence, so persistence of regularity extends the smooth solution globally.
Finally, taking $\eta$ sufficiently close to
$\min\{1,\frac{\alpha-d}{2}\}$, we conclude that,
\begin{equation*}
	\|u(t)\|_{H^{\alpha+n}(\T^d)}
	\lesssim_{u_0}
	\langle t\rangle^{
		\frac{2n+\alpha}{\min\{\alpha-d,2\}}+},
	\qquad \forall t\in\mathbb R.
\end{equation*}

\vspace{7pt}
\subsection{The case $\frac{d^2}{d+2}<\alpha\le d$}

Recall that
\begin{equation*}
	\kappa_\alpha
	:=\min\left\{2,\alpha-\frac{d^2}{d+2}\right\}.
\end{equation*}
We prove that, for every integer $n\ge0$,
\begin{equation}\label{low polynomial growth}
	\|u(t)\|_{H^{\alpha+n}}
	\lesssim_{u_0}
	\langle t\rangle^{\frac{2n+\alpha}{\kappa_\alpha}+},
	\qquad \forall t\in\R.
\end{equation}
We still use the modified energy $\mathcal E_{\alpha,n}$ and
the decomposition \eqref{qq} from the preceding subsection.

First, we collect all the product estimates that will be needed later into this lemma.
The following product estimates will allow us to retain
the $W^{r,\infty}$ factors in its time derivative.
\begin{lemma}\label{low product bounds}
	Let $\alpha>\frac{d}{2}$,  $n\ge0$ be an integer,
	and suppose
	\begin{equation*}
		0<\delta<r<\min\left\{1,\frac{\alpha}{2}\right\},
		\qquad \delta<2\alpha-d.
	\end{equation*}
	For smooth $u$ on $\T^d$,
	\begin{equation}\label{low cubic product estimate}
		\||u|^2u\|_{H^n}
		\lesssim
		\|u\|_{H^{\frac{\alpha}{2}}}^2
		\|u\|_{H^{\alpha+n-\delta}},
	\end{equation}
	\begin{equation}\label{low density product estimate}
		\|\bar u|D|^\alpha u\|_{H^n}
		\lesssim
		\|u\|_{\infty}
		\|u\|_{H^{\alpha+n}}.
	\end{equation}
	If $|\beta|=n$ and $\rho\le\beta$, then
	\begin{equation}\label{low quintic product estimate}
		\bigl\|
		\partial^\rho(|u|^2u)
		\partial^{\beta-\rho}(|u|^2)
		\bigr\|_2
		\lesssim
		\|u\|_\infty^2
		\|u\|_{H^{\frac{\alpha}{2}}}^2
		\|u\|_{H^{\alpha+n-\delta}}.
	\end{equation}
	If, in addition, $|\rho|<n$, then
	\begin{equation}\label{low bilinear product estimate}
		\bigl\|
		(\partial^{\beta-\rho}\bar u)
		(\partial^\rho|D|^\alpha u)
		\bigr\|_2
		\lesssim
		\|u\|_{W^{r,\infty}}
		\|u\|_{H^{\alpha+n-\delta}},
	\end{equation}
	and
	\begin{equation}\label{low product estimate}
		\bigl\|
	(\partial^\rho|D|^\alpha u)
	\partial^{\beta-\rho}(|u|^2)
	\bigr\|_2+
	\bigl\|
	(\partial^{\beta-\rho}u)
	\partial^\rho(\bar u|D|^\alpha u)
	\bigr\|_2
	\lesssim
	\|u\|_{W^{r,\infty}}^2
	\|u\|_{H^{\alpha+n-\delta}}.
	\end{equation}
\end{lemma}

\begin{proof}
For convenience, we assume $\alpha\le d$ below, since the case $\alpha>d$ follows by simpler arguments. For \eqref{low cubic product estimate}, expand
$\partial^\beta(|u|^2u)$ for $|\beta|\le n$.
Each term has the form
\begin{equation*}
	(\partial^{\beta_1}u)
	(\partial^{\beta_2}\bar u)
	(\partial^{\beta_3}u),
	\qquad
	\beta_1+\beta_2+\beta_3=\beta.
\end{equation*}
Since $\alpha\le d$ and $\delta<2\alpha-d$,
\begin{equation*}
	|\beta_j|+\frac{d-\alpha}{2}\ge0,
	\qquad
	\sum_{j=1}^3
	\left(|\beta_j|+\frac{d-\alpha}{2}\right)
	\le n+\frac{3(d-\alpha)}2
	<
	\alpha+n-\delta+\frac{d-\alpha}{2}.
\end{equation*}
We may therefore choose $0<\theta_j<1$ such that
\begin{equation*}
	\sum_{j=1}^3\theta_j=1,
	\qquad
	|\beta_j|+\frac{d-\alpha}{2}
	<
	\theta_j
	\left(\alpha+n-\delta+\frac{d-\alpha}{2}\right).
\end{equation*}
Set $q_j=2/\theta_j$. The preceding inequalities are equivalent to
\begin{equation*}
	|\beta_j|+d\left(\frac12-\frac1{q_j}\right)
	<
	(1-\theta_j)\frac{\alpha}{2}
	+\theta_j(\alpha+n-\delta).
\end{equation*}
Applying H\"older's inequality and \eqref{xvv} with
$s_0=\alpha/2$ and $s_1=\alpha+n-\delta$ give
\begin{equation*}
	\bigl\|
	(\partial^{\beta_1}u)
	(\partial^{\beta_2}\bar u)
	(\partial^{\beta_3}u)
	\bigr\|_2
	\le
	\prod_{j=1}^3
	\|\partial^{\beta_j}u\|_{L^{q_j}}
	\lesssim
	\prod_{j=1}^3
	\left(
	\|u\|_{H^{\alpha/2}}^{1-\theta_j}
	\|u\|_{H^{\alpha+n-\delta}}^{\theta_j}
	\right)
	=
	\|u\|_{H^{\alpha/2}}^2
	\|u\|_{H^{\alpha+n-\delta}}.
\end{equation*}
Summing over the Leibniz expansion and $|\beta|\le n$ proves \eqref{low cubic product estimate}.
	
For \eqref{low density product estimate}, expand
$\partial^\beta(\bar u|D|^\alpha u)$ for $|\beta|\le n$.
The term with no derivative on $\bar u$ satisfies
\begin{equation*}
	\|\bar u\,|D|^\alpha\partial^\beta u\|_2
	\lesssim \|u\|_\infty\|u\|_{H^{\alpha+n}}.
\end{equation*}
For each remaining term, $\beta_1+\beta_2=\beta$ and
$|\beta_1|>0$. Set
\begin{equation*}
	\theta=\frac{|\beta_1|}{\alpha+|\beta|}\in(0,1),
	\qquad
	1-\theta=\frac{\alpha+|\beta_2|}{\alpha+|\beta|}.
\end{equation*}
H\"older's inequality and \eqref{xv} ensure
\begin{equation*}
	\|(\partial^{\beta_1}\bar u)
	(|D|^\alpha\partial^{\beta_2}u)\|_2
	\le
	\|\partial^{\beta_1}\bar u\|_{L^{2/\theta}}
	\||D|^\alpha\partial^{\beta_2}u\|_{L^{2/(1-\theta)}}
\end{equation*}
\begin{equation*}
	\lesssim
	\bigl(\|u\|_\infty^{1-\theta}
	\|u\|_{H^{\alpha+|\beta|}}^\theta\bigr)
	\bigl(\|u\|_\infty^\theta
	\|u\|_{H^{\alpha+|\beta|}}^{1-\theta}\bigr)
	\le \|u\|_\infty\|u\|_{H^{\alpha+n}}.
\end{equation*}
Summing proves \eqref{low density product estimate}.
	
	For \eqref{low quintic product estimate}, choose
	$0<\theta<1$ such that
	\begin{equation*}
		n+\frac{d-\alpha}{2}
		<
		\theta\left(n+\frac{\alpha+d}{2}-\delta\right),
		\qquad
		d-\alpha
		<
		(1-\theta)
		\left(n+\frac{\alpha+d}{2}-\delta\right).
	\end{equation*}
	Such a choice is possible precisely because
	$\delta<2\alpha-d$.
	Put $p_1=\frac2\theta$ and $p_2=\frac4{1-\theta}$.
	Then $\frac1{p_1}+\frac2{p_2}=\frac12$, and
	\eqref{xvv} yields
	\begin{equation}\label{cb}
		\|u\|_{W^{n,p_1}}\|u\|_{L^{p_2}}^2
		\lesssim
		\|u\|_{H^{\frac{\alpha}{2}}}^2
		\|u\|_{H^{\alpha+n-\delta}}.
	\end{equation}
	Expanding the product on the left-hand side of
	\eqref{low quintic product estimate} gives terms
	$\prod_{j=1}^5\partial^{\beta_j}u_j$, where
	$\sum_{j=1}^5|\beta_j|=n$ and each $u_j$ is $u$ or $\bar u$.
	For $n\ge1$, choose
	\begin{equation*}
		\frac1{q_j}
		=
		\frac{|\beta_j|}{np_1}
		+
		\left(1-\frac{|\beta_j|}{n}\right)\frac1{2p_2},
	\end{equation*}
then the Gagliardo--Nirenberg inequality gives
\begin{equation*}
	\|\partial^{\beta_j}u_j\|_{L^{q_j}}
	\lesssim
	\|u\|_{W^{n,p_1}}^{|\beta_j|/n}
	\|u\|_{L^{2p_2}}^{1-|\beta_j|/n}.
\end{equation*}
Therefore, by H\"older's inequality and (\ref{cb}),
\begin{equation*}
	\left\|\prod_{j=1}^5\partial^{\beta_j}u_j\right\|_2
	\le
	\prod_{j=1}^5\|\partial^{\beta_j}u_j\|_{L^{q_j}}
	\lesssim
	\|u\|_{W^{n,p_1}}\|u\|_{L^{2p_2}}^4,
\end{equation*}
\begin{equation*}
	\lesssim
	\|u\|_\infty^2
	\|u\|_{W^{n,p_1}}\|u\|_{L^{p_2}}^2
	\lesssim
	\|u\|_\infty^2
	\|u\|_{H^{\alpha/2}}^2
	\|u\|_{H^{\alpha+n-\delta}}.
\end{equation*}
	For $n=0$, this bound follows directly from H\"older's
	inequality. Summing proves \eqref{low quintic product estimate}.
	
	For \eqref{low bilinear product estimate}, the two derivative
	orders are $n-|\rho|$ and $\alpha+|\rho|$.
	Since $|\rho|<n$ and $r<\min\{1,\alpha/2\}$, both orders exceed $r$.
	Moreover,
	\begin{equation*}
		(n-|\rho|-r)+(\alpha+|\rho|-r)
		=\alpha+n-2r
		<\alpha+n-\delta-r.
	\end{equation*}
	We may therefore choose $0<\theta_1,\theta_2<1$, $\theta_1+\theta_2=1$, such that
	\begin{equation*}
		n-|\rho|<(1-\theta_1)r+\theta_1(\alpha+n-\delta),
		\qquad
		\alpha+|\rho|<(1-\theta_2)r+\theta_2(\alpha+n-\delta).
	\end{equation*}
	Applying \eqref{xv} and H\"older's inequality gives
	\begin{equation*}
		\|(\partial^{\beta-\rho}\bar u)
		(\partial^\rho|D|^\alpha u)\|_2
		\le
		\|\partial^{\beta-\rho}\bar u\|_{L^{2/\theta_1}}
		\|\partial^\rho|D|^\alpha u\|_{L^{2/\theta_2}}
	\end{equation*}
	\begin{equation*}
		\lesssim
		\prod_{j=1}^2
		\left(
		\|u\|_{W^{r,\infty}}^{1-\theta_j}
		\|u\|_{H^{\alpha+n-\delta}}^{\theta_j}
		\right)
		=
		\|u\|_{W^{r,\infty}}
		\|u\|_{H^{\alpha+n-\delta}}.
	\end{equation*}
	This proves \eqref{low bilinear product estimate}.
	Similarly, expanding by the Leibniz rule and applying
	\eqref{xv} and H\"older's inequality to the resulting
	three-factor products proves \eqref{low product estimate}.
	
\end{proof}

We first establish global existence and the spacetime bound
needed below. By \eqref{new parameters},
\begin{equation*}
	\lim_{p\downarrow1}\gamma_{\alpha,p}
	=\max\left\{d-\alpha,\frac{d^2}{d+2}\right\}
	=\frac{d^2}{d+2}<\alpha.
\end{equation*}
Choose $p>1$ sufficiently close to $1$ that
$\gamma_{\alpha,p}<\alpha$.
Theorem~\ref{local well-posedness} and conservation of mass
and energy give a uniform local existence time at regularity
$\frac{\alpha}{2}$. Persistence of regularity therefore
extends the smooth solution globally.

Fix
\begin{equation*}
	0<\delta<r<
	\min\left\{1,\frac{\alpha-\gamma_{\alpha,p}}2\right\},
\end{equation*}
and choose
$\frac{\gamma_{\alpha,p}}2<\gamma_0<\frac{\alpha}{2}-r$.
Proposition~\ref{spacetime polynomial control}, with
$\gamma=\frac{\alpha}{2}$ and $A=0$, gives
\begin{equation*}
	\|u\|_{L^{2p}([0,T];W^{r,\infty})}
	\lesssim_{u_0}\langle T\rangle^{\frac1{2p}}.
\end{equation*}
Here we used
$W^{\frac{\alpha}{2}-\gamma_0,\infty}
\hookrightarrow W^{r,\infty}$.
Consequently,
\begin{equation}\label{low integrated spatial regularity}
	\int_0^T\|u(t)\|_{W^{r,\infty}}^2\,dt
	\lesssim_{u_0}\langle T\rangle.
\end{equation}
Since
$\gamma_{\alpha,p}\ge\frac{d^2}{d+2}\ge\frac d2$,
these choices satisfy all the assumptions of
Lemma~\ref{low product bounds}.

Fix an integer $n\ge0$ and put
\begin{equation*}
	\theta_{\alpha,n}:=\frac{2\delta}{2n+\alpha}.
\end{equation*}
We next show that $\mathcal E_{\alpha,n}$ controls
$\|u\|_{H^{\alpha+n}}^2$.
By \eqref{low cubic product estimate}, conservation and
interpolation,
\begin{equation*}
	|J_1(u)|
	\lesssim_{u_0}
	\|u\|_{H^{\alpha+n}}
	\|u\|_{H^{\alpha+n-\delta}}
	\lesssim_{u_0}
	\|u\|_{H^{\alpha+n}}^{2-\theta_{\alpha,n}}.
\end{equation*}
Choose
$\frac d2<s_\infty<\alpha-\frac{\delta}{2}$,
which is possible because $\delta<2\alpha-d$.
Lemma~\ref{Sobolev product estimate}, Sobolev embedding
and interpolation give
\begin{equation*}
	|J_2(u)|
	\lesssim
	\|u\|_\infty^2
	\|u\|_{H^{\frac{\alpha}{2}+n}}^2
	\lesssim_{u_0}
	\|u\|_{H^{\alpha+n}}^{
		2-\frac{2(\alpha-s_\infty)}{n+\frac{\alpha}{2}}}
	\lesssim_{u_0}
	\bigl(1+\|u\|_{H^{\alpha+n}}\bigr)^{
		2-\theta_{\alpha,n}}.
\end{equation*}
Thus
\begin{equation}\label{low higher energy difference}
	|J_1(u)|+|J_2(u)|
	\lesssim_{u_0}
	\bigl(1+\|u\|_{H^{\alpha+n}}\bigr)^{
		2-\theta_{\alpha,n}}.
\end{equation}
Since $\theta_{\alpha,n}>0$, there exists a constant
$K_{\alpha,n}>0$ such that
\begin{equation}\label{low energy equivalence}
	K_{\alpha,n}+\mathcal E_{\alpha,n}(u)
	\sim_{u_0}1+\|u\|_{H^{\alpha+n}}^2,
	\qquad
	K_{\alpha,n}+\mathcal E_{\alpha,n}(u)\ge1.
\end{equation}

It remains to estimate the time derivative of
$\mathcal E_{\alpha,n}$.
We use \eqref{qq} and the cancellation leading to
\eqref{Leibniz structure}.

For $I_\beta$, substituting
$\dot u=-i|D|^\alpha u-i|u|^2u$ gives
\begin{equation*}
	I_\beta
	=-2\Re\bigl(
	\partial_x^\beta(|u|^2u),
	-i|u|^2\partial_x^\beta|D|^\alpha u
	\bigr).
\end{equation*}
The omitted term is the real part of an imaginary multiple of
$\int_{\T^d}|u|^2|\partial_x^\beta(|u|^2u)|^2\,dx$.
Hence \eqref{low cubic product estimate} and conservation give
\begin{equation*}
	|I_\beta|
	\lesssim_{u_0}
	\|u\|_\infty^2
	\|u\|_{H^{\alpha+n}}
	\|u\|_{H^{\alpha+n-\delta}}.
\end{equation*}

For $II_\beta$, substitute the equation in
$\partial_x^\rho\dot u$ and use
$(|u|^2)^\cdot=2\Im(\bar u|D|^\alpha u)$.
The terms containing $|D|^\alpha u$ are bounded by
\eqref{low product estimate}.
The terms arising from $-i|u|^2u$ are bounded by
\eqref{low quintic product estimate}.
After pairing with $\partial_x^\beta|D|^\alpha u$,
we obtain
\begin{equation*}
	|II_\beta|
	\lesssim_{u_0}
	\|u\|_{W^{r,\infty}}^2
	\|u\|_{H^{\alpha+n}}
	\|u\|_{H^{\alpha+n-\delta}}.
\end{equation*}

For $III_\beta$, \eqref{low density product estimate} implies
\begin{equation}\label{low density derivative}
	\|(|u|^2)^\cdot\|_{H^n}
	\lesssim
	\|u\|_{W^{r,\infty}}
	\|u\|_{H^{\alpha+n}}.
\end{equation}
The additional products in its expansion in the preceding
subsection are bounded by
\eqref{low bilinear product estimate} and
\eqref{low density derivative}.
Their total contribution is therefore bounded by
\begin{equation*}
	\|u\|_{W^{r,\infty}}^2
	\|u\|_{H^{\alpha+n}}
	\|u\|_{H^{\alpha+n-\delta}}.
\end{equation*}
After combining the remaining term with $\frac{d}{dt}J_2(u)$,
we obtain \eqref{Leibniz structure}.
Lemma~\ref{eee} gives
\begin{equation*}
	\|F_\alpha(u)\|_{H^n}
	\lesssim
	\|u\|_{W^{r,\infty}}
	\|u\|_{H^{\alpha+n-\delta}}.
\end{equation*}
Together with \eqref{low density derivative},
this gives the same bound for \eqref{Leibniz structure}.

Combining these estimates yields
\begin{equation*}
	\left|\frac{d}{dt}\mathcal E_{\alpha,n}(u)\right|
	\lesssim_{u_0}
	\|u\|_{W^{r,\infty}}^2
	\|u\|_{H^{\alpha+n}}
	\|u\|_{H^{\alpha+n-\delta}}.
\end{equation*}
By interpolation and conservation,
\begin{equation*}
	\|u\|_{H^{\alpha+n-\delta}}
	\le
	\|u\|_{H^{\frac{\alpha}{2}}}^{\theta_{\alpha,n}}
	\|u\|_{H^{\alpha+n}}^{1-\theta_{\alpha,n}}
	\lesssim_{u_0}
	\|u\|_{H^{\alpha+n}}^{1-\theta_{\alpha,n}}.
\end{equation*}
Using \eqref{low energy equivalence}, we conclude that
\begin{equation*}
	\left|\frac{d}{dt}
	\bigl(K_{\alpha,n}+\mathcal E_{\alpha,n}(u)\bigr)\right|
	\lesssim_{u_0}
	\|u\|_{W^{r,\infty}}^2
	\bigl(K_{\alpha,n}+\mathcal E_{\alpha,n}(u)\bigr)^{
		1-\frac{\theta_{\alpha,n}}2}.
\end{equation*}
Apply Lemma~\ref{Variant2}, using
\eqref{low integrated spatial regularity}, with
$p_1=1$, $A_1=1$ and
$\lambda_1=\frac{\theta_{\alpha,n}}2$.
We obtain
\begin{equation*}
	\|u(t)\|_{H^{\alpha+n}}
	\lesssim_{u_0,\delta}
	\langle t\rangle^{\frac1{\theta_{\alpha,n}}}
	=
	\langle t\rangle^{\frac{2n+\alpha}{2\delta}},
	\qquad t\ge0.
\end{equation*}
In particular,
\begin{equation*}
	\|u(t)\|_{H^\alpha}
	\lesssim_{u_0,\delta}
	\langle t\rangle^{\frac{\alpha}{2\delta}}.
\end{equation*}
Letting $p\downarrow1$, we may take $\delta$ arbitrarily
close to
$\min\{1,\frac{\alpha-d^2/(d+2)}2\}
=\frac{\kappa_\alpha}{2}$, which yields \eqref{low polynomial growth}.

\vspace{10pt}
\section{Local well-posedness and polynomial growth for general algebraic nonlinearities}
\label{general nonlinearity growth}

We consider
\begin{equation}\label{general FNLS}
	i\partial_tu=|D|^\alpha u+|u|^{2\sigma}u,
	\qquad
	u(0)=u_0,
	\qquad x\in\T^d,
\end{equation}
where $d\ge2$ and $\sigma\ge1$ is an integer.
The local argument extends by replacing the cubic product with a
product of $2\sigma+1$ factors. For polynomial growth, the main
difference is the last term in the modified energy: the squared
norm used for the cubic equation must be replaced by a pairing
between $|u|^2$ and $|u|^{2\sigma}$.

\vspace{7pt}
\subsection{Local well-posedness}

\begin{thm}\label{general local well-posedness}
	Let $\alpha>1$, $\sigma<p<\infty$, and
	\begin{equation*}
		\gamma>\frac{\gamma_{\alpha,p}}{2},
		\qquad
		\frac12<b<1-\frac{\sigma}{2p}.
	\end{equation*}
	For every $u_0\in H^\gamma(\T^d)$, equation
	\eqref{general FNLS} has a unique local solution in
	$X_\alpha^{\gamma,b}([-T,T])\subseteq C([-T,T];H^\gamma)$. If $u_0\in H^s$ with $s\ge\gamma$, the solution belongs
	to $X_\alpha^{s,b}([-T,T])$ on an
	interval depending only on $\|u_0\|_{H^\gamma}$.
	
	In particular, a sufficient regularity condition is
	\begin{equation*}
		\gamma>
		\max\left\{
		\frac d2-\frac{\alpha}{2\sigma},
		\frac{d^2\sigma}{2(d\sigma+2)}
		\right\}.
	\end{equation*}
	For $0<\alpha<1$, the same assertions hold with
	$\gamma_{\alpha,p}$ replaced by $\ell_{\alpha,p}$;
	a sufficient condition is then
	$\gamma>\frac d2-\frac{\alpha}{2\sigma}$.
\end{thm}

\begin{proof}
	Choose
	\begin{equation*}
		\frac{\sigma}{2p}<b'<1-b,
		\qquad c:=1-b-b'>0.
	\end{equation*}
	Modifying the proof of Proposition~\ref{multi} gives, for
	$s\ge\gamma$,
	\begin{equation}\label{general multilinear estimate}
		\left\|
		u_1\overline{u_2}\cdots
		u_{2\sigma-1}\overline{u_{2\sigma}}u_{2\sigma+1}
		\right\|_{X_\alpha^{s,-b'}}
		\lesssim
		\sum_{j=1}^{2\sigma+1}
		\|u_j\|_{X_\alpha^{s,b}}
		\prod_{\ell\ne j}\|u_\ell\|_{X_\alpha^{\gamma,b}}.
	\end{equation}
	
	Use the map \eqref{Phi} with $|u|^2u$ replaced by
	$|u|^{2\sigma}u$.
	Lemma~\ref{Gini} and \eqref{general multilinear estimate}
	give
	\begin{equation*}
		\|\Phi(u)\|_{X_\alpha^{\gamma,b}}
		\lesssim
		\|u_0\|_{H^\gamma}
		+T^c\|u\|_{X_\alpha^{\gamma,b}}^{2\sigma+1},
	\end{equation*}
	and
	\begin{equation*}
		\|\Phi(u)-\Phi(v)\|_{X_\alpha^{\gamma,b}}
		\lesssim
		T^c
		\bigl(\|u\|_{X_\alpha^{\gamma,b}}
		+\|v\|_{X_\alpha^{\gamma,b}}\bigr)^{2\sigma}
		\|u-v\|_{X_\alpha^{\gamma,b}}.
	\end{equation*}
	Thus the contraction argument applies with
	\begin{equation*}
		T\sim
		\bigl(1+\|u_0\|_{H^\gamma}\bigr)^{-\frac{2\sigma}{c}},
	\end{equation*}
	with a sufficiently small fixed multiplicative constant.
	At higher regularity,
	\begin{equation*}
		\|\Phi(u)\|_{X_\alpha^{s,b}}
		\lesssim
		\|u_0\|_{H^s}
		+T^c\|u\|_{X_\alpha^{\gamma,b}}^{2\sigma}
		\|u\|_{X_\alpha^{s,b}}.
	\end{equation*}
	The completeness, persistence, uniqueness, and difference
	arguments in the proof of
	Theorem~\ref{local well-posedness} therefore apply on the
	same interval. Finally, let $p\downarrow\sigma$ to obtain
	the stated sufficient regularities.
	The proof for $0<\alpha<1$ uses its corresponding linear
	estimate.
	
\end{proof}

\vspace{7pt}
\subsection{Polynomial growth estimates}
Define
\begin{equation*}
	\alpha_\sigma:=\frac{d^2\sigma}{d\sigma+2},
	\qquad
	\kappa_{\alpha,\sigma}:=\min\{2,\alpha-\alpha_\sigma\}.
\end{equation*}
Then the corresponding results for equation \eqref{general FNLS} can be stated as follows:
\begin{thm}\label{general polynomial growth}
	Let $d\ge2$, let $\sigma\ge1$ be an integer, and let
	$u_0\in C^\infty(\T^d)$.
	If $\alpha>\alpha_\sigma$, equation
	\eqref{general FNLS} has a unique global smooth solution.
	
	If $\alpha>d$, then, for every integer $n\ge0$
	and every $\varepsilon>0$,
	\begin{equation}\label{general high growth}
		\|u(t)\|_{H^{\alpha+n}}
		\lesssim_{u_0,\varepsilon}
		\langle t\rangle^{
			\frac{\alpha+2n}{\min\{2,\alpha-d\}}
			+\varepsilon}, \quad \forall t\in \R.
	\end{equation}
	Moreover, throughout the range $d\ge \alpha>\alpha_\sigma$,
	for every integer $n\ge0$ and every $\varepsilon>0$,
	\begin{equation}\label{general higher growth}
		\|u(t)\|_{H^{\alpha+n}}
		\lesssim_{u_0,\varepsilon}
		\langle t\rangle^{
			\frac{\alpha+2n}{\kappa_{\alpha,\sigma}}
			+\varepsilon}, \quad \forall t\in \R.
	\end{equation}
\end{thm}
\begin{rem}
	The argument yielding \eqref{general higher growth} (after simple modification) still applies throughout $\alpha>\alpha_\sigma$.
\end{rem}

\vspace{7pt}
\subsection{Proof of the polynomial growth estimates}
Mass and energy conservation give
\begin{equation}\label{general conserved bounds}
	\|u(t)\|_{H^{\frac{\alpha}{2}}}
	+\|u(t)\|_{L^{2\sigma+2}}
	\lesssim_{u_0}1.
\end{equation}
We retain $F_\alpha$ from Lemma~\ref{eee} and define
\begin{equation}\label{general G definition}
	G_{\alpha,\sigma}(u)
	:=|D|^\alpha(\rho^\sigma)
	-\sigma\rho^{\sigma-1}|D|^\alpha\rho,
\end{equation}
where $\rho:=|u|^2$. In particular, $G_{\alpha,1}=0$.

For every integer $n\ge0$, define
\begin{equation*}
	\mathcal E_{\alpha,n}^{(\sigma)}(u)
	:=\|u\|_2^2+\||D|^{\alpha+n}u\|_2^2
	+J_1(u)+J_2(u),
\end{equation*}
where
\begin{equation*}
	J_1(u)
	:=2\Re\bigl(|D|^{\alpha+n}u,|D|^n(\rho^\sigma u)\bigr),
	\qquad
	J_2(u)
	:=-\frac12
	\bigl(|D|^{\alpha+n}\rho,|D|^n\rho^\sigma\bigr).
\end{equation*}
For $\sigma=1$, this is the modified energy used in the
preceding section.

We first collect the bounds on this energy. The estimates
for $I_\beta$, $II_\beta$ and $III_\beta$ follow the cubic
argument; the additional terms arising from
$\rho^{\sigma-1}$ are treated explicitly below.

\begin{lemma}\label{general modified energy bounds}
	Let $n\ge0$ be an integer.
	
	If $\alpha>d$ and
	\begin{equation*}
		0<\eta<\min\left\{1,\frac{\alpha-d}{2}\right\},
	\end{equation*}
	then
	\begin{equation}\label{general high energy derivative}
		\left|\frac{d}{dt}
		\mathcal E_{\alpha,n}^{(\sigma)}(u)\right|
		\lesssim_{u_0}
		\|u\|_{H^{\alpha+n}}
		\|u\|_{H^{\alpha+n-\eta}}.
	\end{equation}
	
	If $\alpha_\sigma<\alpha\le d$ and
	\begin{equation*}
		0<\delta<r<\min\left\{1,\frac{\alpha}{2}\right\},
		\qquad
		\delta<(\sigma+1)\alpha-\sigma d,
	\end{equation*}
	then
	\begin{equation}\label{general low energy derivative}
		\left|\frac{d}{dt}
		\mathcal E_{\alpha,n}^{(\sigma)}(u)\right|
		\lesssim_{u_0}
		\|u\|_{W^{r,\infty}}^{2\sigma}
		\|u\|_{H^{\alpha+n}}
		\|u\|_{H^{\alpha+n-\delta}}.
	\end{equation}
	
	In either case, there exists a sufficiently large
	constant $K>0$ such that
	\begin{equation}\label{general energy equivalence}
		K+\mathcal E_{\alpha,n}^{(\sigma)}(u)
		\sim_{u_0}1+\|u\|_{H^{\alpha+n}}^2.
	\end{equation}
\end{lemma}

\begin{proof}
	We first differentiate the modified energy.
	As in \eqref{qq}, define
	\begin{align*}
		I_\beta
		&:=2\Re\bigl(
		\partial^\beta|D|^\alpha u,
		(\partial^\beta\dot u)\rho^\sigma
		\bigr),\\
		II_\beta
		&:=2\Re\left(
		\partial^\beta|D|^\alpha u,
		\sum_{\substack{\mu\le\beta\\|\mu|<n}}
		\binom{\beta}{\mu}
		\left[
		(\partial^\mu\dot u)
		(\partial^{\beta-\mu}\rho^\sigma)
		+
		(\partial^{\beta-\mu}u)
		(\partial^\mu(\rho^\sigma)^\cdot)
		\right]\right),\\
		III_\beta
		&:=2\Re\bigl(
		\partial^\beta|D|^\alpha u,
		u\partial^\beta(\rho^\sigma)^\cdot
		\bigr).
	\end{align*}
	The calculation preceding \eqref{qq} gives
	\begin{equation*}
		\frac{d}{dt}\mathcal E_{\alpha,n}^{(\sigma)}(u)
		=
		\sum_{|\beta|=n}\frac{n!}{\beta!}
		(I_\beta+II_\beta+III_\beta)
		+\frac{d}{dt}J_2(u).
	\end{equation*}
	From the equation (\ref{general FNLS}), we have the identities
	\begin{equation}\label{general density identities}
		\dot\rho=2\Im(\bar u|D|^\alpha u),
		\qquad
		(\rho^\sigma)^\cdot
		=\sigma\rho^{\sigma-1}\dot\rho.
	\end{equation}
	
	For convenience, we just consider the range $\alpha_\sigma<\alpha\le d$ and the remaining range $\alpha>d$ can be handled similarly.
	
	The cancellation used for $I_\beta$ in the cubic proof
	gives
	\begin{equation*}
		I_\beta
		=-2\Re\bigl(
		\partial^\beta(\rho^\sigma u),
		-i\rho^\sigma\partial^\beta|D|^\alpha u
		\bigr).
	\end{equation*}
	Thus \eqref{general nonlinear product} and
	\eqref{general conserved bounds} imply
	\begin{equation*}
		|I_\beta|
		\lesssim_{u_0}
		\|u\|_\infty^{2\sigma}
		\|u\|_{H^{\alpha+n}}
		\|u\|_{H^{\alpha+n-\delta}}.
	\end{equation*}
In $II_\beta$, the nonlinear part of $\dot u$ is
controlled by \eqref{general higher nonlinear product}.
The remaining products contain $2\sigma+1$ factors,
have total derivative order $\alpha+n$, and have at
most $\alpha+n-1$ derivatives on any individual factor.
 Since $r<\min\{1,\alpha/2\}$, their derivative orders
$a_j$ satisfy
\begin{equation*}
	\sum_{j=1}^{2\sigma+1}(a_j-r)_+
	\le\alpha+n-2r<\alpha+n-\delta-r.
\end{equation*}
Choose $0<\theta_j<1$ such that
\begin{equation*}
	\sum_{j=1}^{2\sigma+1}\theta_j=1,
	\qquad
	a_j<(1-\theta_j)r+\theta_j(\alpha+n-\delta).
\end{equation*}
H\"older's inequality with exponents $2/\theta_j$
and \eqref{xv} bound each product in $L^2$ by
\begin{equation}\label{aza}
	\prod_{j=1}^{2\sigma+1}
	\left(
	\|u\|_{W^{r,\infty}}^{1-\theta_j}
	\|u\|_{H^{\alpha+n-\delta}}^{\theta_j}
	\right)
	=
	\|u\|_{W^{r,\infty}}^{2\sigma}
	\|u\|_{H^{\alpha+n-\delta}}.
\end{equation}
Pairing with $\partial^\beta|D|^\alpha u$ therefore gives
\begin{equation*}
	|II_\beta|
	\lesssim_{u_0}
	\|u\|_{W^{r,\infty}}^{2\sigma}
	\|u\|_{H^{\alpha+n}}
	\|u\|_{H^{\alpha+n-\delta}}.
\end{equation*}

	For $III_\beta$, the proof of
	\eqref{low density product estimate}, applied to
	\eqref{general density identities}, gives
	\begin{equation}\label{general density estimates}
		\|\dot\rho\|_{H^n}
		\lesssim
		\|u\|_{W^{r,\infty}}\|u\|_{H^{\alpha+n}},
		\qquad
		\|(\rho^\sigma)^\cdot\|_{H^n}
		\lesssim
		\|u\|_{W^{r,\infty}}^{2\sigma-1}
		\|u\|_{H^{\alpha+n}}.
	\end{equation}
Moving $\bar u$ inside $\partial^\beta$ in $III_\beta$,
as in the calculation leading to \eqref{Leibniz structure},
we have
\begin{equation*}
	\bar u\,\partial^\beta|D|^\alpha u
	=
	\partial^\beta(\bar u|D|^\alpha u)
	-
	\sum_{0<\nu\le\beta}\binom{\beta}{\nu}
	(\partial^\nu\bar u)
	(\partial^{\beta-\nu}|D|^\alpha u).
\end{equation*}
Consequently,
\begin{equation*}
	III_\beta
	=
	2\Re\bigl(
	\partial^\beta(\bar u|D|^\alpha u),
	\partial^\beta(\rho^\sigma)^\cdot
	\bigr)
	-
	2\Re\sum_{0<\nu\le\beta}\binom{\beta}{\nu}
	\bigl(
	(\partial^\nu\bar u)
	(\partial^{\beta-\nu}|D|^\alpha u),
	\partial^\beta(\rho^\sigma)^\cdot
	\bigr).
\end{equation*}
Since $|\beta-\nu|<n$, each summand in the second term is
controlled by \eqref{low bilinear product estimate}
and \eqref{general density estimates}.
Their contribution is bounded by
\begin{equation*}
	\|u\|_{W^{r,\infty}}^{2\sigma}
	\|u\|_{H^{\alpha+n}}
	\|u\|_{H^{\alpha+n-\delta}}.
\end{equation*}
Since $2\Re(\bar u|D|^\alpha u)=|D|^\alpha\rho+F_\alpha(u)$,
summing the remaining terms of $III_\beta$ gives
\begin{equation*}
	2\Re\sum_{|\beta|=n}\frac{n!}{\beta!}
	\bigl(
	\partial^\beta(\bar u|D|^\alpha u),
	\partial^\beta(\rho^\sigma)^\cdot
	\bigr)
	=
	\bigl(
	|D|^n(|D|^\alpha\rho+F_\alpha(u)),
	|D|^n(\rho^\sigma)^\cdot
	\bigr).
\end{equation*}
On the other hand, differentiating $J_2(u)$ and using
self-adjointness gives
\begin{equation*}
	\frac{d}{dt}J_2(u)
	=
	-\frac12
	\bigl(
	|D|^n|D|^\alpha(\rho^\sigma),|D|^n\dot\rho
	\bigr)
	-\frac12
	\bigl(
	|D|^{\alpha+n}\rho,|D|^n(\rho^\sigma)^\cdot
	\bigr).
\end{equation*}
Combining these identities and substituting
\begin{equation*}
	(\rho^\sigma)^\cdot=\sigma\rho^{\sigma-1}\dot\rho,
	\qquad
	|D|^\alpha(\rho^\sigma)
	=G_{\alpha,\sigma}(u)
	+\sigma\rho^{\sigma-1}|D|^\alpha\rho
\end{equation*}
from \eqref{general density identities} and
\eqref{general G definition}, we obtain
\begin{equation}\label{za}
	\bigl(|D|^nF_\alpha(u),|D|^n(\rho^\sigma)^\cdot\bigr)
	-\frac12
	\bigl(|D|^nG_{\alpha,\sigma}(u),|D|^n\dot\rho\bigr)
\end{equation}
plus
\begin{equation}\label{az}
	\frac{\sigma}{2}
	\left[
	\bigl(
	|D|^{\alpha+n}\rho,
	|D|^n(\rho^{\sigma-1}\dot\rho)
	\bigr)
	-
	\bigl(
	|D|^n(\rho^{\sigma-1}|D|^\alpha\rho),
	|D|^n\dot\rho
	\bigr)
	\right].
\end{equation}
The term (\ref{az}) equals 
\begin{equation*}
	\frac{\sigma}{2}
	\sum_{|\beta|=n}\frac{n!}{\beta!}
	\left[
	\bigl(
	\partial^\beta|D|^\alpha\rho,
	\partial^\beta(\rho^{\sigma-1}\dot\rho)
	\bigr)
	-
	\bigl(
	\partial^\beta(\rho^{\sigma-1}|D|^\alpha\rho),
	\partial^\beta\dot\rho
	\bigr)
	\right].
\end{equation*}
By the Leibniz rule,
\begin{equation*}
	\partial^\beta(\rho^{\sigma-1}\dot\rho)
	=
	\sum_{\nu\le\beta}\binom{\beta}{\nu}
	(\partial^\nu\rho^{\sigma-1})
	(\partial^{\beta-\nu}\dot\rho),
\end{equation*}
and
\begin{equation*}
	\partial^\beta(\rho^{\sigma-1}|D|^\alpha\rho)
	=
	\sum_{\nu\le\beta}\binom{\beta}{\nu}
	(\partial^\nu\rho^{\sigma-1})
	(\partial^{\beta-\nu}|D|^\alpha\rho).
\end{equation*}
The terms with $\nu=0$ cancel, so (\ref{az}) equals
\begin{equation}\label{general additional derivative terms}
	\frac{\sigma}{2}
	\sum_{|\beta|=n}\frac{n!}{\beta!}
	\sum_{0<\nu\le\beta}\binom{\beta}{\nu}
	\Big[
	\bigl(
	\partial^\beta|D|^\alpha\rho,
	(\partial^\nu\rho^{\sigma-1})
	(\partial^{\beta-\nu}\dot\rho)
	\bigr)
	-
	\bigl(
	(\partial^\nu\rho^{\sigma-1})
	(\partial^{\beta-\nu}|D|^\alpha\rho),
	\partial^\beta\dot\rho
	\bigr)
	\Big].
\end{equation}

	Lemma~\ref{eee}, \eqref{general G spatial estimate}
	and \eqref{general density estimates} bound the two
	pairings in (\ref{za}).
	For \eqref{general additional derivative terms},
	a similar argument as in (\ref{aza}) gives
    \begin{equation*}
	\bigl\|
	(\partial^\nu\rho^{\sigma-1})
	(\partial^{\beta-\nu}\dot\rho)
	\bigr\|_2
	+
	\bigl\|
	(\partial^\nu\rho^{\sigma-1})
	(\partial^{\beta-\nu}|D|^\alpha\rho)
	\bigr\|_2
	\lesssim
	\|u\|_{W^{r,\infty}}^{2\sigma-1}
	\|u\|_{H^{\alpha+n-\delta}}.
    \end{equation*}
	On the other side of each pairing,
	Lemma~\ref{Sobolev product estimate} and
	\eqref{general density estimates} give
	\begin{equation*}
		\|\partial^\beta|D|^\alpha\rho\|_2
		+\|\partial^\beta\dot\rho\|_2
		\lesssim
		\|u\|_{W^{r,\infty}}
		\|u\|_{H^{\alpha+n}}.
	\end{equation*}
	Thus all these terms satisfy the bound in
	\eqref{general low energy derivative}.
	This proves \eqref{general low energy derivative}.

	Finally, we prove \eqref{general energy equivalence}. Note that \eqref{general nonlinear product} and conservation give
	\begin{equation*}
		|J_1(u)|
		\lesssim_{u_0}
		\|u\|_{H^{\alpha+n}}
		\|u\|_{H^{\alpha+n-\delta}}
		\lesssim_{u_0}
		\|u\|_{H^{\alpha+n}}^{
			2-\frac{2\delta}{\alpha+2n}}.
	\end{equation*}
	Choose
	\begin{equation*}
		\frac d2<s_\infty<
		\frac{(\sigma+1)\alpha-\delta}{2\sigma},
	\end{equation*}
	which is possible by the assumption on $\delta$.
	The standard product estimate and the interpolation
	used for \eqref{low higher energy difference} yield
	\begin{equation*}
		|J_2(u)|
		\lesssim
		\|u\|_\infty^{2\sigma}
		\|u\|_{H^{\frac{\alpha}{2}+n}}^2
		\lesssim_{u_0}
		\|u\|_{H^{\alpha+n}}^{
			2-\frac{(\sigma+1)\alpha-2\sigma s_\infty}
			{n+\frac{\alpha}{2}}}
		\lesssim
		\bigl(1+\|u\|_{H^{\alpha+n}}\bigr)^{
			2-\frac{2\delta}{\alpha+2n}}.
	\end{equation*}
	All correction terms have powers strictly smaller
	than $2$, which proves
	\eqref{general energy equivalence}.
	
\end{proof}

\begin{proof}[Proof of Theorem~\ref{general polynomial growth}]
	We first establish global existence. Since
	\begin{equation*}
		\alpha_\sigma
		=\frac{d^2\sigma}{d\sigma+2}
		\ge\frac{d\sigma}{\sigma+1},
	\end{equation*}
	for $\alpha>\alpha_\sigma$ we have
	\begin{equation*}
		\lim_{p\downarrow\sigma}\gamma_{\alpha,p}
		=
		\max\left\{
		d-\frac{\alpha}{\sigma},
		\alpha_\sigma
		\right\}
		=\alpha_\sigma<\alpha.
	\end{equation*}
	Theorem~\ref{general local well-posedness} therefore
	applies at regularity $\frac{\alpha}{2}$.
	The conserved bound \eqref{general conserved bounds}
	gives a uniform local lifespan, and persistence of
	regularity extends the smooth solution globally.
	
	\medskip
	\noindent
	\textit{The case $\alpha>d$.}
	
	Fix $n\ge0$ and
	$0<\eta<\min\{1,\frac{\alpha-d}{2}\}$.
	Interpolation and conservation give
	\begin{equation*}
		\|u\|_{H^{\alpha+n-\eta}}
		\lesssim_{u_0}
		\|u\|_{H^{\alpha+n}}^{
			1-\frac{2\eta}{\alpha+2n}}.
	\end{equation*}
	By \eqref{general high energy derivative} and
	\eqref{general energy equivalence},
	\begin{equation*}
		\left|\frac{d}{dt}
		\bigl(K+\mathcal E_{\alpha,n}^{(\sigma)}(u)\bigr)
		\right|
		\lesssim_{u_0}
		\bigl(K+\mathcal E_{\alpha,n}^{(\sigma)}(u)\bigr)^{
			1-\frac{\eta}{\alpha+2n}}.
	\end{equation*}
	As in the cubic case, Lemma~\ref{Variant2} yields
	\begin{equation*}
		\|u(t)\|_{H^{\alpha+n}}
		\lesssim_{u_0,\eta}
		\langle t\rangle^{\frac{\alpha+2n}{2\eta}}.
	\end{equation*}
	Letting $\eta$ approach
	$\min\{1,\frac{\alpha-d}{2}\}$ proves
	\eqref{general high growth}.
	
	\medskip
	\noindent
	\textit{The case $\alpha_\sigma<\alpha\le d$.}
	
	Choose $p>\sigma$ sufficiently close to $\sigma$ that
	$\gamma_{\alpha,p}<\alpha$, and fix
	\begin{equation*}
		0<\delta<r<
		\min\left\{1,\frac{\alpha-\gamma_{\alpha,p}}2\right\}.
	\end{equation*}
	These choices satisfy the assumptions of
	Lemma~\ref{general modified energy bounds}.
	Indeed, $r<\frac{\alpha}{2}$ and
	\begin{equation*}
		\delta<\frac{\alpha-\alpha_\sigma}{2}
		<(\sigma+1)\alpha-\sigma d.
	\end{equation*}
	
	The local spacetime argument of
	Proposition~\ref{spacetime polynomial control},
	now using Theorem~\ref{general local well-posedness},
	gives
	\begin{equation*}
		\|u\|_{L^{2p}([0,T];W^{r,\infty})}
		\lesssim_{u_0}\langle T\rangle^{\frac1{2p}}.
	\end{equation*}
	More precisely, choose
	$\frac{\gamma_{\alpha,p}}2<\gamma_0<
	\frac{\alpha}{2}-r$.
	The uniform local solution bound at regularity
	$\frac{\alpha}{2}$ and the linear transfer estimate
	control the $L^{2p}W^{r,\infty}$ norm on each local
	interval. Summing over these intervals and applying
	H\"older's inequality, with $p>\sigma$, gives
	\begin{equation}\label{general integrated spatial estimate}
		\int_0^T
		\|u(t)\|_{W^{r,\infty}}^{2\sigma}\,dt
		\lesssim_{u_0}\langle T\rangle.
	\end{equation}
	
	For every integer $n\ge0$, interpolation gives
	\begin{equation*}
		\|u\|_{H^{\alpha+n-\delta}}
		\lesssim_{u_0}
		\|u\|_{H^{\alpha+n}}^{
			1-\frac{2\delta}{\alpha+2n}}.
	\end{equation*}
	Thus \eqref{general low energy derivative} and
	\eqref{general energy equivalence} imply
	\begin{equation*}
		\left|\frac{d}{dt}
		\bigl(K+\mathcal E_{\alpha,n}^{(\sigma)}(u)\bigr)
		\right|
		\lesssim_{u_0}
		\|u\|_{W^{r,\infty}}^{2\sigma}
		\bigl(K+\mathcal E_{\alpha,n}^{(\sigma)}(u)\bigr)^{
			1-\frac{\delta}{\alpha+2n}}.
	\end{equation*}
	Apply Lemma~\ref{Variant2}, using
	\eqref{general integrated spatial estimate}, with
	$p_1=1$, $A_1=1$ and
	$\lambda_1=\frac{\delta}{\alpha+2n}$.
	We obtain
	\begin{equation*}
		\|u(t)\|_{H^{\alpha+n}}
		\lesssim_{u_0,\delta}
		\langle t\rangle^{\frac{\alpha+2n}{2\delta}},
		\qquad t\ge0.
	\end{equation*}
	Letting $p\downarrow\sigma$, we may take $\delta$
	arbitrarily close to
	\begin{equation*}
		\min\left\{1,\frac{\alpha-\alpha_\sigma}{2}\right\}
		=\frac{\kappa_{\alpha,\sigma}}2.
	\end{equation*}
	This proves \eqref{general higher growth}.
	
\end{proof}

		\vspace{10pt}
		\section{Improved estimates from the circle method}\label{improved section}
		The estimates in this section are additional inputs. Combining those inputs, one can slightly improve the local well-posedness and polynomial growth results. For technical reasons, we exclude the standard Schr\"odinger case
		$\alpha=2$, for which stronger estimates and more refined tools
		are available.

		\vspace{7pt}
		\subsection{The fractional circle-method estimate and its proof}
		We give the arithmetic argument, including its fractional time
		representation. The quadratic sums are classical; compare
		\cite[Lemma 3.2]{GRM24} and \cite{13}. The following proof supplies the
		local average and time subdivision needed for the fractional estimate.

		Recall $U_\alpha(t):=\exp(-it|D|^\alpha)$ and the Littlewood-Paley projection $P_N$ is given by 
		\begin{equation*}
			\widehat{P_Nf}(k)=\chi(|k|^2/N^2)\widehat f(k),
		\end{equation*}
		where $\chi\in C_c^\infty((0,\infty))$ is supported in a compact interval away from zero.

		\begin{thm}\label{cm:thm:main}
			Let $d\ge5$, $\alpha>1$, $\alpha\ne 2$ and $1\le p<\infty$. Define
			\begin{equation}\label{cm:eq:exponent}
				b_{\rm CM}(\alpha,p)=
				\begin{cases}
					\displaystyle\max\left\{\frac d2-\frac{\alpha}{2p},
					\frac d2-\frac{d-2}{p(d-3)+2}\right\},&1\le p\le2,\\[6pt]
					\displaystyle\max\left\{\frac d2-\frac{\alpha}{2p},
					\frac d2-\frac{2}{p+2}\right\},&2\le p<\infty.
				\end{cases}
			\end{equation}
			For every interval $J$ of length at most one,
			\begin{equation}\label{cm:eq:dyadic}
				\|U_\alpha(t)P_Nf\|_{L^{2p}(J;L^\infty(\T^d))}
				\lesssim N^{b_{\rm CM}(\alpha,p)}\|f\|_2.
			\end{equation}
			The constant is independent of $N$ and the position of $J$.
			Consequently, for $s>b_{\rm CM}(\alpha,p)$,
			\begin{equation*}
				\|U_\alpha(t)f\|_{L^{2p}(J;L^\infty)}\lesssim \|f\|_{H^s}.
			\end{equation*}
		\end{thm}
		
		\vspace{7pt}
		\subsubsection*{Quadratic sums and their local averages}
		
		Choose $\phi\in C_c^\infty(\R)$ and define
		\begin{equation*}
			\Phi(k/N):=\prod_{j=1}^d\phi(k_j/N).
		\end{equation*}
		We additionally require that $\phi$ be equal to one on a sufficiently large interval, so that
		$\Phi(k/N)=1$ whenever $\chi(|k|^2/N^2)\ne0$. Write
		\begin{equation*}
			Q_N(v,x)=\sum_{k\in\Z^d}\Phi(k/N)
			\exp(ik\cdot x-2\pi i v|k|^2),\qquad
			H_N(v)=\sup_{x\in\T^d}|Q_N(v,x)|.
		\end{equation*}
		
		First, we need the following $L^\infty$-control:
		\begin{lemma}\label{cm:lem:gauss}
			If $(a,q)=1$, $1\le q\le N$, and $|v-a/q|\le1/(qN)$, then
			\begin{equation}\label{cm:eq:majorarc}
				H_N(v)\le C N^d q^{-d/2}
				(1+N^2|v-a/q|)^{-d/2}.
			\end{equation}
			Here $a=0,q=1$ is allowed.
		\end{lemma}
		\begin{proof}
			Since $Q_N(v,x)$ is a product of one-dimensional sums, it suffices to prove
			\begin{equation}\label{zx}
				\left|\sum_{k\in\Z}\phi(k/N)e^{ikx-2\pi ivk^2}\right|
				\lesssim_\phi
				Nq^{-1/2}(1+N^2|v-a/q|)^{-1/2}.
			\end{equation}
			Put $\beta=v-a/q$. Define
			\begin{equation*}
				G(a,m;q)
				:=\sum_{r\bmod q}e^{\frac{2\pi i}{q}(-ar^2+mr)}.
			\end{equation*}
			Then
			\begin{equation*}
				e^{-2\pi iak^2/q}
				=\frac1q\sum_{m\bmod q}G(a,m;q)e^{-2\pi imk/q}.
			\end{equation*}
			Moreover, setting $h=r-r'$ gives
			\begin{align*}
				|G(a,m;q)|^2
				&=\sum_{h\bmod q}
				e^{2\pi i(-ah^2+mh)/q}
				\sum_{r'\bmod q}e^{-2\pi i(2ah)r'/q}\\
				&=q\sum_{\substack{h\bmod q\\q\mid 2ah}}
				e^{2\pi i(-ah^2+mh)/q}.
			\end{align*}
			Since $(a,q)=1$,
			\begin{equation*}
				|G(a,m;q)|^2
				\le q\,\#\{h\bmod q:q\mid2h\}
				=q\gcd(2,q)\le2q.
			\end{equation*}
			
			 Poisson summation therefore gives
			\begin{align*}
				\sum_{k\in\Z}\phi(k/N)e^{ikx-2\pi ivk^2}
				&=\frac1q\sum_{m\bmod q}G(a,m;q)
				\sum_{k\in\Z}\phi(k/N)
				e^{ik(x-2\pi m/q)-2\pi i\beta k^2}\\
				&=\frac1q\sum_{m\in\Z}G(a,m;q)
				\int_\R\phi(y/N)
				e^{iy(x-2\pi m/q)-2\pi i\beta y^2}\,dy.
			\end{align*}
			Consequently,
			\begin{equation*}
				\left|\sum_{k\in\Z}\phi(k/N)e^{ikx-2\pi ivk^2}\right|
				\lesssim q^{-1/2}\sum_{m\in\Z}|I_m|,
			\end{equation*}
			where
			\begin{equation*}
				I_m
				:=N\int_\R\phi(z)
				e^{2\pi i(N(\frac{x}{2\pi}-\frac{m}{q})z-N^2\beta z^2)}\,dz.
			\end{equation*}
			For convenience, we denote $B:=N^2|\beta|$ and $w_m:=N(\frac{x}{2\pi}-\frac{m}{q})$.
			
			If $B\ge1$, the stationary phase method yields
			$|I_m|\lesssim_\phi NB^{-1/2}$.
			If $B<1$, then $|I_m|\le N\|\phi\|_{L^1}$.
			Thus, for every $m$,
			\begin{equation}\label{aaaa}
				|I_m|\lesssim_\phi N(1+B)^{-1/2}.
			\end{equation}
			
			If $|w_m|\gg B+1$, then for $z\in \supp\; \phi$, we have 
			\begin{equation*}
				|w_m-2N^2\beta z|\gtrsim\frac{|w_m|}{2}.
			\end{equation*}
		
			Integrating by parts $M$ times therefore gives
			\begin{equation}\label{aaaaa}
				|I_m|\lesssim_M N|w_m|^{-M}.
			\end{equation}
			
			Since $|\beta|\le1/(qN)$ and $q\le N$,
			\begin{equation*}
				B\le\frac Nq,
				\qquad
				|w_{m+1}-w_m|=\frac Nq\ge1.
			\end{equation*}
			Hence, we see
			\begin{equation*}
				\#\{m:|w_m|\lesssim B+1\}\lesssim 1.
			\end{equation*}
			Combining (\ref{aaaa}) and (\ref{aaaaa}), one can conclude
			\begin{equation*}
				\sum_{m\in\Z}|I_m|
				\lesssim_\phi N(1+B)^{-1/2}.
			\end{equation*}
			This proves \eqref{zx} and consequently (\ref{cm:eq:majorarc}).
			
		\end{proof}
		
		Next, we control the integral of $H_N$ on the region away from zero:
		\begin{lemma}\label{cm:lem:average}
			Let $d>4$ and fix $0<c<C<\infty$. For $N^{-1}\le L\le1$,
			\begin{equation}\label{cm:eq:local-average}
				\int_{\lbrace v\sim L\rbrace}H_N(v)\,dv
				\lesssim N^{d-2}L^{d/2-1}.
			\end{equation}
			For $L\ge1$,
			\begin{equation}\label{cm:eq:long-average}
				\int_{\lbrace v\sim L\rbrace}H_N(v)\,dv\lesssim N^{d-2}L.
			\end{equation}
		\end{lemma}
		\begin{proof}
			By Dirichlet approximation,
			\begin{equation*}
				\R=
				\bigcup_{q=1}^{N}
				\bigcup_{\substack{a\in\Z\\(a,q)=1}}
				\left[\frac aq-\frac1{qN},\frac aq+\frac1{qN}\right].
			\end{equation*}
			Hence \eqref{cm:eq:majorarc} gives
			\begin{equation*}
				H_N(v)\lesssim
				N^d\sum_{q=1}^{N}q^{-d/2}
				\sum_{\substack{a\in\Z\\(a,q)=1}}
				\frac{\mathbf 1_{\{|v-a/q|\le1/(qN)\}}}
				{(1+N^2|v-a/q|)^{d/2}}.
			\end{equation*}
			For $N^{-1}\le L\le1$,
			\begin{align*}
				v\sim L,\quad
				\left|v-\frac aq\right|\le\frac1{qN},\quad a\ne0
				\quad\Longrightarrow\quad
				1\le a\lesssim qL,\quad
				q\gtrsim\frac1{L}.
			\end{align*}
			Moreover,
			\begin{equation*}
				\#\left\{a\in\Z:1\le a\lesssim qL\right\}
				\lesssim qL,
			\end{equation*}
			and
			\begin{equation*}
				\int_\R
				\frac{dv}{(1+N^2|v-a/q|)^{d/2}}
				=N^{-2}\int_\R(1+|w|)^{-d/2}\,dw
				\lesssim_d N^{-2}.
			\end{equation*}
			Including the remaining case $(a,q)=(0,1)$, it follows that, for $d>4$,
			\begin{equation*}
				\int_{\lbrace v\sim L\rbrace}H_N(v)\,dv
				\lesssim
				N^dL(1+N^2L)^{-d/2}
				+N^{d-2}
				\sum_{q\gtrsim L^{-1}}
				(qL)q^{-d/2}
		    	\lesssim N^{d-2}L^{d/2-1}.
		    \end{equation*}
			This proves \eqref{cm:eq:local-average}.
			
			For the integral over one period,
			\begin{equation*}
				v\in[0,1],\quad
				\left|v-\frac aq\right|\le\frac1{qN}
				\quad\Longrightarrow\quad
				a\in\{0,\ldots,q\}.
			\end{equation*}
			Consequently,
			\begin{equation*}
				\int_0^1H_N(v)\,dv
				\lesssim
				N^d\sum_{q=1}^{N}q^{-d/2}
				\sum_{a=0}^{q}
				\int_\R(1+N^2|v-a/q|)^{-d/2}\,dv
			\end{equation*}
			\begin{equation*}
			\lesssim
			N^{d-2}\sum_{q=1}^{N}q^{1-d/2}\lesssim N^{d-2}.
			\end{equation*}
			where $d>4$.
			Finally, for $L\ge1$, partitioning $\lbrace v\sim L\rbrace$ into unit interval, the periodicity of $H_{N}$ gives \eqref{cm:eq:long-average}.
			
		\end{proof}

		\vspace{10pt}
		\subsubsection*{The fractional kernel}
		In this subsection, we will connect the fractional kernel with the above quadratic sum $Q_N(v,x)$, which enables us to reduce the analysis of the fractional kernel to that of the quadratic sum
		
	The kernel of $U_\alpha(t)P_N$ is
	\begin{equation*}
		K_{\alpha,N}(t,x)=\sum_{k\in\Z^d}
		\chi(|k|^2/N^2)\exp(ik\cdot x-it|k|^\alpha).
	\end{equation*}
	First assume $\alpha>1$, $\alpha\ne2$, and $t>0$.
	Extend $\chi(r/N^2)\exp(-it r^{\alpha/2})$ by zero off the positive
	axis and set
	\begin{equation}\label{cm:eq:amplitude}
		a_{t,N}(v):=\int_\R\chi(r/N^2)
		\exp(-it r^{\alpha/2}+2\pi i vr)\,dr.
	\end{equation}
	Since the extended function belongs to $C_c^\infty(\R)$,
	we have $a_{t,N}\in\mathcal S(\R)$, and Fourier inversion gives
	\begin{equation*}
		\int_\R a_{t,N}(v)e^{-2\pi iv|k|^2}\,dv
		=\chi(|k|^2/N^2)e^{-it|k|^\alpha}.
	\end{equation*}
	Using the definition of $Q_N$ and the identity
	$\Phi(k/N)\chi(|k|^2/N^2)=\chi(|k|^2/N^2)$, we obtain
	\begin{align*}
		\int_\R a_{t,N}(v)Q_N(v,x)\,dv
		&=\sum_{k\in\Z^d}\Phi(k/N)e^{ik\cdot x}
		\int_\R a_{t,N}(v)e^{-2\pi iv|k|^2}\,dv\\
		&=\sum_{k\in\Z^d}\Phi(k/N)\chi(|k|^2/N^2)
		e^{ik\cdot x-it|k|^\alpha}\\
		&=K_{\alpha,N}(t,x).
	\end{align*}
    Thus, we have represented the fractional kernel $K_{\alpha,N}$ by quadratic sum $Q_N(v,x)$: 
	\begin{equation}\label{cm:eq:representation}
		K_{\alpha,N}(t,x)=\int_\R a_{t,N}(v)Q_N(v,x)\,dv.
	\end{equation}
		Write
		\begin{equation*}
			\lambda=tN^\alpha,
			\qquad
			L=tN^{\alpha-2}=\frac{\lambda}{N^2}.
		\end{equation*}
		Changing variables $r=N^2y$ in \eqref{cm:eq:amplitude}, we obtain
		\begin{equation*}
			a_{t,N}(v)
			=N^2\int_0^\infty\chi(y)e^{i\psi_v(y)}\,dy,
			\qquad
			\psi_v(y):=-\lambda y^{\alpha/2}+2\pi N^2vy.
		\end{equation*}
		Its first two derivatives are
		\begin{equation*}
			\psi_v'(y)
			=-\lambda\frac{\alpha}{2}y^{\alpha/2-1}+2\pi N^2v,
			\qquad
			\psi_v''(y)
			=-\lambda\frac{\alpha}{2}
			\left(\frac{\alpha}{2}-1\right)y^{\alpha/2-2}.
		\end{equation*}
		Since $\alpha\ne2$ and the support of $\chi$ is away from zero,
		\begin{equation*}
			|\psi_v''(y)|\gtrsim \lambda.
		\end{equation*}
	    Then the Van der Corput Lemma (Theorem \ref{Van}) gives 
		\begin{equation*}
			|a_{t,N}(v)|
			\lesssim
			N^2\langle \lambda\rangle^{-1/2}.
		\end{equation*}
		
		To specify the interval needed below, set
		\begin{equation*}
			c_0:=\min_{a\le y\le b}
			\frac{\alpha}{2}y^{\alpha/2-1}>0,
			\qquad
			C_0:=\max_{a\le y\le b}
			\frac{\alpha}{2}y^{\alpha/2-1},
			\qquad
			c:=\frac{c_0}{4\pi},
			\quad
			C:=\frac{C_0}{\pi}.
		\end{equation*}
		In particular,
		\begin{equation}\label{cm:eq:amp-main}
			|a_{t,N}(v)|
			\lesssim N^2\lambda^{-1/2},
			\qquad cL\le v\le CL.
		\end{equation}
		
		For $v\notin[cL,CL]$, we have the following lower bounds:
		\begin{align*}
			v\le0
			&\quad\Longrightarrow\quad
			|\psi_v'(y)|
			\ge c_0\lambda+2\pi N^2|v|,\\
			0<v<cL
			&\quad\Longrightarrow\quad
			|\psi_v'(y)|
			\ge c_0\lambda-2\pi N^2v
			\ge\frac{c_0}{2}\lambda,
			\qquad N^2v\le c\lambda,\\
			v>CL
			&\quad\Longrightarrow\quad
			|\psi_v'(y)|
			\ge2\pi N^2v-C_0\lambda
			\ge\pi N^2v,
			\qquad \lambda\le C^{-1}N^2v.
		\end{align*}
		Consequently,
		\begin{equation*}
			|\psi_v'(y)|\gtrsim\lambda+N^2|v|,
			\qquad \forall v\notin [cL,CL], \; y\in \supp\; \chi.
		\end{equation*}
	Then for every integer $M>1$, integration by parts gives
		\begin{equation*}
			|a_{t,N}(v)|
			\lesssim_{M}
			N^2(\lambda+N^2|v|)^{-M},
			\qquad v\notin[cL,CL].
		\end{equation*}
		Using the trivial bound $H_N(v)\lesssim N^d$, for $M>1$ we obtain
		\begin{equation*}
			\int_{\R\setminus[cL,CL]}
			|a_{t,N}(v)|H_N(v)\,dv
			\lesssim_M
			N^{d+2}\int_\R(\lambda+N^2|v|)^{-M}\,dv \lesssim N^d\lambda^{1-M}.
		\end{equation*}
		Choosing an integer $M>1+d/2$ gives
		\begin{equation}\label{cm:eq:tail-integrated}
			\int_{\R\setminus[cL,CL]}
			|a_{t,N}(v)|H_N(v)\,dv
			\lesssim N^d\lambda^{-d/2},
			\qquad \lambda\ge1.
		\end{equation}
		For $0\le\lambda\le N$, we have $0\le t\le N^{1-\alpha}$.
		The short-time estimate in \eqref{omega}, applied with
		the smooth cutoff $\chi(|\xi|^2)$, gives
		\begin{equation}\label{cm:eq:short-kernel}
			\|K_{\alpha,N}(t)\|_\infty
			\lesssim N^d(1+\lambda)^{-d/2},
			\qquad 0\le\lambda\le N.
		\end{equation}

		For $\lambda\ge N$, \eqref{cm:eq:representation} gives
		\begin{equation*}
			\|K_{\alpha,N}(t)\|_\infty
			\le
			\int_{cL}^{CL}|a_{t,N}(v)|H_N(v)\,dv
			+\int_{\R\setminus[cL,CL]}|a_{t,N}(v)|H_N(v)\,dv.
		\end{equation*}
		If $N\le\lambda\le N^2$, then $N^{-1}\le L\le1$.
		By \eqref{cm:eq:amp-main} and \eqref{cm:eq:local-average},
		\begin{equation*}
			\int_{cL}^{CL}|a_{t,N}(v)|H_N(v)\,dv
			\lesssim N^2\lambda^{-1/2}
			\int_{cL}^{CL}H_N(v)\,dv
			\lesssim N^2\lambda^{-1/2}N^{d-2}L^{d/2-1}=N^2\lambda^{(d-3)/2}.
		\end{equation*}
		If $\lambda\ge N^2$, then $L\ge1$, and
		\eqref{cm:eq:long-average} instead gives
		\begin{equation*}
			\int_{cL}^{CL}|a_{t,N}(v)|H_N(v)\,dv
			\lesssim N^2\lambda^{-1/2}N^{d-2}L
			=N^{d-2}\lambda^{1/2}.
		\end{equation*}
		Together with \eqref{cm:eq:tail-integrated} and
		\eqref{cm:eq:short-kernel}, these imply
		\begin{equation*}
			\|K_{\alpha,N}(t)\|_\infty
			\lesssim
			\begin{cases}
				N^d(1+\lambda)^{-d/2},
				&0\le\lambda\le N,\\[2pt]
				N^d\lambda^{-d/2}+N^2\lambda^{(d-3)/2},
				&N\le\lambda\le N^2,\\[2pt]
				N^d\lambda^{-d/2}+N^{d-2}\lambda^{1/2},
				&\lambda\ge N^2.
			\end{cases}
		\end{equation*}
		For $N\le\lambda\le N^2$, we also have
		\begin{equation*}
			N^2\lambda^{(d-3)/2}
			=N^{d-2}\lambda^{1/2}
			\left(\frac{\lambda}{N^2}\right)^{(d-4)/2}
			\le N^{d-2}\lambda^{1/2},
		\end{equation*}
		since $d\ge5$.
		Substituting $\lambda=tN^\alpha$ therefore proves
		\begin{equation}\label{cm:eq:kernel-small}
			\|K_{\alpha,N}(t)\|_\infty
			\lesssim
			N^d(1+N^\alpha|t|)^{-d/2}
			+N^{2+\alpha(d-3)/2}|t|^{(d-3)/2},
			\qquad tN^\alpha\le N^2,
		\end{equation}
		and
		\begin{equation}\label{cm:eq:kernel-global}
			\|K_{\alpha,N}(t)\|_\infty
			\lesssim
			N^d(1+N^\alpha|t|)^{-d/2}
			+N^{d-2+\alpha/2}|t|^{1/2},
			\qquad t\in\R.
		\end{equation}
	Since $\chi$ is real,
	\begin{equation*}
		K_{\alpha,N}(-t,x)
		=\overline{K_{\alpha,N}(t,-x)}.
	\end{equation*}
	Thus the kernel estimates also hold for negative times.
	They remain valid with $\chi$ replaced by $\chi^2$,
	as required in the $TT^*$ argument below.
	
	\subsubsection*{The two time subdivisions}
	
	Let $I$ be an interval of length $T\le1$, and denote by
	$K_{\alpha,N}^{[\chi^2]}$ the kernel obtained by replacing
	$\chi$ with $\chi^2$.
	The same $TT^*$ argument and Young's inequality used in
	the derivation of \eqref{local short estimate} give
	\begin{equation*}
		\|U_\alpha(t)P_Nf\|_{L^{2p}(I;L^\infty)}
		\lesssim
		\left(
		\int_{-T}^{T}
		\|K_{\alpha,N}^{[\chi^2]}(t,\cdot)\|_{L^\infty}^{p}\,dt
		\right)^{1/(2p)}
		\|f\|_2.
	\end{equation*}

	If $1\le p\le2$, $TN^\alpha\le N^2$, we can apply \eqref{cm:eq:kernel-small} and  obtain
	\begin{equation*}
		\|U_\alpha(t)P_Nf\|_{L^{2p}(I;L^\infty)}
		\lesssim
		\left[
		N^{d/2-\alpha/(2p)}
		+N^{1+\alpha(d-3)/4}
		T^{(d-3)/4+1/(2p)}
		\right]\|f\|_2.
	\end{equation*}
	Applying the time subdivision used in
	\eqref{partition estimate} therefore gives, for every interval
	$J$ of length at most one,
	\begin{equation*}
		\|U_\alpha(t)P_Nf\|_{L^{2p}(J;L^\infty)}
		\lesssim
		\left[
		N^{d/2-\alpha/(2p)}T^{-1/(2p)}
		+N^{1+\alpha(d-3)/4}T^{(d-3)/4}
		\right]\|f\|_2.
	\end{equation*}
	Choose
	\begin{equation*}
		T=\min\{1,N^{\beta_p-\alpha}\},
		\qquad
		\beta_p=\frac{2p(d-2)}{p(d-3)+2}.
	\end{equation*}
	Since $p\le2$,
	\begin{equation*}
		\beta_p\le2,
		\qquad
		TN^\alpha=N^{\min\{\alpha,\beta_p\}}\le N^2.
	\end{equation*}
	Consequently,
	\begin{equation*}
		\|U_\alpha(t)P_Nf\|_{L^{2p}(J;L^\infty)}
		\lesssim
		N^{\max\left\{
			d/2-\alpha/(2p),\,
			d/2-(d-2)/(p(d-3)+2)
			\right\}}\|f\|_2.
	\end{equation*}
	
	For $2\le p<\infty$, \eqref{cm:eq:kernel-global} gives
	\begin{equation*}
		\|U_\alpha(t)P_Nf\|_{L^{2p}(I;L^\infty)}
		\lesssim
		\left[
		N^{d/2-\alpha/(2p)}
		+N^{(d-2)/2+\alpha/4}T^{1/4+1/(2p)}
		\right]\|f\|_2.
	\end{equation*}
	The same time subdivision yields
	\begin{equation*}
		\|U_\alpha(t)P_Nf\|_{L^{2p}(J;L^\infty)}
		\lesssim
		\left[
		N^{d/2-\alpha/(2p)}T^{-1/(2p)}
		+N^{(d-2)/2+\alpha/4}T^{1/4}
		\right]\|f\|_2.
	\end{equation*}
	Choose
	\begin{equation*}
		T=\min\{1,N^{\widetilde\beta_p-\alpha}\},
		\qquad
		\widetilde\beta_p=\frac{4p}{p+2}.
	\end{equation*}
	Hence
	\begin{equation*}
		\|U_\alpha(t)P_Nf\|_{L^{2p}(J;L^\infty)}
		\lesssim
		N^{\max\left\{
			d/2-\alpha/(2p),\,d/2-2/(p+2)
			\right\}}\|f\|_2.
	\end{equation*}
	This concludes \eqref{cm:eq:dyadic}.

	\vspace{7pt}
\subsection{Improved local theory and polynomial growth}

We apply Theorem~\ref{cm:thm:main} to the arguments of
Section~\ref{general nonlinearity growth}.
For an integer $\sigma\ge1$, define
\begin{equation*}
	\alpha_\sigma^{\mathrm{CM}}
	:=
	\begin{cases}
		\displaystyle
		d-\frac{2(d-2)}{d-1},
		&\sigma=1,\\[8pt]
		\displaystyle
		d-\frac4{\sigma+2},
		&\sigma\ge2,
	\end{cases}
	\qquad
	\kappa_{\alpha,\sigma}^{\mathrm{CM}}
	:=\min\{2,\alpha-\alpha_\sigma^{\mathrm{CM}}\}.
\end{equation*}
 By \eqref{cm:eq:exponent}, we also denote
\begin{equation*}
	s_\sigma^{\mathrm{CM}}(\alpha)
	:=\inf_{p>\sigma}b_{\mathrm{CM}}(\alpha,p)
	=\max\left\{
	\frac d2-\frac{\alpha}{2\sigma},
	\frac{\alpha_\sigma^{\mathrm{CM}}}{2}
	\right\}.
\end{equation*}

Then following the proof of Theorem \ref{general local well-posedness} and Theorem \ref{general polynomial growth}, one can directly derive:
\begin{thm}\label{improved local well-posedness}
	Let $d\ge5$, $\alpha>1$, and let $\sigma\ge1$ be an integer.
	For every $\gamma>s_\sigma^{\mathrm{CM}}(\alpha)$ and
	$u_0\in H^\gamma(\T^d)$, equation~\eqref{general FNLS}
	has a unique local solution
	\begin{equation*}
		u\in X_\alpha^{\gamma,b}([-T,T])\subseteq C([-T,T];H^\gamma(\T^d)),
	\end{equation*}
	where $T>0$ depends only on $\|u_0\|_{H^\gamma}$.
	If $u_0\in H^s$ for some $s\ge\gamma$, then $u\in  X_\alpha^{s,b}([-T,T])$ on the same interval.
\end{thm}

\begin{thm}\label{improved polynomial growth theorem}
	Let $d\ge5$, let $\sigma\ge1$ be an integer, and suppose
	\begin{equation*}
		\alpha>\alpha_\sigma^{\mathrm{CM}}.
	\end{equation*}
	For every $u_0\in C^\infty(\T^d)$, the solution of
	\eqref{general FNLS} is global with the polynomial growth
	\begin{equation*}
		\|u(t)\|_{H^{\alpha+n}}
		\lesssim_{u_0}
		\langle t\rangle^{
			\frac{\alpha+2n}
			{\kappa_{\alpha,\sigma}^{\mathrm{CM}}}
			+},
		\qquad \forall t\in\R.
	\end{equation*}
\end{thm}

\begin{rem}\label{improved nonlinear comparison}
	Recalling that $\alpha_\sigma=d^2\sigma/(d\sigma+2)$,  we have
	\begin{equation*}
		\alpha_\sigma-\alpha_\sigma^{\mathrm{CM}}
		=
		\begin{cases}
			\displaystyle
			\frac{2(d-4)}
			{(d+2)\bigl(d-1\bigr)},
			&\sigma=1,\\[8pt]
			\displaystyle
			\frac{2(d\sigma+4-2d)}
			{(d\sigma+2)(\sigma+2)},
			&\sigma\ge2,
		\end{cases}
		>0.
	\end{equation*}
	Thus Theorem~\ref{improved local well-posedness}
	reduces the sufficient regularity threshold in
	Theorem~\ref{general local well-posedness} from
	\begin{equation*}
		\max\left\{
		\frac d2-\frac{\alpha}{2\sigma},
		\frac{\alpha_\sigma}{2}
		\right\}
		\quad\text{to}\quad
		\max\left\{
		\frac d2-\frac{\alpha}{2\sigma},
		\frac{\alpha_\sigma^{\mathrm{CM}}}{2}
		\right\}.
	\end{equation*}
	
	Theorem~\ref{improved polynomial growth theorem}
	reduces the sufficient condition for polynomial growth
	from $\alpha>\alpha_\sigma$ to
	$\alpha>\alpha_\sigma^{\mathrm{CM}}$.
	In the common range $\alpha>\alpha_\sigma$,
	\begin{equation*}
		\kappa_{\alpha,\sigma}^{\mathrm{CM}}
		\ge\kappa_{\alpha,\sigma},
		\qquad
		\frac{\alpha+2n}{\kappa_{\alpha,\sigma}^{\mathrm{CM}}}
		\le
		\frac{\alpha+2n}{\kappa_{\alpha,\sigma}}.
	\end{equation*}
	The latter inequality is strict for
	$\alpha_\sigma<\alpha<\alpha_\sigma+2$.
\end{rem}

		\newpage
		\appendix
		\section{Uniform estimates for oscillatory integrals}
	    We collect several classical results on oscillatory integrals.
	    Some are not used directly in the proofs but are included for reference.
		
		In Appendix A, we will introduce some basic concepts of uniform estimates for oscillatory integrals. For more on this topic, we refer the readers to \cite{11,12}.
		
		Following the notations in \cite{2} or \cite{3}, we first introduce some concepts. Recall that we define the oscillatory integral \(J_{h,\zeta}(\lambda)\) as
		\begin{equation*}
			J_{h,\zeta}(\lambda):=\int_{\R^{d}}e^{i\lambda h(\xi)}\zeta(\xi)\,d\xi.
		\end{equation*}
		\begin{defi}\label{A1}
			For \(r,s>0\), the space \(\mathcal{H}_{r}(s)\) is defined as follows:
			\begin{equation*}
				\mathcal H_r(s):=\left\{P\;\middle|\;
				\begin{gathered}
					P\in\mathcal O(B_{\C^d}(0,r))\cap C(\overline B_{\C^d}(0,r)),\\
					P(\bar w)=\overline{P(w)},\quad |P(w)|<s\quad
					(w\in\overline B_{\C^d}(0,r))
				\end{gathered}\right\}.
			\end{equation*}
		\end{defi}
		\begin{defi}\label{A2}
			Suppose that \(h:\mathbb{R}^{d}\to \mathbb{R}\) is real analytic at \(0\). We write 
			\begin{equation*}
				M(h)\curlyeqprec (\beta,p), \; \beta\le 0,\; p\in \mathbb{N},
			\end{equation*}
			if for sufficiently small \(r>0\), there exist \(s>0\), \(C>0\), and a neighborhood \(\Omega\subseteq B_{\mathbb{R}^{d}}(0,r)\) of the origin such that
			\begin{equation*}
				\left|J_{h+P,\zeta}(\lambda)\right|\le C(1+|\lambda|)^{\beta}\log^{p}(2+|\lambda|)\|\zeta\|_{C^{N}(\Omega)},
			\end{equation*}
			for every $\lambda\in\R$, $\zeta\in C_c^\infty(\Omega)$ and $P\in\mathcal H_r(s)$,
			where \(N=N(h)\in \mathbb{N}\), with 
			\begin{equation*}
				\|\zeta\|_{C^{N}(\Omega)}:=\sup\left\lbrace|\partial^{\gamma}\zeta(\xi)|\;\Big|\;\xi\in\Omega,\; \gamma\in \mathbb{N}^{d},\; |\gamma|\le N\right\rbrace.
			\end{equation*}
			We have the following writing convention:
			\begin{itemize}
				\item We write \(M(h,\xi)\curlyeqprec(\beta,p)\) if 
				\begin{equation*}
					M(\tau_{\xi}h)\curlyeqprec(\beta,p), \quad \text{where } \tau_{\xi}h(y)=h(y+\xi),\; \forall y\in\mathbb{R}^{d};
				\end{equation*}
				\item We write \(M(h_{2})\curlyeqprec M(h_{1})+(\beta_{2},p_{2})\) if 
				\begin{equation*}
					M(h_{1})\curlyeqprec(\beta_{1},p_{1}) \;\; \text{implies} \;\; M(h_{2})\curlyeqprec(\beta_{1}+\beta_{2},p_{1}+p_{2}).
				\end{equation*}
			\end{itemize} 
			
		\end{defi}
		\vspace{10pt}
		Let \(\gamma=(\gamma_{1},\dots, \gamma_{d})\in \mathbb{R}^{d}\), with \(\gamma_{i}>0\) for all \(i=1,\dots,d\). For any \(c>0\), we define the dilation as follows:
		\begin{equation*}
			\delta_{c}^{\gamma}(\xi):=\left(c^{\gamma_{1}}\xi_{1},\dots,c^{\gamma_{d}}\xi_{d}\right), \quad \forall \xi\in \mathbb{R}^{d}.
		\end{equation*}
		\begin{defi}
			A polynomial \(f\) on \(\mathbb{R}^{d}\) is called \(\gamma\)-homogeneous of degree \(\rho\) if 
			\begin{equation*}
				f\circ\delta_{c}^{\gamma}(\xi)=c^{\rho}f(\xi), \quad \forall \xi\in\mathbb{R}^{d},\; c>0.
			\end{equation*}
		\end{defi}
		Let \(\mathcal{E}_{\gamma,d}\) be the set of all \(\gamma\)-homogeneous polynomials on \(\mathbb{R}^{d}\) of degree \(1\). Let \(H_{\gamma,d}\) be the set of all functions that are real-analytic at \(0\) and whose Taylor series have the form \(\sum_{\gamma\cdot\alpha>1}a_{\alpha}\xi^{\alpha}\); i.e., each monomial is \(\gamma\)-homogeneous of degree \(>1\).
		
		We now briefly introduce some useful lemmas, which can be found in \cite{2} or \cite{3}. 
		\begin{lemma}
			If \(h\) is real analytic at \(0\) and \(\nabla h(0)\ne 0\), then
			\begin{equation*}
				M(h)\curlyeqprec (-n,0), \quad \forall n\in\mathbb{N}.		  
			\end{equation*}
		\end{lemma}
		\begin{lemma}\label{E}
			If \(h\in \mathcal{E}_{\gamma,d}\) and \(P\in H_{\gamma,d}\), then 
			\begin{equation*}
				M(h+P)\curlyeqprec M(h).
			\end{equation*}
		\end{lemma}
		Although the above lemmas are important in the context of uniform estimates, in this paper we only require the following theorem from \cite{5} and its direct corollary. For this reason, we present them in detail.
		\begin{thm}\label{stationary}
			Let \( K \subseteq \mathbb{R}^n \) be a compact set, \( X \) an open neighborhood of \( K \), and \( k \) a positive integer. If \( u \in C_0^{2k}(K) \), \( f \in C^{3k+1}(X) \), \( \operatorname{Im} f \geq 0 \) in \( X \), \( \operatorname{Im} f(x_0) = 0 \), \( \nabla f(x_0) = 0 \), \(\det \operatorname{Hess} f(x_0)\ne 0\), and \( \nabla f \neq 0 \) in \( K \setminus \{x_0\} \), then
			\[
			\begin{aligned}
				\bigg|\int_{\R^d} e^{i\omega f(x)}u(x) \, dx 
				- e^{i\omega f(x_0)} \Bigl( \det \bigl( \omega \operatorname{Hess} f(x_0)/2\pi i \bigr) \Bigr)^{-\frac{1}{2}} 
				\sum_{j < k} \omega^{-j} L_j u\bigg| \\
				\leq C \omega^{-k} \sum_{|\alpha| \leq 2k} \sup |D^\alpha u|, \qquad \forall \omega > 0.
			\end{aligned}
			\]
			Here \( C \) is bounded when \( f \) stays in a bounded set in \( C^{3k+1}(X) \) and \( |x - x_0|/|f'(x)| \) has a uniform bound. With
			\[
			g_{x_0}(x) = f(x) - f(x_0) - \frac{1}{2}\langle f''(x_0)(x - x_0), x - x_0 \rangle,
			\]
			which vanishes to third order at \( x_0 \), we have
			\[
			L_j u = \sum_{v - \mu = j} \sum_{2v \geq 3\mu} i^{-j} 2^{-v} 
			\frac{\langle f''(x_0)^{-1} \nabla, \nabla \rangle^v (g_{x_0}^\mu u)(x_0)}{\mu! \, v!},
			\]
			which is a differential operator of order \( 2j \) acting on \( u \) at \( x_0 \). 
		\end{thm}
		
		\begin{coro}\label{Q}
			For \(Q(y)=\sum_{j=1}^{d}c_{j}y_{j}^{2}\) with \(c_{j}=\pm1\) for \(j=1,\dots,d\), we have
			\begin{equation*}
				M(Q)\curlyeqprec\Bigl(-\frac{d}{2},0\Bigr).
			\end{equation*}
		\end{coro}
		\begin{proof}
			The perturbations in $\mathcal H_r(s)$ are required to be real on the
			real domain. Cauchy's inequalities make their first and second
			derivatives uniformly small on a smaller ball when $s$ is sufficiently
			small relative to $r$. The equation
			$2\operatorname{diag}(c_j)y+\nabla P(y)=0$ is then a contraction
			on a suitably chosen real ball and has a unique real solution $y_P$.
			The Hessian stays uniformly invertible there. Shrink the support of
			the amplitude and use a smooth partition near $y_P$ and away from it.
			Near $y_P$ apply Theorem~\ref{stationary} with an integer $k>d/2$;
			all coefficients and remainders are uniform in $P$. Away from $y_P$
			the gradient has a uniform positive lower bound, so integration by
			parts gives arbitrarily fast decay. For bounded $\lambda$ use the
			trivial bound. This proves $M(Q)\curlyeqprec(-d/2,0)$.
			
		\end{proof}

		Finally, we recall the following Van der Corput Lemma (see \cite{10}).
		
		\begin{thm}\label{Van}
			For all \(k\ge 2\), let \(u\) be a \(C^{k}\) function such that \(u^{(k)}(t)\ge 1\). Then for any \(-\infty<a<b<\infty\), \(\lambda>0\), and any function \(\psi\) with an integrable derivative, the following holds:
			\begin{equation*}
				\Big|\int_{a}^{b}e^{i\lambda u(t)}\psi(t)\,dt\Big|\le 12k\lambda^{-1/k}\left[|\psi(b)|+\int_{a}^{b}|\psi'(t)|\,dt\right].
			\end{equation*}
			For \(k=1\), if \(u'(t)\) is monotone on \((a,b)\) and \(u'(t)\ge 1\), we also have 
			\begin{equation*}
				\Big|\int_{a}^{b}e^{i\lambda u(t)}\psi(t)\,dt\Big|\le 3\lambda^{-1}\left[|\psi(b)|+\int_{a}^{b}|\psi'(t)|\,dt\right].
			\end{equation*}
		\end{thm}
		\begin{rem}
			The Van der Corput Lemma directly ensures a uniform and optimal decay estimate in one dimension. In higher dimensions, nondegenerate stationary phase provides the decay used here, while degenerate phases require additional geometric information; see \cite{11,12}.
		\end{rem}
		
		\vspace{10pt}
		\section{Fractional Leibniz rules and product estimates}
		We recall fractional Leibniz rules and related estimates, including
		some additional results for reference.
		
		In Appendix B, we will introduce some useful fractional Leibniz rules for the operator \(|D|^{\alpha}:=(-\Delta)^{\alpha/2}\), which greatly extend the fundamental estimate established by Kenig, Ponce and Vega in \cite{7}:
		\begin{equation*}
			\||D|^s(fg)-f|D|^{s}g-g|D|^{s}f\|_{L^{p}(\R^d)}\lesssim_{s_1,s_2,p_1,p_2}\||D|^{s_1}f\|_{L^{p_1}(\R^d)}\||D|^{s_2}f\|_{L^{p_2}(\R^d)},
		\end{equation*}
		provided 
		\begin{equation}\label{<1}
			\begin{cases}
				0<s=s_1+s_2<1, \quad s_1,s_2\ge0,\\[4pt]
				\frac{1}{p}=\frac{1}{p_1}+\frac{1}{p_2},\quad 1<p, p_1, p_2<\infty.
			\end{cases}
		\end{equation}
		
		The following fractional Leibniz rule from \cite{8} overcomes the restriction ``\(s<1\)" in (\ref{<1}), enabling us to establish estimates for Sobolev norm growth for more general fractional orders \(\alpha\), especially when \(\alpha\ge 1\).
		
		\begin{thm}\label{aa}
			Let \( \ell \in \mathbb{N} \). Let \( p, p_1, p_2 \) satisfy \( 1 < p, p_1, p_2 < \infty \) and \( 1/p = 1/p_1 + 1/p_2 \). Let \( s, s_1, s_2 \) satisfy \( 0 \leq s_1, s_2 \leq \ell \) and \( s = s_1 + s_2 \). Then the following bilinear estimate
			\[
			\Bigl\| |D|^s(fg) - \sum_{k \in \mathbb{Z}} \sum_{m=0}^{\ell-1} A_s^m(0)(P_{\leq k-3}f, P_kg) 
			- \sum_{j \in \mathbb{Z}} \sum_{m=0}^{\ell-1} A_s^m(0)(P_{\leq j-3}g, P_jf) \Bigr\|_{L^p}
			\]
			\[
			\lesssim_{s_1,s_2,p_1,p_2} \||D|^{s_1} f\|_{L^{p_1}} \||D|^{s_2} g\|_{L^{p_2}}
			\]
			holds for all \( f, g \in \mathcal{S}(\R^d) \). 
			
			Here, for \(a_s(\xi,\eta,\theta):=|\eta+\theta\xi|^{s}\), we define 
			\begin{equation*}
				A_{s}^{m}(\theta)(f,g):=\int_{\R^d}\int_{\R^d}e^{ix(\xi+\eta)}\frac{1}{m!}\partial_{\theta}^{m}a_s(\xi,\eta,\theta)\widehat{f}(\xi)\widehat{g}(\eta)\,d\xi d\eta,
			\end{equation*}
			and \(P_{j}\), \(P_{\le j}\) are standard Littlewood-Paley projections onto the frequency regions \(\lbrace|\xi|\sim 2^{j}\rbrace\) and \(\lbrace|\xi|\lesssim 2^{j}\rbrace\), respectively.
		\end{thm}
		
		\begin{coro}\label{fractional Leibniz 1}
			Let \( p, p_1, p_2 \) satisfy \( 1 < p, p_1, p_2 < \infty \) and \( 1/p = 1/p_1 + 1/p_2 \). 
			Let \( s, s_1, s_2 \) satisfy \( 0 \leq s_1, s_2 \leq 1 \), and \( s = s_1 + s_2 \). 
			Then the following bilinear estimate
			\[
			\||D|^s(fg) - f|D|^sg - g|D|^sf\|_{L^p} \lesssim_{s_1,s_2,p_1,p_2,d} \||D|^{s_1} f\|_{L^{p_1}} \||D|^{s_2} g\|_{L^{p_2}}
			\]
			holds for all \( f, g \in \mathcal{S}(\R^d) \).
		\end{coro}
		
		\begin{coro}\label{aaa}
			Let \( p, p_1, p_2 \) satisfy \( 1 < p, p_1, p_2 < \infty \) and \( 1/p = 1/p_1 + 1/p_2 \). 
			Let \( s, s_1, s_2 \) satisfy \( 0 \leq s_1, s_2 \leq 2 \) and \( s = s_1 + s_2 \geq 2 \). 
			Then the following bilinear estimate
			\[
			\begin{aligned}&\||D|^s(fg) - f|D|^sg - g|D|^sf + s|D|^{s-2}(\nabla f \cdot \nabla g)\|_{L^p} 
				\\ &\lesssim_{s_1,s_2,p_1,p_2,d} \||D|^{s_1} f\|_{L^{p_1}} \||D|^{s_2} g\|_{L^{p_2}}
			\end{aligned}
			\]
			holds for all \( f, g \in \mathcal{S}(\R^d) \).
		\end{coro}
		
		In Theorem \ref{aa} above, the fractional Leibniz rule is characterized by the bilinear operator \(A_{s}^{m}(\theta)\) and only deals with non-endpoint cases. The following fractional Leibniz rule includes various endpoint cases and represents fractional derivatives in a more elegant way; see \cite{9}.
		
		\begin{thm}\label{fractional Leibniz 2}
			In the following statement, we adopt the usual multi-index notation: \( \alpha = (\alpha_1, \dots, \alpha_d) \), 
			\( \partial^\alpha = \partial^\alpha_x = \partial^{\alpha_1}_{x_1} \cdots \partial^{\alpha_d}_{x_d} \), 
			\( |\alpha| = \sum_{j=1}^d \alpha_j \), and \( \alpha! = \alpha_1! \cdots \alpha_d! \). 
			The operator \( |D|^{s,\alpha} \) is defined via the Fourier transform as
			\[
			\widehat{|D|^{s,\alpha} g}(\xi) = i^{-|\alpha|}\partial_\xi^{\alpha}(|\xi|^{s})\hat{g}(\xi).
			\]
			
			\textbf{Case 1:} \( 1 < p < \infty \).
			
			Let \( s > 0 \) and \( 1 < p < \infty \). Then for any \( s_1, s_2 \geq 0 \) with \( s_1 + s_2 = s \), and any \( f, g \in \mathcal{S}(\mathbb{R}^d) \), the following hold:
			
			\begin{enumerate}
				\item If \( 1 < p_1, p_2 < \infty \) with \( \frac{1}{p} = \frac{1}{p_1} + \frac{1}{p_2} \), then
				\[
				\Bigl\| |D|^s(fg) - \sum_{|\alpha| \leq s_1} \frac{1}{\alpha!} \partial^\alpha f |D|^{s,\alpha} g 
				- \sum_{|\beta| \leq s_2} \frac{1}{\beta!} \partial^\beta g |D|^{s,\beta} f \Bigr\|_{L^{p}}
				\]
				\[
				\lesssim_{s_1, s_2, p_1, p_2, d} \||D|^{s_1} f\|_{L^{p_1}} \||D|^{s_2} g\|_{L^{p_2}}.
				\]
				
				\item If \( p_1 = p, p_2 = \infty \), then
				\[
				\Bigl\| |D|^s(fg) - \sum_{|\alpha| < s_1} \frac{1}{\alpha!} \partial^\alpha f |D|^{s,\alpha} g 
				- \sum_{|\beta| \leq s_2} \frac{1}{\beta!} \partial^\beta g |D|^{s,\beta} f \Bigr\|_{L^{p}}
				\]
				\[
				\lesssim_{s_1, s_2, p, d} \||D|^{s_1} f\|_{L^{p}} \||D|^{s_2} g\|_{\mathrm{BMO}}.
				\]
				
				\item If \( p_1 = \infty, p_2 = p \), then
				\[
				\Bigl\| |D|^s(fg) - \sum_{|\alpha| \leq s_1} \frac{1}{\alpha!} \partial^\alpha f |D|^{s,\alpha} g 
				- \sum_{|\beta| < s_2} \frac{1}{\beta!} \partial^\beta g |D|^{s,\beta} f \Bigr\|_{L^{p}}
				\]
				\[
				\lesssim_{s_1, s_2, p, d} \||D|^{s_1} f\|_{\mathrm{BMO}} \||D|^{s_2} g\|_{L^{p}}.
				\]
			\end{enumerate}

			\textbf{Case 2:} \( \frac{1}{2} < p \leq 1 \).
			
			If \( \frac{1}{2} < p \leq 1 \), \( s > \frac{d}{p} - d \) or \( s \in 2\mathbb{N} \), 
			then for any \( 1 < p_1, p_2 < \infty \) with \( \frac{1}{p} = \frac{1}{p_1} + \frac{1}{p_2} \), 
			and any \( s_1, s_2 \geq 0 \) with \( s_1 + s_2 = s \),
			\[
			\Bigl\| |D|^s(fg) - \sum_{|\alpha| \leq s_1} \frac{1}{\alpha!} \partial^\alpha f |D|^{s,\alpha} g 
			- \sum_{|\beta| \leq s_2} \frac{1}{\beta!} \partial^\beta g |D|^{s,\beta} f \Bigr\|_{L^{p}}
			\]
			\[
			\lesssim_{s_1, s_2, p_1, p_2, d} \||D|^{s_1} f\|_{L^{p_1}} \||D|^{s_2} g\|_{L^{p_2}}.
			\]
		\end{thm}
		
		\begin{rem}\label{cv}
			Based on the classical Mihlin multiplier theorem (see \cite{10}), we can immediately derive a useful estimate: for any \(f\in \mathcal{S}(\R^d)\),
			\begin{equation*}
				\||D|^{s,\alpha} f\|_{L^{p}(\R^d)}\lesssim \||D|^{s-|\alpha|}f\|_{L^{p}(\R^d)}, \quad 1<p<\infty,
			\end{equation*}
			which is key to ensuring polynomial growth of Sobolev norms when the fractional order \(\alpha\) is large.
		\end{rem}
		
		\begin{rem}
			For smooth periodic inputs, the finite-exponent versions used in this paper transfer to the torus by the following localization argument. Endpoint BMO statements above are recorded in their Euclidean form and are not needed for this transfer. We need fractional Leibniz rules on the torus \(\T^d\). These can be derived directly from the above inequalities on \(\R^d\) as follows: for any \(F\in L^{p_1}(\T^d)\) and \(G\in L^{p_2}(\T^d)\), one can take 
			\begin{equation*}
				\widehat{f_{\varepsilon}}(\xi):=\varepsilon^{d/p_1}\sum_{k\in\Z^d}a_{k}\varphi_{\varepsilon}(\xi-k),\quad 
				\widehat{g_{\varepsilon}}(\xi):=\varepsilon^{d/p_2}\sum_{k\in\Z^d}b_{k}\varphi_{\varepsilon}(\xi-k), \quad \xi\in\R^d,
			\end{equation*}
			where \(\lbrace \varphi_{\varepsilon}\rbrace_{\varepsilon>0}\) is an approximation of the identity, and
			\begin{equation*}
				F(x)=\sum_{k\in\Z^d}a_k e^{i k\cdot x},\quad 
				G(x)=\sum_{k\in\Z^d}b_k e^{i k\cdot x}, \quad x\in \T^d.
			\end{equation*}
			Substituting \(f_{\varepsilon}, g_{\varepsilon}\) into the above fractional Leibniz rules on \(\R^d\) and letting \(\varepsilon\to 0\), first for trigonometric polynomials and then by density, treating the constant Fourier modes separately, we obtain the corresponding finite-exponent estimates on the torus \(\T^d\).
		\end{rem}
		
		\vspace{10pt}
		\section{Product and commutator estimates}
		\label{appendix general commutators}
		In this appendix, we provide detailed proofs of some product estimates appearing in Section \ref{general nonlinearity growth}. Throughout this appendix, functions are smooth on $\T^d$,
	$d\ge2$, and $\sigma\ge1$ is an integer.

	\begin{lemma}\label{appendix tame product}
		Let $s\ge0$, $1<p<\infty$, and $m\ge1$.
		If $u_j\in\{u,\bar u\}$, then
		\begin{equation}\label{appendix tame product estimate}
			\left\|\prod_{j=1}^m u_j\right\|_{W^{s,p}}
			\lesssim
			\|u\|_\infty^{m-1}\|u\|_{W^{s,p}}.
		\end{equation}
	\end{lemma}
	
	\begin{proof}
		The standard product estimate (Kato-Ponce inequality) gives
		\begin{equation*}
			\|fg\|_{W^{s,p}}
			\lesssim
			\|f\|_\infty\|g\|_{W^{s,p}}
			+\|g\|_\infty\|f\|_{W^{s,p}}.
		\end{equation*}
		Iterating this inequality yields \eqref{appendix tame product estimate}.
		
	\end{proof}
	
	\begin{lemma}
		Let $n\ge0$ be an integer and suppose
		\begin{equation*}
			\alpha>\frac{d\sigma }{\sigma+1},
			\qquad
			0<\delta<
			\min\left\{\frac{\alpha}{2},
			(\sigma+1)\alpha-\sigma d\right\}.
		\end{equation*}
		Then
		\begin{equation}\label{general nonlinear product}
			\||u|^{2\sigma}u\|_{H^n}
			\lesssim
			\|u\|_{H^{\alpha/2}}^{2\sigma}
			\|u\|_{H^{\alpha+n-\delta}}.
		\end{equation}
		If $u_j\in\{u,\bar u\}$ and
		$\sum_{j=1}^{4\sigma+1}|\beta_j|\le n$, then
		\begin{equation}\label{general higher nonlinear product}
			\left\|
			\prod_{j=1}^{4\sigma+1}\partial^{\beta_j}u_j
			\right\|_2
			\lesssim
			\|u\|_\infty^{2\sigma}
			\|u\|_{H^{\alpha/2}}^{2\sigma}
			\|u\|_{H^{\alpha+n-\delta}}.
		\end{equation}
	\end{lemma}
	
	\begin{proof}
		We first estimate products whose total derivative order
		is exactly $n$. Products of lower total order follow
		by applying the same argument with $n$ replaced by
		that order and then using Sobolev monotonicity.
		
		If $\alpha>d$, then for $n\ge1$ the
		Gagliardo--Nirenberg inequality gives
		\begin{equation*}
			\|\partial^{\beta_j}u_j\|_{L^{2n/|\beta_j|}}
			\lesssim
			\|u\|_\infty^{1-|\beta_j|/n}
			\|u\|_{H^n}^{|\beta_j|/n},
		\end{equation*}
		with the $L^\infty$ norm when $|\beta_j|=0$.
		Thus, for $\sum_{j=1}^m|\beta_j|=n$,
		\begin{equation*}
			\left\|\prod_{j=1}^m\partial^{\beta_j}u_j\right\|_2
			\lesssim
			\|u\|_\infty^{m-1}\|u\|_{H^n}.
		\end{equation*}
		The same bound holds directly for $n=0$.
		Taking $m=2\sigma+1$ or $m=4\sigma+1$, and using
		\begin{equation*}
			\|u\|_\infty\lesssim\|u\|_{H^{\alpha/2}},
			\qquad
			\|u\|_{H^n}\le\|u\|_{H^{\alpha+n-\delta}},
		\end{equation*}
		proves both assertions in this case.
		
		Suppose now that $\alpha\le d$.
		Since
		\begin{equation*}
			\left(n+\frac{\alpha+d}{2}-\delta\right)
			-\left(n+\frac{d-\alpha}{2}\right)
			-\sigma(d-\alpha)
			=(\sigma+1)\alpha-\sigma d-\delta>0,
		\end{equation*}
		we may choose $0<\theta<1$ such that
		\begin{equation*}
			n+\frac{d-\alpha}{2}
			<
			\theta\left(n+\frac{\alpha+d}{2}-\delta\right),
			\qquad
			\sigma(d-\alpha)
			<
			(1-\theta)
			\left(n+\frac{\alpha+d}{2}-\delta\right).
		\end{equation*}
		Set
		\begin{equation*}
			p_1=\frac2\theta,\qquad
			p_2=\frac{4\sigma}{1-\theta},
			\qquad
			\frac1{p_1}+\frac{2\sigma}{p_2}=\frac12.
		\end{equation*}
		The preceding inequalities are equivalent to
		\begin{equation*}
			n+d\left(\frac12-\frac1{p_1}\right)
			<
			(1-\theta)\frac{\alpha}{2}
			+\theta(\alpha+n-\delta),
		\end{equation*}
		and
		\begin{equation*}
			d\left(\frac12-\frac1{p_2}\right)
			<
			\left(1-\frac{1-\theta}{2\sigma}\right)
			\frac{\alpha}{2}
			+\frac{1-\theta}{2\sigma}(\alpha+n-\delta).
		\end{equation*}
		Applying \eqref{xvv} with
		$s_0=\alpha/2$, $s_1=\alpha+n-\delta$, and parameters
		$\theta$ and $(1-\theta)/(2\sigma)$ gives
		\begin{equation*}
			\|u\|_{W^{n,p_1}}
			\lesssim
			\|u\|_{H^{\alpha/2}}^{1-\theta}
			\|u\|_{H^{\alpha+n-\delta}}^\theta,
		\end{equation*}
		\begin{equation*}
			\|u\|_{L^{p_2}}
			\lesssim
			\|u\|_{H^{\alpha/2}}^{1-(1-\theta)/(2\sigma)}
			\|u\|_{H^{\alpha+n-\delta}}^{(1-\theta)/(2\sigma)}.
		\end{equation*}
		Consequently,
		\begin{equation}\label{qa}
			\|u\|_{W^{n,p_1}}\|u\|_{L^{p_2}}^{2\sigma}
			\lesssim
			\|u\|_{H^{\alpha/2}}^{2\sigma}
			\|u\|_{H^{\alpha+n-\delta}}.
		\end{equation}
		
		For $n\ge1$ and a product of $2\sigma+1$ factors,
		choose
		\begin{equation*}
			\frac1{q_j}
			=
			\frac{|\beta_j|}{np_1}
			+\left(1-\frac{|\beta_j|}{n}\right)\frac1{p_2}.
		\end{equation*}
		Then
		\begin{equation*}
			\sum_{j=1}^{2\sigma+1}\frac1{q_j}
			=\frac1{p_1}+\frac{2\sigma}{p_2}=\frac12.
		\end{equation*}
		The Gagliardo--Nirenberg, H\"older inequalities and (\ref{qa}) give
		\begin{equation*}
			\left\|\prod_{j=1}^{2\sigma+1}
			\partial^{\beta_j}u_j\right\|_2
			\lesssim
			\prod_{j=1}^{2\sigma+1}
			\left(
			\|u\|_{W^{n,p_1}}^{|\beta_j|/n}
			\|u\|_{L^{p_2}}^{1-|\beta_j|/n}
			\right)
			=
			\|u\|_{W^{n,p_1}}\|u\|_{L^{p_2}}^{2\sigma}\lesssim \|u\|_{H^{\alpha/2}}^{2\sigma}
			\|u\|_{H^{\alpha+n-\delta}}.
		\end{equation*}
		This concludes (\ref{general nonlinear product}) by summing the Leibniz expansion.
		
		For $4\sigma+1$ factors, instead choose
		\begin{equation*}
			\frac1{q_j}
			=
			\frac{|\beta_j|}{np_1}
			+\left(1-\frac{|\beta_j|}{n}\right)\frac1{2p_2}.
		\end{equation*}
		Again $\sum_j1/q_j=1/2$, and
		\begin{equation*}
			\left\|\prod_{j=1}^{4\sigma+1}
			\partial^{\beta_j}u_j\right\|_2
			\lesssim
			\prod_{j=1}^{4\sigma+1}
			\left(
			\|u\|_{W^{n,p_1}}^{|\beta_j|/n}
			\|u\|_{L^{2p_2}}^{1-|\beta_j|/n}
			\right)
			=
			\|u\|_{W^{n,p_1}}\|u\|_{L^{2p_2}}^{4\sigma}
		\end{equation*}
		\begin{equation*}
			\lesssim \|u\|_{W^{n,p_1}}\|u\|_\infty^{2\sigma}\|u\|_{L^{p_2}}^{2\sigma}\lesssim\|u\|_\infty^{2\sigma}\|u\|_{H^{\alpha/2}}^{2\sigma}
			\|u\|_{H^{\alpha+n-\delta}},
		\end{equation*}
		which proves (\ref{general higher nonlinear product}).
	
		For $n=0$, use H\"older directly:
		\begin{equation*}
			\||u|^{2\sigma}u\|_2
			\le\|u\|_{L^{p_1}}\|u\|_{L^{p_2}}^{2\sigma},
			\qquad
			\||u|^{4\sigma+1}\|_2
			\le\|u\|_\infty^{2\sigma}\|u\|_{L^{p_2}}^{2\sigma}
			\|u\|_{L^{p_1}}.
		\end{equation*}

	\end{proof}
	
	Recall that, with $\rho=|u|^2$,
	\begin{equation*}
		F_\alpha(u)
		:=u|D|^\alpha\bar u+\bar u|D|^\alpha u
		-|D|^\alpha\rho,
		\qquad
		G_{\alpha,\sigma}(u)
		:=|D|^\alpha(\rho^\sigma)
		-\sigma\rho^{\sigma-1}|D|^\alpha\rho.
	\end{equation*}
	
	\begin{lemma}\label{appendix commutator estimates}
		Let $n\ge0$ be an integer.
		If $\alpha>0$ and
		\begin{equation*}
			0<\delta<r<\min\left\{1,\frac{\alpha}{2}\right\},
		\end{equation*}
		then
		\begin{equation}\label{xcv}
			\|F_\alpha(u)\|_{H^n}
			\lesssim
			\|u\|_{W^{r,\infty}}
			\|u\|_{H^{\alpha+n-\delta}},
		\end{equation}
		and
		\begin{equation}\label{general G spatial estimate}
			\|G_{\alpha,\sigma}(u)\|_{H^n}
			\lesssim
			\|u\|_{W^{r,\infty}}^{2\sigma-1}
			\|u\|_{H^{\alpha+n-\delta}}.
		\end{equation}
		If $\alpha>d$ and
		\begin{equation*}
			0<\eta<\min\left\{1,\frac{\alpha-d}{2}\right\},
		\end{equation*}
		then
		\begin{equation}\label{cbb}
			\|F_\alpha(u)\|_{H^n}
			\lesssim
			\|u\|_{H^{\alpha/2}}
			\|u\|_{H^{\alpha+n-\eta}},
		\end{equation}
		and
		\begin{equation}\label{general G Sobolev estimate}
			\|G_{\alpha,\sigma}(u)\|_{H^n}
			\lesssim
			\|u\|_{H^{\alpha/2}}^{2\sigma-1}
			\|u\|_{H^{\alpha+n-\eta}}.
		\end{equation}
	\end{lemma}
	\begin{proof}
		Write
		\begin{equation*}
			R_\alpha(f,g)
			=
			|D|^\alpha(fg)-f|D|^\alpha g-g|D|^\alpha f.
		\end{equation*}
		Theorem~\ref{fractional Leibniz 2}, applied with
		$s_1=s_2=\alpha/2$, and the multiplier bound for
		$|D|^{\alpha,\mu}$ give
		\begin{equation*}
			\|R_\alpha(f,g)\|_{L^p}
			\lesssim
			\|f\|_{W^{\alpha/2,p_1}}
			\|g\|_{W^{\alpha/2,p_2}}
		\end{equation*}
		\begin{equation*}
			+
			\sum_{1\le|\mu|\le\alpha/2}
			\left(
			\|f\|_{W^{|\mu|,p_1}}
			\|g\|_{W^{\alpha-|\mu|,p_2}}
			+
			\|f\|_{W^{\alpha-|\mu|,p_1}}
			\|g\|_{W^{|\mu|,p_2}}
			\right),
		\end{equation*}
		where $1<p,p_1,p_2<\infty$ and
		$1/p=1/p_1+1/p_2$.
		Moreover, since ordinary derivatives commute with
		$|D|^\alpha$,
		\begin{equation*}
			\partial^\nu R_\alpha(f,g)
			=
			\sum_{\nu_1+\nu_2=\nu}
			\frac{\nu!}{\nu_1!\nu_2!}
			R_\alpha(\partial^{\nu_1}f,\partial^{\nu_2}g).
		\end{equation*}
		
		Let $m=2\sigma$, and let $u_1,\ldots,u_m$ consist
		of $\sigma$ copies of $u$ and $\sigma$ copies of $\bar u$.
		The telescoping identity
		\begin{equation*}
			|D|^\alpha\left(\prod_{\ell=1}^{m}u_\ell\right)
			-
			\sum_{j=1}^{m}
			(|D|^\alpha u_j)\prod_{\ell\ne j}u_\ell
			=
			\sum_{j=2}^{m}
			R_\alpha\left(\prod_{\ell<j}u_\ell,u_j\right)
			\prod_{\ell>j}u_\ell
		\end{equation*}
		and the definition of $F_\alpha$ give
		\begin{equation*}
			G_{\alpha,\sigma}(u)
			=
			\sum_{j=2}^{m}
			R_\alpha\left(\prod_{\ell<j}u_\ell,u_j\right)
			\prod_{\ell>j}u_\ell
			+\sigma\rho^{\sigma-1}F_\alpha(u),
			\qquad
			F_\alpha(u)=-R_\alpha(u,\bar u).
		\end{equation*}
		We first prove \eqref{general G spatial estimate}.
		Fix $|\beta|\le n$. Differentiating the $j$th summand
		in the preceding identity gives a finite linear
		combination of
		\begin{equation*}
			R_\alpha\left(
			\partial^{\beta_0}\prod_{\ell<j}u_\ell,
			\partial^{\beta_j}u_j
			\right)
			\prod_{\ell>j}\partial^{\beta_\ell}u_\ell,
			\qquad
			\beta_0+\beta_j+\cdots+\beta_m=\beta.
		\end{equation*}
	Choose $q_0,q_j,\ldots,q_m\in(2,\infty)$ such that
	\begin{equation*}
		\frac1{q_0}+\sum_{\ell=j}^{m}\frac1{q_\ell}
		=\frac12,
		\qquad
		\frac1p=\frac1{q_0}+\frac1{q_j}.
	\end{equation*}
	Since
	\begin{equation*}
		\frac1p+\sum_{\ell=j+1}^{m}\frac1{q_\ell}
		=\frac12,
	\end{equation*}
	H\"older's inequality gives
	\begin{equation*}
		\left\|
		R_\alpha\left(
		\partial^{\beta_0}\prod_{\ell<j}u_\ell,
		\partial^{\beta_j}u_j
		\right)
		\prod_{\ell>j}\partial^{\beta_\ell}u_\ell
		\right\|_2
		\le
		\left\|
		R_\alpha\left(
		\partial^{\beta_0}\prod_{\ell<j}u_\ell,
		\partial^{\beta_j}u_j
		\right)
		\right\|_{L^p}
		\prod_{\ell>j}
		\|\partial^{\beta_\ell}u_\ell\|_{L^{q_\ell}}.
	\end{equation*}
	Applying the estimate for $R_\alpha$ with
	$p_1=q_0$ and $p_2=q_j$, we obtain
	\begin{equation*}
		\left\|
		R_\alpha\left(
		\partial^{\beta_0}\prod_{\ell<j}u_\ell,
		\partial^{\beta_j}u_j
		\right)
		\right\|_{L^p}
		\lesssim
		\left\|\partial^{\beta_0}\prod_{\ell<j}u_\ell
		\right\|_{W^{\alpha/2,q_0}}
		\|\partial^{\beta_j}u_j\|_{W^{\alpha/2,q_j}}
	\end{equation*}
	\begin{equation*}
		+
		\sum_{1\le|\mu|\le\alpha/2}
		\left[
		\left\|\partial^{\beta_0}\prod_{\ell<j}u_\ell
		\right\|_{W^{|\mu|,q_0}}
		\|\partial^{\beta_j}u_j\|_{W^{\alpha-|\mu|,q_j}}
		+
		\left\|\partial^{\beta_0}\prod_{\ell<j}u_\ell
		\right\|_{W^{\alpha-|\mu|,q_0}}
		\|\partial^{\beta_j}u_j\|_{W^{|\mu|,q_j}}
		\right].
	\end{equation*}
	
	For each value
	\begin{equation*}
		s=\frac{\alpha}{2},
		\qquad\text{or}\qquad
		s=|\mu|,\ \alpha-|\mu|,
	\end{equation*}
	\eqref{appendix tame product estimate} implies
	\begin{equation*}
		\left\|\partial^{\beta_0}\prod_{\ell<j}u_\ell
		\right\|_{W^{s,q_0}}
		\lesssim
		\left\|\prod_{\ell<j}u_\ell
		\right\|_{W^{|\beta_0|+s,q_0}}
		\lesssim
		\|u\|_\infty^{j-2}
		\|u\|_{W^{|\beta_0|+s,q_0}}.
	\end{equation*}
	Therefore each corresponding product of norms can be bounded by 
	\begin{equation}\label{cvb}
		\|u\|_\infty^{j-2}
		\|u\|_{W^{|\beta_0|+s,q_0}}
		\|u\|_{W^{|\beta_j|+\alpha-s,q_j}}
		\prod_{\ell>j}
		\|u\|_{W^{|\beta_\ell|,q_\ell}}.
	\end{equation}
	The total derivative order on the right is
	\begin{equation*}
		(|\beta_0|+s)+(|\beta_j|+\alpha-s)
		+\sum_{\ell>j}|\beta_\ell|
		=\alpha+|\beta|
		\le\alpha+n,
	\end{equation*}
	and every individual order satisfies
	\begin{equation*}
		\max\left\{
		|\beta_0|+s,\,
		|\beta_j|+\alpha-s,\,
		|\beta_\ell|:\ell>j
		\right\}
		\le\alpha+n-\min\left\{1,\frac{\alpha}{2}\right\},
	\end{equation*}
	\begin{equation*}
		(|\beta_0|+s-r)+(|\beta_j|+\alpha-s-r)
		+\sum_{\ell>j}\max\{|\beta_\ell|-r,0\}
		\le\alpha+n-2r
		<\alpha+n-\delta-r.
	\end{equation*}
	We may therefore choose $0<\theta_0,\theta_j,\ldots,\theta_m<1$
	such that
	\begin{equation*}
		\theta_0+\sum_{\ell=j}^{m}\theta_\ell=1,
		\qquad
		|\beta_0|+s
		<(1-\theta_0)r+\theta_0(\alpha+n-\delta),
	\end{equation*}
	\begin{equation*}
		|\beta_j|+\alpha-s
		<(1-\theta_j)r+\theta_j(\alpha+n-\delta),
		\qquad
		|\beta_\ell|
		<(1-\theta_\ell)r+\theta_\ell(\alpha+n-\delta)
		\quad(\ell>j).
	\end{equation*}
	Set $q_i=2/\theta_i$ for $i=0,j,\ldots,m$. Then
	\begin{equation*}
		\frac1{q_0}+\sum_{\ell=j}^{m}\frac1{q_\ell}
		=\frac12.
	\end{equation*}
	Applying \eqref{xv} to each factor gives
	\begin{equation*}
		\|u\|_{W^{|\beta_0|+s,q_0}}
		\lesssim
		\|u\|_{W^{r,\infty}}^{1-\theta_0}
		\|u\|_{H^{\alpha+n-\delta}}^{\theta_0},
		\qquad
		\|u\|_{W^{|\beta_j|+\alpha-s,q_j}}
		\lesssim
		\|u\|_{W^{r,\infty}}^{1-\theta_j}
		\|u\|_{H^{\alpha+n-\delta}}^{\theta_j},
	\end{equation*}
	and
	\begin{equation*}
		\|u\|_{W^{|\beta_\ell|,q_\ell}}
		\lesssim
		\|u\|_{W^{r,\infty}}^{1-\theta_\ell}
		\|u\|_{H^{\alpha+n-\delta}}^{\theta_\ell}
		\quad(\ell>j).
	\end{equation*}
	Consequently, (\ref{cvb}) can be bounded by 
	\begin{equation*}
		\|u\|_{W^{r,\infty}}^{
			j-2+(1-\theta_0)+\sum_{\ell=j}^{m}(1-\theta_\ell)}
		\|u\|_{H^{\alpha+n-\delta}}^{
			\theta_0+\sum_{\ell=j}^{m}\theta_\ell}
		=
		\|u\|_{W^{r,\infty}}^{m-1}
		\|u\|_{H^{\alpha+n-\delta}}.
	\end{equation*}
	Summing over the derivative expansion and $|\beta|\le n$
	proves, for every $2\le j\le m$,
	\begin{equation*}
		\left\|
		R_\alpha\left(\prod_{\ell<j}u_\ell,u_j\right)
		\prod_{\ell>j}u_\ell
		\right\|_{H^n}
		\lesssim
		\|u\|_{W^{r,\infty}}^{m-1}
		\|u\|_{H^{\alpha+n-\delta}}.
	\end{equation*}
	
	The term $\rho^{\sigma-1}F_\alpha(u)$ is also covered
	by this estimate: take $j=2$, $u_1=u$, $u_2=\bar u$,
	and choose the remaining factors so that
	$\prod_{\ell=3}^{m}u_\ell=\rho^{\sigma-1}$. Since
	\begin{equation*}
		\rho^{\sigma-1}F_\alpha(u)
		=-R_\alpha(u,\bar u)\rho^{\sigma-1},
	\end{equation*}
	we obtain
	\begin{equation}\label{xvc}
		\|\rho^{\sigma-1}F_\alpha(u)\|_{H^n}
		=
		\left\|R_\alpha(u,\bar u)
		\prod_{\ell=3}^{m}u_\ell\right\|_{H^n}
		\lesssim
		\|u\|_{W^{r,\infty}}^{m-1}
		\|u\|_{H^{\alpha+n-\delta}}.
	\end{equation}
	Thus, we have proven \eqref{general G spatial estimate}. Take $\sigma=1$ in (\ref{xvc}), we can conclude (\ref{xcv}).
	
	We next prove \eqref{general G Sobolev estimate}.
	Suppose $\alpha>d$ and
	\begin{equation*}
		0<\eta<
		\min\left\{1,\frac{\alpha-d}{2}\right\}.
	\end{equation*}
    One can check that
	\begin{equation*}
		\max\left\{|\beta_0|+s-\frac{\alpha-d}{2},0\right\}
		+
		\max\left\{|\beta_j|+\alpha-s-\frac{\alpha-d}{2},0\right\}
		+
		\sum_{\ell>j}
		\max\left\{|\beta_\ell|-\frac{\alpha-d}{2},0\right\}
	\end{equation*}
    \begin{equation*}
    	<
    	\alpha+n-\eta-\frac{\alpha-d}{2}.
    \end{equation*}
	
	Choose $0<\theta_0,\theta_j,\ldots,\theta_m<1$ such that
	\begin{equation*}
		\theta_0+\sum_{\ell=j}^{m}\theta_\ell=1,
		\qquad
		|\beta_0|+s
		<(1-\theta_0)\frac{\alpha-d}{2}
		+\theta_0(\alpha+n-\eta),
	\end{equation*}
	\begin{equation*}
		|\beta_j|+\alpha-s
		<(1-\theta_j)\frac{\alpha-d}{2}
		+\theta_j(\alpha+n-\eta),
	\end{equation*}
	\begin{equation*}
		|\beta_\ell|
		<(1-\theta_\ell)\frac{\alpha-d}{2}
		+\theta_\ell(\alpha+n-\eta)
		\quad(\ell>j).
	\end{equation*}
	With $q_i=2/\theta_i$, these inequalities become
	\begin{equation*}
		|\beta_0|+s+d\left(\frac12-\frac1{q_0}\right)
		<
		(1-\theta_0)\frac{\alpha}{2}
		+\theta_0(\alpha+n-\eta),
	\end{equation*}
	\begin{equation*}
		|\beta_j|+\alpha-s+d\left(\frac12-\frac1{q_j}\right)
		<
		(1-\theta_j)\frac{\alpha}{2}
		+\theta_j(\alpha+n-\eta),
	\end{equation*}
	\begin{equation*}
		|\beta_\ell|+d\left(\frac12-\frac1{q_\ell}\right)
		<
		(1-\theta_\ell)\frac{\alpha}{2}
		+\theta_\ell(\alpha+n-\eta)
		\quad(\ell>j).
	\end{equation*}
	Applying \eqref{xvv} with $s_0=\alpha/2$ and
	$s_1=\alpha+n-\eta$, and using
	$H^{\alpha/2}\hookrightarrow L^\infty$, we can bound (\ref{cvb}) by 
	\begin{equation*}
		\|u\|_{H^{\alpha/2}}^{
			j-2+(1-\theta_0)+\sum_{\ell=j}^{m}(1-\theta_\ell)}
		\|u\|_{H^{\alpha+n-\eta}}^{
			\theta_0+\sum_{\ell=j}^{m}\theta_\ell}
		=
		\|u\|_{H^{\alpha/2}}^{m-1}
		\|u\|_{H^{\alpha+n-\eta}}.
	\end{equation*}
	Summing the derivative expansion yields
	\begin{equation*}
		\left\|
		R_\alpha\left(\prod_{\ell<j}u_\ell,u_j\right)
		\prod_{\ell>j}u_\ell
		\right\|_{H^n}
		\lesssim
		\|u\|_{H^{\alpha/2}}^{m-1}
		\|u\|_{H^{\alpha+n-\eta}}.
	\end{equation*}
	Taking $j=2$, $u_1=u$, $u_2=\bar u$, and
	$\prod_{\ell=3}^{m}u_\ell=\rho^{\sigma-1}$ also gives
	\begin{equation}\label{cbbb}
		\|\rho^{\sigma-1}F_\alpha(u)\|_{H^n}
		\lesssim
		\|u\|_{H^{\alpha/2}}^{m-1}
		\|u\|_{H^{\alpha+n-\eta}}.
	\end{equation}
	Therefore
	\begin{equation*}
		\|G_{\alpha,\sigma}(u)\|_{H^n}
		\le
		\sum_{j=2}^{m}
		\left\|
		R_\alpha\left(\prod_{\ell<j}u_\ell,u_j\right)
		\prod_{\ell>j}u_\ell
		\right\|_{H^n}
		+\sigma\|\rho^{\sigma-1}F_\alpha(u)\|_{H^n}
		\lesssim
		\|u\|_{H^{\alpha/2}}^{2\sigma-1}
		\|u\|_{H^{\alpha+n-\eta}}.
	\end{equation*}
	This proves \eqref{general G Sobolev estimate}. Take $\sigma=1$ in (\ref{cbbb}), we also conclude (\ref{cbb}).
		
	\end{proof}

		\newpage

		\section*{Acknowledgement}
		The author is grateful to Prof. Alex Cohen for helpful introductions and discussions.
		
		\section*{Conflict of interest statement}
		The author does not have any possible conflict of interest.
		
		\section*{Data availability statement}
		The manuscript has no associated data.
		\bigskip
		\bigskip

		\bibliographystyle{alpha}
		\bibliography{document2601_revised}

\begin{thebibliography}{BGZZK26}

\bibitem[BCH23]{2}
Cheng Bi, Jiawei Cheng, and Bobo Hua.
\newblock {The wave equation on lattices and oscillatory integrals}, 2023.

\bibitem[BGT04]{BGT04}
Nicolas Burq, Patrick G{\'e}rard, and Nikolay Tzvetkov.
\newblock Strichartz inequalities and the nonlinear schr{\"o}dinger equation on
  compact manifolds.
\newblock {\em American Journal of Mathematics}, 126(3):569--605, 2004.

\bibitem[BGT05]{15}
Nicolas Burq, Patrick G{\'e}rard, and Nikolay Tzvetkov.
\newblock {Bilinear eigenfunction estimates and the nonlinear Schr{\"o}dinger
  equation on surfaces}.
\newblock {\em Inventiones mathematicae}, 159(1):187--223, 2005.

\bibitem[BGZZK26]{12}
Saugata Basu, Shaoming Guo, Ruixiang Zhang, and Pavel Zorin-Kranich.
\newblock {A stationary set method for estimating oscillatory integrals}.
\newblock {\em Journal of the European Mathematical Society}, 28(5):2067--2100,
  2026.

\bibitem[Bou93]{13}
J~Bourgain.
\newblock {Fourier Transform Restriction Phenomena for Certain Lattice Subsets
  and Applications to Nonlinear Evolution Equations, Part I: Schr{\"o}dinger
  Equations.}
\newblock {\em Geometric and functional analysis}, 3(2):107--156, 1993.

\bibitem[Bou96]{18}
Jean Bourgain.
\newblock {On the growth in time of higher Sobolev norms of smooth solutions of
  Hamiltonian PDE.}
\newblock {\em IMRN: International Mathematics Research Notices},
  1996(6):277--304, 1996.

\bibitem[CCW99]{11}
Anthony Carbery, Michael Christ, and James Wright.
\newblock {Multidimensional van der Corput and sublevel set estimates}.
\newblock {\em Journal of the American Mathematical Society}, 12(4):981--1015,
  1999.

\bibitem[CHKL15]{30}
Yonggeun Cho, Gyeongha Hwang, Soonsik Kwon, and Sanghyuk Lee.
\newblock {Well-posedness and ill-posedness for the cubic fractional
  Schr{\"o}dinger equations}.
\newblock {\em Discrete and Continuous Dynamical Systems}, 35(7):2863--2880,
  2015.

\bibitem[CKO12]{17}
James Colliander, Soonsik Kwon, and Tadahiro Oh.
\newblock {A remark on normal forms and the ``upside-down'' I-method for
  periodic NLS: growth of higher Sobolev norms}.
\newblock {\em Journal d'Analyse Math{\'e}matique}, 118(1):55--82, 2012.

\bibitem[COX11]{COX11}
Yonggeun Cho, Tohru Ozawa, and Suxia Xia.
\newblock Remarks on some dispersive estimates.
\newblock {\em Communications on Pure and Applied Analysis}, 10(4):1121--1128,
  2011.

\bibitem[CW10]{CW10}
Fabrice Catoire and Wei-Min Wang.
\newblock Bounds on {Sobolev} norms for the defocusing nonlinear
  {Schr{\"o}dinger} equation on general flat tori.
\newblock {\em Communications on Pure and Applied Analysis}, 9(2):483--491,
  2010.

\bibitem[Den19]{Den19}
Yu~Deng.
\newblock On growth of {Sobolev} norms for energy critical {NLS} on irrational
  tori: Small energy case.
\newblock {\em Communications on Pure and Applied Mathematics}, 72(4):801--834,
  2019.

\bibitem[DG19]{DG19}
Yu~Deng and Pierre Germain.
\newblock Growth of solutions to {NLS} on irrational tori.
\newblock {\em International Mathematics Research Notices}, 2019(9):2919--2950,
  2019.

\bibitem[Din17]{D17}
Van~Duong Dinh.
\newblock Strichartz estimates for the fractional schr{\"o}dinger and wave
  equations on compact manifolds without boundary.
\newblock {\em Journal of Differential Equations}, 263(12):8804--8837, 2017.

\bibitem[Din18]{29}
Van~Duong Dinh.
\newblock {Well-posedness of nonlinear fractional Schr{\"o}dinger and wave
  equations in Sobolev spaces}.
\newblock {\em International Journal of Applied Mathematics}, 31(4):483--525,
  2018.

\bibitem[EGT19]{31}
M~Burak Erdo{\u{g}}an, T~Burak G{\"u}rel, and Nikolaos Tzirakis.
\newblock {Smoothing for the fractional Schr{\"o}dinger equation on the torus
  and the real line}.
\newblock {\em Indiana University Mathematics Journal}, 68(2):369--392, 2019.

\bibitem[Eva22]{evans2022partial}
Lawrence~C Evans.
\newblock {\em Partial differential equations}, volume~19.
\newblock American mathematical society, 2022.

\bibitem[FGO18]{8}
Kazumasa Fujiwara, Vladimir Georgiev, and Tohru Ozawa.
\newblock {Higher order fractional Leibniz rule}.
\newblock {\em Journal of Fourier Analysis and Applications}, 24(3):650--665,
  2018.

\bibitem[GG12]{GG12}
Patrick G{\'e}rard and Sandrine Grellier.
\newblock Effective integrable dynamics for a certain nonlinear wave equation.
\newblock {\em Analysis \& PDE}, 5(5):1139--1155, 2012.

\bibitem[GH11]{28}
Boling Guo and Zhaohui Huo.
\newblock {Global well-posedness for the fractional nonlinear Schr{\"o}dinger
  equation}.
\newblock {\em Communications in Partial Differential Equations},
  36(2):247--255, 2011.

\bibitem[GH13]{27}
Boling Guo and Zhaohui Huo.
\newblock {Well-posedness for the nonlinear fractional Schr{\"o}dinger equation
  and inviscid limit behavior of solution for the fractional Ginzburg-Landau
  equation.}
\newblock {\em Fractional Calculus \& Applied Analysis}, 16(1):226--242, 2013.

\bibitem[GHX08]{37}
Boling Guo, Yongqian Han, and Jie Xin.
\newblock {Existence of the global smooth solution to the period boundary value
  problem of fractional nonlinear Schr{\"o}dinger equation}.
\newblock {\em Applied Mathematics and Computation}, 204(1):468--477, 2008.

\bibitem[Gin96]{21}
Jean Ginibre.
\newblock Le probl\`eme de {Cauchy} pour des {EDP} semi-lin\'eaires
  p\'eriodiques en variables d'espace.
\newblock In {\em S\'eminaire Bourbaki : volume 1994/95, expos\'es 790--804},
  number 237 in Ast\'erisque, pages 163--187. Soci\'et\'e math\'ematique de
  France, 1996.

\bibitem[Gra08]{10}
Loukas Grafakos.
\newblock {\em {Classical fourier analysis}}, volume 249 of {\em Graduate Texts
  in Mathematics}.
\newblock Springer, second edition, 2008.

\bibitem[GRM24]{GRM24}
Pierre Germain and Simon~L. Rydin~Myerson.
\newblock Bounds for spectral projectors on generic tori.
\newblock {\em Mathematische Annalen}, 388:731--767, 2024.

\bibitem[H{\"o}r03]{5}
Lars H{\"o}rmander.
\newblock {\em The Analysis of Linear Partial Differential Operators I:
  Distribution Theory and Fourier Analysis}.
\newblock Classics in Mathematics. Springer Berlin, Heidelberg, 2 edition,
  2003.

\bibitem[HS15]{26}
Younghun Hong and Yannick Sire.
\newblock {On Fractional Schr{\"o}dinger Equations in sobolev spaces}.
\newblock {\em Communications on Pure and Applied Analysis}, 14(6):2265--2282,
  2015.

\bibitem[IP14]{36}
Alexandru~D Ionescu and Fabio Pusateri.
\newblock {Nonlinear fractional Schr{\"o}dinger equations in one dimension}.
\newblock {\em Journal of Functional Analysis}, 266(1):139--176, 2014.

\bibitem[Kar86]{3}
Vladimir~Nikolaevich Karpushkin.
\newblock {Uniform estimates of oscillatory integrals with parabolic or
  hyperbolic phase}.
\newblock {\em Journal of Soviet Mathematics}, 33(5):1159--1188, 1986.

\bibitem[KLS13]{KLS}
Kay Kirkpatrick, Enno Lenzmann, and Gigliola Staffilani.
\newblock On the continuum limit for discrete nonlinear {Schr\"odinger}
  equations with long-range lattice interactions.
\newblock {\em Communications in Mathematical Physics}, 317(3):563--591, 2013.

\bibitem[KPV93]{7}
Carlos~E Kenig, Gustavo Ponce, and Luis Vega.
\newblock {Well-posedness and scattering results for the generalized
  Korteweg-de Vries equation via the contraction principle}.
\newblock {\em Communications on Pure and Applied Mathematics}, 46(4):527--620,
  1993.

\bibitem[Las02]{34}
Nick Laskin.
\newblock {Fractional schr{\"o}dinger equation}.
\newblock {\em Physical Review E}, 66(5):056108, 2002.

\bibitem[Li19]{9}
Dong Li.
\newblock {On Kato--Ponce and fractional Leibniz}.
\newblock {\em Revista matem{\'a}tica iberoamericana}, 35(1):23--100, 2019.

\bibitem[MS25]{32}
Alexandre Megretski and Nikolaos Skouloudis.
\newblock {Global Well-posedness for the periodic fractional cubic NLS in 1D},
  2025.

\bibitem[OV16]{20}
Tohru Ozawa and Nicola Visciglia.
\newblock {An improvement on the Br{\'e}zis--Gallou{\"e}t technique for 2D NLS
  and 1D half-wave equation}.
\newblock {\em Annales de l'Institut Henri Poincar{\'e} C, Analyse non
  lin{\'e}aire}, 33:1069--1079, 2016.

\bibitem[PTV17]{PTV}
Fabrice Planchon, Nikolay Tzvetkov, and Nicola Visciglia.
\newblock On the growth of sobolev norms for nls on 2- and 3-dimensional
  manifolds.
\newblock {\em Analysis \& PDE}, 10(5):1123--1147, 2017.

\bibitem[Soh11]{Soh11}
Vedran Sohinger.
\newblock Bounds on the growth of high {Sobolev} norms of solutions to
  nonlinear {Schr{\"o}dinger} equations on {$S^1$}.
\newblock {\em Differential and Integral Equations}, 24(7-8):653--718, 2011.

\bibitem[Soh12]{Soh12}
Vedran Sohinger.
\newblock Bounds on the growth of high {Sobolev} norms of solutions to {2D
  Hartree} equations.
\newblock {\em Discrete and Continuous Dynamical Systems}, 32(10):3733--3771,
  2012.

\bibitem[Sta97]{16}
Gigliola Staffilani.
\newblock {On the growth of high Sobolev norms of solutions for $ KdV $ and
  Schr{\"o}dinger equations}.
\newblock {\em Duke Math. J.}, 90(1):109--142, 1997.

\bibitem[Thi17]{1}
Joseph Thirouin.
\newblock {On the growth of Sobolev norms of solutions of the fractional
  defocusing NLS equation on the circle}.
\newblock {\em Annales de l'Institut Henri Poincar{\'e} C, Analyse Non
  Lin{\'e}aire}, 34:509--531, 2017.

\bibitem[Tri92]{14}
Hans Triebel.
\newblock {\em {Theory of Function Spaces. II}}, volume~84.
\newblock Birkh\"auser Verlag, Basel, 1992.

\bibitem[Tsu89]{22}
Masayoshi Tsutsumi.
\newblock {On smooth solutions to the initial-boundary value problem for the
  nonlinear Schr{\"o}dinger equation in two space dimensions}.
\newblock {\em Nonlinear Analysis: Theory, Methods \& Applications},
  13(9):1051--1056, 1989.

\bibitem[Wan26]{38}
Jiajun Wang.
\newblock Continuum limit of 3d fractional nonlinear schr{\"o}dinger equation.
\newblock {\em Journal of Differential Equations}, 455:113989, 2026.

\bibitem[Zho08]{19}
Sijia Zhong.
\newblock {The growth in time of higher Sobolev norms of solutions to
  Schr{\"o}dinger equations on compact Riemannian manifolds}.
\newblock {\em Journal of Differential Equations}, 245(2):359--376, 2008.

\end{thebibliography}

	\end{document}